%% file: metamour-graphs.tex
\documentclass[DIV15]{scrartcl}
\pdfoutput=1
\usepackage{ucs}
\usepackage[utf8x]{inputenc}
\usepackage[LGR,T1]{fontenc}
\usepackage{lmodern}
\usepackage[greek,english]{babel}
\usepackage[draft=false,kerning=true]{microtype}

\usepackage{amsmath, amsthm, amsfonts}
\usepackage[shortlabels]{enumitem}
\usepackage{fontawesome}
\usepackage{mathtools}
\usepackage{todonotes}
\usepackage{tikz}
\usetikzlibrary{spy}

\usepackage{hyperref}

\numberwithin{equation}{section}
\numberwithin{figure}{section}

\setlist{topsep=0.75ex, itemsep=0ex, after=\vspace*{0.25ex}}

\DeclarePairedDelimiter{\abs}{\lvert}{\rvert}
\newcommand{\CC}[1]{\mathcal{C}(#1)}
\newcommand{\complement}[1]{\overline{#1}}
\DeclarePairedDelimiterX{\edge}[2]{\{}{\}}{#1,#2}
\DeclarePairedDelimiter{\floor}{\lfloor}{\rfloor}
\newcommand{\join}{\mathbin\nabla}
\newcommand{\oeis}[1]{\href{https://oeis.org/#1}{#1}}

\DeclarePairedDelimiter{\set}{\{}{\}}
\DeclarePairedDelimiterX{\setm}[2]{\{}{\}}{#1\,\delimsize\vert\,\mathopen{}#2}

\theoremstyle{plain}
\newtheorem{theorem}{Theorem}[section]

\newtheorem{corollary}[theorem]{Corollary}
\newtheorem{lemma}[theorem]{Lemma}
\newtheorem{observation}[theorem]{Observation}
\newtheorem{proposition}[theorem]{Proposition}
\newtheorem{question}[theorem]{Question}
\newtheorem*{question*}{Question}

\theoremstyle{definition}
\newtheorem{definition}[theorem]{Definition}

\newtheorem{convention}[theorem]{Convention}

\theoremstyle{remark}
\newtheorem{remark}[theorem]{Remark}

\theoremstyle{plain}
\newcounter{beh}
\newcommand{\resetbeh}{\setcounter{beh}{1}}\resetbeh
\newtheorem{behh}{}

\newcommand{\behref}[1]{{\bfseries \ref{#1}}}
\newenvironment{proofbeh}[1]%
{\renewcommand{\qedsymbol}{%
    \tikz[baseline=(sq.base)]{
      \node (sq) at (0,0) {$\square$};
      \node at (0,0) [scale=0.6] {\behref{#1}};}}%
  \proof[Proof of~\behref{#1}]}%
{\endproof}

\theoremstyle{plain}

\newcommand{\footlabeltext}[2]{%
    \addtocounter{footnote}{1}%
    \footnotetext[\thefootnote]{%
        \addtocounter{footnote}{-1}%
        \refstepcounter{footnote}\label{#1}%
        #2%
    }%
}
\newcommand{\footlabel}[2]{%
  \footlabeltext{#1}{#2}%
    $^{\ref{#1}}$%
}
\makeatletter
\let\footref\@undefined
\makeatother
\newcommand{\footref}[1]{%
    $^{\ref{#1}}$%
}
\newcounter{normalfootc}
\renewcommand{\footnote}[1]{%
    \footlabel{footsaferefiwontuse\thenormalfootc}{#1}%
    \addtocounter{normalfootc}{1}%
}

\newlength\alignrightwidth
\newcommand{\alignright}[1]{%
  \settowidth{\alignrightwidth}{#1}%
  \hspace*{-\alignrightwidth}%
  {#1}%
}
  
\tikzset{
  edge/.style={ultra thick},
  nonedge/.style={ultra thick, dotted},
  mmedge/.style={line width=3.2pt, red, draw opacity=0.25},
  cycleedge/.style={line width=6.4pt, green},
  vertex/.style={fill, black, circle, inner sep=0, minimum size=2.5mm},
  svertex/.style={fill, black, circle, inner sep=0, minimum size=1.625mm}
}

\pgfdeclarelayer{back}
\pgfsetlayers{back,main}

\def\cinner{1.3}
\def\cradius{1.5}
\def\couter{1.7}
\def\cstep{15}

\def\cvlabel#1#2{
  \ifnum #1<90
  \node [label={[anchor=west, rotate=#1]#1:{#2}}] at (#1:\cradius) {};
  \else
  \ifnum #1<270
  \node [label={[anchor=east, rotate=#1+180]#1:{#2}}] at (#1:\cradius) {};
  \else
  \node [label={[anchor=west, rotate=#1]#1:{#2}}] at (#1:\cradius) {};
  \fi
  \fi
}

\def\cvertex#1#2#3{
  \node [vertex] (#3) at (#1:\cradius) {};
  \cvlabel{#1}{#2}
}

\def\cvertices#1#2#3#4#5{
  \fill [lightgray, rounded corners=1.625mm]
  (#1-\cstep/2:\cradius)
  -- (#1-\cstep/2:\cinner)
  arc (#1-\cstep/2:#2+\cstep/2:\cinner)
  -- (#2+\cstep/2:\couter)
  arc (#2+\cstep/2:#1-\cstep/2:\couter)
  --
  (#1-\cstep/2:\cradius);
  
  \draw [edge] (#1-\cstep/2:\cradius) arc (#1-\cstep/2:#2+\cstep/2:\cradius);
  
  \foreach \i in {0,...,#5}
  {
    \node [svertex] at (#1+\cstep*\i:\cradius) {};
  }
  \cvlabel{#1}{#3}
  \cvlabel{#2}{#4}
}

\def\vtow#1{\csname w#1\endcsname}

\def\cedgevv#1#2#3#4{
  \draw [#1] (#2) to [out=\vtow{#2}+180, in=\vtow{#3}-180, #4] (#3);
}

\def\cedgevw#1#2#3#4{
  \draw [#1] (#2) to [out=\vtow{#2}+180, in=#3-180, #4] (#3:\cinner);
}

\def\cedgewv#1#2#3#4{
  \draw [#1] (#2:\cinner) to [out=#2+180, in=\vtow{#3}-180, #4] (#3);
}

\def\cedgeww#1#2#3#4{
  \draw [#1] (#2:\cinner) to [out=#2+180, in=#3-180, #4] (#3:\cinner);
}

\newcommand{\specialH}[2]{H_{#1}^{#2}}

\begin{document}
\title{A characterization of graphs \\ with regular distance-$2$ graphs}
\author{Elisabeth Gaar\textsuperscript\faStar%
  \and
  Daniel Krenn\textsuperscript\faLeaf%
}
\date{}
\maketitle

\bgroup
\let\thefootnote\relax
\footnotetext{%
    \alignright{\textsuperscript\faStar}%
    Elisabeth Gaar,
    \href{mailto:elisabeth.gaar@jku.at}{\texttt{elisabeth.gaar@jku.at}},
    	Institute of Production and Logistics Management, Johannes Kepler 
    	University Linz, Austria,
    supported by the Austrian Science Fund (FWF): I\,3199-N31%
}
\footnotetext{%
    \alignright{\textsuperscript\faLeaf}%
    Daniel Krenn,
    \href{mailto:math@danielkrenn.at}{\texttt{math@danielkrenn.at}}
    or \href{mailto:daniel.krenn@plus.ac.at}{\texttt{daniel.krenn@plus.ac.at}},
    Paris Lodron University of Salzburg, Austria, \\
    \url{http://www.danielkrenn.at},
    supported by the Austrian Science Fund (FWF): P\,28466-N35%
}
\egroup

\begin{abstract}\small
  \textbf{Abstract.}
  For non-negative integers~$k$,
  we consider graphs
  in which every vertex has exactly $k$ vertices at distance~$2$,
  i.e., graphs whose distance-$2$ graphs are $k$-regular.
  We call such graphs
  $k$-metamour-regular  
  motivated by the terminology in polyamory.

  While constructing $k$-metamour-regular graphs is relatively easy---we
  provide a generic construction for arbitrary~$k$---finding all
  such graphs is much more challenging. 
  We show that only $k$-metamour-regular graphs with a 
  certain property cannot be built with this construction.
  Moreover, we derive a complete characterization of 
  $k$-metamour-regular graphs 
  for each $k=0$, $k=1$ and $k=2$.
  In particular, a connected graph with~$n$ vertices is $2$-metamour-regular
  if and only if $n\ge5$ and the graph is
  \begin{itemize}
  \item a join of complements of cycles
    (equivalently every vertex has degree~$n-3$),
  \item a cycle, or
  \item one of $17$ exceptional graphs with $n\le8$.
  \end{itemize}
  Moreover, a characterization of graphs in which every vertex has
  at most one metamour is acquired. Each characterization is accompanied
  by an investigation of the corresponding
  counting sequence of unlabeled graphs.
\end{abstract}

\providecommand{\faHashtag}{\#}
\bgroup
\let\thefootnote\relax
\footnotetext{%
  \alignright{\textsuperscript\faHashtag}%
}
\egroup

\def\somepolycule{
  \begin{tikzpicture}
    \begin{scope}[shift={(0,0.5)}]
      \node [vertex] (f) at (0.5,0) {};
      \node [vertex] (b) at (-0.5,0) {};
      \node [vertex] (n) at (0,1) {};
      \node [vertex] (d) at (0,2) {};
      \node [vertex] (c) at (-0.5,3) {};
      \node [vertex] (e) at (0.5,3) {};

      \draw [edge] (b) -- (n) -- (d) -- (c);
      \draw [edge] (d) -- (e);
      \draw [edge] (f) -- (n);

      \draw [mmedge] (b) to [bend left] (d);
      \draw [mmedge] (n) to [bend left] (c);
      \draw [mmedge] (n) to [bend right] (e);
      \draw [mmedge] (c) to [bend left] (e);
      \draw [mmedge] (f) to [bend right] (d);
      \draw [mmedge] (b) to [bend right] (f);
    \end{scope}
    \begin{scope}[shift={(12,1.5)}]
      \node [vertex] (f) at (0,-2) {};
      \node [vertex] (n) at (0.5,-1) {};
      \node [vertex] (d) at (0,0) {};
      \node [vertex] (c) at (-0.5,-1) {};
      \node [vertex] (e) at (0,1) {};
      \node [vertex] (p) at (1,1) {};
      \node [vertex] (r) at (1,0) {};
      \node [vertex] (z) at (0.5,2) {};
      \node [vertex] (l) at (-0.5,2) {};
      
      \draw [edge] (n) -- (d) -- (c);
      \draw [edge] (d) -- (e);
      \draw [edge] (f) -- (n);

      \draw [edge] (d) -- (r) -- (p) -- (e);
      \draw [edge] (e) -- (l) -- (z);

      \draw [mmedge] (n) -- (c);
      \draw [mmedge] (n) to [bend right=15] (e);
      \draw [mmedge] (c) to [bend left=15] (e);
      \draw [mmedge] (f) -- (d);

      \draw [mmedge] (d) -- (p);
      \draw [mmedge] (e) -- (r);
      \draw [mmedge] (n) -- (r);
      \draw [mmedge] (c) -- (r);

      \draw [mmedge] (l) to [bend right=15] (d);
      \draw [mmedge] (l) -- (p);
      \draw [mmedge] (z) -- (e);
    \end{scope}
    \begin{scope}[shift={(6,1)}, yscale=-1]
      \node [vertex] (a) at (-4,0) {};
      \node [vertex] (b) at (-3,0.5) {};
      \node [vertex] (c) at (-3,-1) {};
      \draw [edge] (a) -- (c);
      \draw [edge] (b) -- (c);
      \node [vertex] (d) at (-2,-2.5) {};
      \node [vertex] (e) at (-1.5,-1.5) {};
      \draw [edge] (d) -- (e);
      \node [vertex] (f) at (-2.5,-1.5) {};
      \node [vertex] (g) at (-2,0) {};
      \draw [edge] (f) -- (g);
      \draw [edge] (e) -- (g);
      \draw [edge] (c) -- (g);
      \node [vertex] (h) at (-1.5-0.5/1.5+0.75/1.5,0.5+1.5/1.5+0.25/1.5) {};
      \node [vertex] (i) at (-1.5,0.5) {};
      \draw [edge] (h) -- (i);
      \node [vertex] (j) at (-0.5+0.33/1.5,0.5+1/1.5) {};
      \node [vertex] (k) at (-0.5+0.5/1.5+0.75/1.5,0.5+1.5/1.5-0.25/1.5) {};
      \draw [edge] (j) -- (k);
      \node [vertex] (l) at (-0.5,0.5) {};
      \draw [edge] (k) -- (l);
      \node [vertex] (m) at (-0.5+0.5/1.5-0.75/1.5,0.5+1.5/1.5+0.25/1.5) {};
      \draw [edge] (m) -- (l);
      \node [vertex] (n) at (-1,-1) {};
      \draw [edge] (l) -- (n);
      \draw [edge] (i) -- (n);
      \draw [edge] (g) -- (n);
      \node [vertex] (o) at (0,-1.5) {};
      \node [vertex] (p) at (0,0) {};
      \draw [edge] (o) -- (p);
      \draw [edge] (n) -- (p);
      \node [vertex] (q) at (1,-1) {};
      \draw [edge] (p) -- (q);
      \node [vertex] (r) at (2,-1.5) {};
      \node [vertex] (s) at (2,0) {};
      \draw [edge] (r) -- (s);
      \draw [edge] (q) -- (s);
      \node [vertex] (t) at (3,-1) {};
      \draw [edge] (s) -- (t);
      \node [vertex] (u) at (2.5,1.5) {};
      \draw [edge] (t) -- (u);
      \node [vertex] (v) at (3,0.5) {};
      \draw [edge] (t) -- (v);
      \node [vertex] (w) at (3.5,1.5) {};
      \draw [edge] (t) -- (w);
      \node [vertex] (x) at (4,0) {};
      \draw [edge] (t) -- (x);


      \draw [mmedge] (a) -- (b);
      \draw [mmedge] (b) -- (g);
      \draw [mmedge] (g) -- (a);

      
      \draw [mmedge] (d) -- (g);


      \draw [mmedge] (c) -- (n);
      \draw [mmedge] (n) -- (e);
      \draw [mmedge] (e) -- (f);
      \draw [mmedge] (f) -- (c);
      \draw [mmedge] (c) -- (e);
      \draw [mmedge] (f) -- (n);
      

      \draw [mmedge] (h) -- (n);


      \draw [mmedge] (j) -- (l);

      \draw [mmedge] (m) -- (k);
      \draw [mmedge] (k) -- (n);
      \draw [mmedge] (n) -- (m);

      
      \draw [mmedge] (g) -- (i);
      \draw [mmedge] (i) -- (l);
      \draw [mmedge] (l) -- (p);
      \draw [mmedge] (p) -- (g);
      \draw [mmedge] (g) -- (l);
      \draw [mmedge] (i) -- (p);


      \draw [mmedge] (n) -- (o);
      \draw [mmedge] (o) -- (q);
      \draw [mmedge] (q) -- (n);

      \draw [mmedge] (p) -- (s);


      \draw [mmedge] (q) -- (r);
      \draw [mmedge] (r) -- (t);
      \draw [mmedge] (t) -- (q);
      
      \draw [mmedge] (s) -- (u);
      \draw [mmedge] (u) -- (v);
      \draw [mmedge] (v) -- (w);
      \draw [mmedge] (w) -- (x);
      \draw [mmedge] (x) -- (s);
      \draw [mmedge] (s) -- (v);
      \draw [mmedge] (v) -- (x);
      \draw [mmedge] (x) -- (u);
      \draw [mmedge] (u) -- (w);
      \draw [mmedge] (w) -- (s);

    \end{scope}
    \begin{scope}[shift={(10,3)}]
      \node [vertex] (a) at (180:0.667) {};
      \node [vertex] (m) at (60:0.667) {};
      \node [vertex] (z) at (300:0.667) {};
      \draw [edge] (a) -- (m) -- (z) -- (a);
    \end{scope}
    \begin{scope}[shift={(1.25,-0.5)}]
      \node [vertex] (f) at (30+180:0.5) {};
      \node [vertex] (p) at (30:0.5) {};
      \draw [edge] (f) -- (p);
    \end{scope}
    \begin{scope}[shift={(0,-0.5)}]
      \node [vertex] (m) at (0,0) {};
    \end{scope}
    \begin{scope}[shift={(6,3.5)}]
      \node [vertex] (a) at (-30+180:0.5) {};
      \node [vertex] (k) at (-30:0.5) {};
      \draw [edge] (a) -- (k);
    \end{scope}
    \begin{scope}[shift={(7.5,4)}]
      \node [vertex] (a) at (0,0) {};
    \end{scope}
    \begin{scope}[shift={(1.75+0.1,3)}]
      \node [vertex] (x) at (150:0.667) {};
      \node [vertex] (j) at (270:0.667) {};
      \node [vertex] (y) at (390:0.667) {};
      \draw [edge] (x) -- (j) -- (y);
      \draw [mmedge] (x) -- (y);
    \end{scope}
  \end{tikzpicture}
}

\section{Introduction}

For a given graph, let us construct its distance-$2$ graph as follows:
It has the same vertices as the original graph, and there is an edge
between two vertices if these vertices are at distance~$2$ in the
original graph. Here, distance~$2$ means that such two vertices are
different, not adjacent, and have a common neighbor.
An example is shown in Figure~\ref{fig:polycule}.
\begin{figure}
  \centering
  \somepolycule
  \caption[A graph and its distance-$2$ graph]%
  {A graph and its distance-$2$ graph.
    A line \tikz{
      \node [vertex] (v0) at (0,0) {};
      \node [vertex] (v1) at (1,0) {};
      \draw [edge] (v0) -- (v1); }
    presents an edge in the graph and
    \tikz{
      \node [vertex] (v0) at (0,0) {};
      \node [vertex] (v1) at (1,0) {};
      \draw [mmedge] (v0) -- (v1); }
    an edge in its distance-$2$ graph.}
  \label{fig:polycule}
\end{figure}
Distance-$2$ graphs and properties of vertices at distance~$2$ have
been heavily studied in the literature; see the survey in
Section~\ref{sec:related-literature}.
We pursue the theme of characterizing all graphs whose
distance-$2$ graphs are in a given graph class. We specifically set our
focus on the graph class of regular graphs.

The research is motivated by the relationship concept
\emph{polyamory},%
\footnote{%
  Hyde and DeLamater~\cite{Hyde-DeLamater:2006:underst-human-sexuality}
  describe \emph{polyamory} as
  ``the non-possessive, honest, responsible, and ethical philosophy
  and practice of loving multiple people simultaneously''.
  Other descriptions of polyamory are around; see for example
  Haritaworn, Lin and Klesse~\cite{Haritaworn-Lin-Klesse:2006:poly-logue}
  or Sheff~\cite{Sheff:2016:someone-is-poly}.
  The word polyamory appeared in an article by
  Zell-Ravenheart~\cite{ZellRavenheart:1990:bouquet-lovers} in 1990 and
  is itself a combination of
  the Greek word ``\foreignlanguage{greek}{πολύ}'' (poly) meaning ``many''
  and of the Latin word ``amor'' meaning ``love''.
}
where every person might be in a relationship with any number of
other persons. Naturally, this can be modeled as a graph, where each vertex 
represents a person, and two vertices are adjacent if the corresponding persons 
are in a relationship.
In polyamory, two persons that are both in a relationship with the same 
third person, but are not in a relationship with each other, are called
\emph{metamours}.%
\footnote{For \emph{metamour} see for example
  Hardy and Easten~\cite[p.~219ff, 298]{Hardy-Easton:2017:ethical-slut}
  or Veaux and Rickert~\cite[p.~397ff, 455]{Veaux-Rickert:2014:more-than-two},
  or online at \url{https://www.morethantwo.com/polyglossary.html\#metamour}.
}
We adopt this terminology
and use the term \emph{metamour of a vertex}  
for a vertex at distance~$2$. Correspondingly, we call the
distance-$2$ graph of a graph $G$, the \emph{metamour graph} of $G$.

We have discussed the context and set-up the necessary vocabulary,
so we are now ready to talk about the content and results of this article.
We investigate graphs where each
vertex has the same number of metamours. If this number is~$k$, we
say that the graph is \emph{$k$-metamour-regular}.
Reformulated, a $k$-metamour-regular graph is a graph whose
distance-$2$ graph is $k$-regular.
The leftmost connected component of the graph in Figure~\ref{fig:polycule}
shows an example of
a $2$-metamour-regular (sub)graph with six vertices.
We ask: 
\begin{question*}
Can we find all $k$-metamour-regular graphs and give a
precise description of how they look like?
\end{question*}

Certainly not every graph satisfies this property, but some do.
For example, for $k=2$ 
it is not hard to check that in (connected) graphs with at least
five vertices, where
\begin{itemize}
\item every vertex has two neighbors,
  i.e., cyclic graphs, or
\item every vertex is adjacent to every other vertex but two,
  i.e., complements of cyclic graphs,
\end{itemize}
each vertex indeed has exactly two metamours. Hence, these graphs are 
$2$-metamour-regular.

The second construction above can be generalized, and 
in this article we provide a generic construction that allows to create 
$k$-metamour-regular graphs for any number~$k$.
One of our key results is that the vast majority of these graphs 
can indeed be built by this generic construction. To be more precise,
only $k$-metamour-regular graphs whose metamour graph consists of at most two connected components
cannot necessarily be constructed this way.

This key result lays the foundations for
another main result of this article, namely the 
identification of all $2$-metamour-regular graphs, 
so we answer the question above for $k=2$. 
Our findings are as follows:
Every $2$-metamour-regular graph of any size falls
either into one of the two groups (cyclic or complements of cyclic graphs)
above or 
into the third group of 
\begin{itemize}
\item $17$ exceptional graphs with at least six and at most eight vertices.
\end{itemize}
We provide a systematic and explicit
description of the graphs in the first two groups. 
All $2$-metamour-regular graphs with at most nine vertices---this includes
the $17$ exceptional graphs of the third group---are shown in 
Figures~\ref{fig:graphs-5} to \ref{fig:graphs-9}.
Summarized, we present a complete characterization of all $2$-metamour-regular
graphs.
Note that as a consequence, exceptional graphs exist 
only for up to eight vertices.

In addition to the above result we derive several structural properties of
$k$-metamour-regular graphs for any number~$k$.
We also characterize all graphs in which every 
vertex has no metamour ($k=0$), exactly one metamour ($k=1$), and
at most one metamour.
As a byproduct of every characterization including the one for $k=2$,
we are able to count the number of graphs with these properties. 
Moreover---and
this might be the one sentence take-away message of this article---our findings 
imply 
that besides the graphs that are simple to discover (i.e., can be built by the 
generic construction), only 
a few (if any) small exceptional graphs are
$0$-metamour-regular,
$1$-metamour-regular,
graphs where every vertex has most one metamour and
$2$-metamour-regular.

\subsection{Outline}

We now provide a short overview on the structure
of this paper. The terms discussed so far are formally defined in
Section~\ref{sec:formal-stuff}. Moreover, there we introduce joins of graphs
which are used in the systematic and explicit description of the graphs
in the characterizations of metamour-regular graphs.
The section also includes some basic properties related to those concepts.

Section~\ref{sec:results} is a collection of all results derived in
this article. This is accompanied by plenty of consequences of these
results and discussions.
The proofs of all results are given in
Sections~\ref{sec:proof:foundations} to \ref{sec:proofs:2-mmr}.
We conclude in Section~\ref{sec:conclusions} and provide many
questions, challenges and open problems for future work.

So far not mentioned is the next section. There, literature related to
this article is discussed.

\subsection{Related literature}
\label{sec:related-literature}

In this section we discuss concepts that
have already been examined in literature and that are related
to metamours (i.e., vertices having distance~$2$, or ``neighbors of neighbors'')
and the metamour graphs induced by their relations.
Metamour graphs are called 
\emph{distance-$2$ graphs} in 
Iqbal, Koolen, Park and Rehman~\cite{Iqbal-Koolen-Park-Rehman:2020:dist-reg}, 
and some authors
also call them 
\emph{$2$-distance graphs} (e.g., Azimi and 
Farrokhi~\cite{Azimi-Farrokhi:2014:2-distance-graphs-paths-cycles})
or \emph{$2$nd distance graphs} (e.g., 
Simi\'c~\cite{Simic:1983:graph-equations-distance-graphs}).
This notion also appears in the book by
Brouwer, Cohen and 
Neumaier~\cite[p.~437]{Brouwer-Cohen-Neumaier:1989:distance-regular-graphs}.

The overall question is to characterize all graphs whose distance-$n$ graph
equals some graph of a given graph class.
Simi\'{c}~\cite{Simic:1983:graph-equations-distance-graphs} answers the
question when the distance-$n$ graph of a graph equals the line graph
of this graph.
Characterizing when the distance-$2$ graph is a path or a cycle
is done by
Azimi and Farrokhi~\cite{Azimi-Farrokhi:2014:2-distance-graphs-paths-cycles},
and when it is a union of short paths or and a union of two complete graphs by
Ching and Garces~\cite{Ching-Garces:2019:characterizing-2-distance-graphs}.
Azimi and Farrokhi~\cite{Azimi-Farrokhi:2017:self-2-distance-graphs}
also tackled the question when the distance-$2$ graph of a graph
equals the graph itself. This question is also topic of the online 
discussion~\cite{Zypen:2019:graphs-distance-2}.
Bringing the context to our article, we investigate the above question for
the graph class of regular graphs.

Moreover, vertices in a graph that have distance two, i.e., metamours,
or more generally vertices that have a given 
specific distance, 
are discussed in the existing literature in
many different contexts.
The persons participating in the 
exchange~\cite{Chaudhary:2018:vertices-having-distance}
discuss algorithms for efficiently finding vertices having specific distance
on trees. 
Moreover, the notion of dominating sets is extended to vertices at
specific distances in Zelinka~\cite{Zelinka:1983:k-domatic-numbers}
and in particular to distance two
in Kiser and Haynes~\cite{Kiser-Haynes:2015:distance-2-domatic-numbers-grids}.

Also various kinds of colorings of graphs with respect to vertices of given
distance are studied. Typically, the corresponding chromatic number is
analyzed, for instance  
Bonamy, L\'{e}v\^{e}que and 
Pinlou~\cite{Bonamy-Leveque-Ponlou:2014:2-distance-coloring},
 Borodin, Ivanova and 
Neustroeva~\cite{Borodin-Ivanova-Neustroeva:2004:2-dist-coloring}, and  
Bu and Wang~\cite{Bu-Wang:2019:2-distance-coloring} provide such results
for vertices at distance two. Algorithms for finding such colorings are also investigated. We mention here
Bozda\u{g}, \c{C}ataly\"{u}rek, Gebremedhin, 
Manne, Boman and 
\"{O}zg\"{u}ner~\cite{Bozdag-Catalyurek-Gebremedhin-Manne-Boman-Ozguner:2010:coloring-alg-distance-2}
as an example.
Kamga, Wang, Wang and  
Chen~\cite{Kamga-Wang-Wang-Chen:2018:edge-color-vertices-distance-2}
study
variants of so-called vertex distinguishing colorings, i.e.,
edge colorings where additionally vertices at distance two
have distinct sets of colors.
Their motivation comes from network problems.
The concept is studied more generally but for more specific graph classes 
in Zhang, Li, Chen, Cheng and 
Yao~\cite{Zhang-Li-Chen-Cheng-Yao:2006:x-vertex-dist-edge-coloring}.

Many of the mentioned results also investigate vertices at distance at
most two (compared to exactly two).
This is closely related to the concept of the square of a graph, i.e.,
graphs with the same vertex set as the original graph and two vertices are
adjacent if they have distance at most two in the original graph.
More generally, this concept is known as powers of graphs;
see Bondy and Murty~\cite[p. 82]{Bondy-Murty:2008:graph-theory}.
The overall question to characterize all graphs whose $n$th distance graph
equals some graph of a given graph class
is studied also for powers of graphs instead of distance-$n$ graphs; see
Akiyama, Kaneko and Simi\'{c}~\cite{Akiyama-Kaneko-Simic:1978:line-graphs-power-graphs}.
Colorings are studied for powers of graphs by a motivation coming,
among others, from wireless communication networks or graph drawings.
The corresponding chromatic number is analyzed, for example, in
Kramer and Kramer~\cite{Kramer-Kramer:2008:survey-distance-coloring},
Alon and Mohar~\cite{Alon-Mohar:2002:chromatic-number-powers} and
Molloy and 
Salavatipour~\cite{Molloy-Salavatipour:2005:chromatic-number-square-panar}.
Results on the hamiltonicity of powers of graphs are studied in Bondy and
Murty~\cite[p. 105]{Bondy-Murty:2008:graph-theory}
and Underground~\cite{Underground:1978:hamilton-squares}.

Finally, there are distance-regular graphs. 
Even though the name might suggest that these graphs are closely related 
to metamour graphs, this is not the case: 
A graph is distance-regular if it is regular and  
for any two vertices $v$ and $w$, the number of vertices at distance $j$ from 
$v$ and at distance $k$ from $w$ depend only upon $j$, $k$ and the distance of 
$u$ and $v$.
The book by Brouwer, Cohen and 
Neumaier~\cite{Brouwer-Cohen-Neumaier:1989:distance-regular-graphs}
is a good starting point for this whole research area.
Plenty of publications related to distance-regular graphs are available, in 
particular 
recently Iqbal, Koolen, Park and Rehman 
\cite{Iqbal-Koolen-Park-Rehman:2020:dist-reg} considered distance-regular graphs
whose distance-$2$ graphs are strongly regular.

\section{Definitions, notation \& foundations}
\label{sec:formal-stuff}

This section is devoted to definitions and some simple properties. Moreover, we 
state (graph-theoretic) 
conventions and set up the necessary notation that will be used in this 
article.
The proofs of the properties of this section are postponed to
Section~\ref{sec:proof:foundations}.

\subsection{Graph-theoretic definitions, notation \& conventions}
\label{sec:definitions}

In this graph-theoretic article we use standard graph-theoretic
definitions and notation; see for example 
Diestel~\cite{Diestel:2017:graph-theory:5th}.
We use the convention that all graphs in this article contain at
least one vertex, i.e., we do not talk about the empty graph.
Moreover, we use the following convention
for the sake of convenience.
\begin{convention}
  If two graphs are isomorphic, we will
  call them \emph{equal} and use the equality-sign.
\end{convention}

In many places it is convenient to extend adjacency to subsets of
vertices and subgraphs. We give the following definition that is
used heavily in Section~\ref{sec:results:kmr} and Section~\ref{sec:proofs:k-mmr}.

\begin{definition}
  Let $G$ be a graph, and let $W_1$ and $W_2$ be disjoint subsets of
  the vertices of~$G$.
  \begin{itemize}
  \item We say that $W_1$ is \emph{adjacent} in~$G$ to~$W_2$ if there
    is a vertex~$v_1 \in W_1$ adjacent in~$G$ to some
    vertex~$v_2 \in W_2$.
  \item We say that $W_1$ is \emph{completely adjacent} in~$G$
    to~$W_2$ if every vertex~$v_1 \in W_1$ is adjacent in~$G$ to every
    vertex~$v_2 \in W_2$.
  \end{itemize}
\end{definition}
By identifying a vertex~$v \in V(G)$ with the
subset~$\set{v} \subseteq V(G)$, we may also use (complete) adjacency
between~$v$ and a subset of~$V(G)$. Moreover, for simplicity, whenever
we say that subgraphs of~$G$ are (completely) adjacent, we mean that
the underlying vertex sets are (completely) adjacent.

We explicitly state the negation of adjacent:
We say that $W_1$ is \emph{not adjacent} in~$G$ to~$W_2$ if no
vertex~$v_1 \in W_1$ is adjacent in~$G$ to any vertex~$v_2 \in W_2$.
We will not need the negation of completely adjacent.

We recall the following standard concepts and terminology to
fix their notation.

\begin{itemize}
\item For a set $W \subseteq V(G)$ of vertices of a graph~$G$, the
  \emph{induced subgraph} $G[W]$ is the subgraph of~$G$ with vertices~$W$
  and all edges of~$G$ that are subsets of~$W$, i.e.,
  edges incident only to vertices of~$W$.
\item A set $\mu \subseteq E(G)$ of edges of a graph~$G$ is called
  \emph{matching} if no vertex of~$G$ is incident to more than one edge in~$\mu$.
  In particular, the empty set is a matching.
  The set $\mu$ is called \emph{perfect matching} if every vertex of $G$ is incident
  to exactly one edge in~$\mu$.
\item For a set $\nu \subseteq E(G)$ of edges of a graph~$G$,
  we denote by $G - \nu$ the graph with vertices~$V(G-\nu) = V(G)$ and 
  edges~$E(G-\nu) = E(G) \setminus \nu$.
\item
  The \emph{union} of graphs $G_1$ and~$G_2$, written as $G_1 \cup G_2$,
  is the graph with
  vertex set~$V(G_1) \cup V(G_2)$ and edge set~$E(G_1) \cup E(G_2)$.
\item A \emph{path~$\pi$ of length~$n$} in a graph~$G$ is
  written as $\pi=(v_1,\dots,v_n)$ for vertices~$v_i \in V(G)$ that
  are pairwise distinct.
\item A \emph{cycle~$\gamma$ of length~$n$} or \emph{$n$-cycle~$\gamma$} in a graph~$G$ is
  written as $\gamma=(v_1,\dots,v_n,v_1)$ for vertices~$v_i \in V(G)$ that
  are pairwise distinct.
\item
  For a graph $G$, we write $\CC{G}$ for the set of connected components
  of~$G$. 
Note that each element of~$\CC{G}$ is a subgraph of~$G$.
\item
  We write $\complement{G}$ for the \emph{complement} of a graph~$G$, i.e., the graph
  with the same vertices as~$G$ but with an edge between vertices exactly
  where $G$ has no edge.
\item
We use
the complete graph~$K_n$ for $n\ge1$,
the complete $t$-partite graph~$K_{n_1,\dots,n_t}$ for $n_i\ge1$,
  $i\in\set{1,\dots,t}$,
the path graph~$P_n$ for $n \ge 1$, and
the cycle graph~$C_n$ for $n\ge3$.
\end{itemize}

\begin{remark}\label{rem:frequently-used-complements}
  We will frequently use the complement~$\complement{C_n}$ of the
  cycle graph~$C_n$ for $n\ge3$. Note that
\begin{itemize}
\item $\complement{C_3}$ is the graph with $3$ isolated vertices,
\item $\complement{C_4}$ is the graph with $2$ disjoint single edges, and
\item $\complement{C_5}$ equals $C_5$ (see also Figure~\ref{fig:graphs-5}).
\end{itemize}
\end{remark}

We close this section and continue with definitions and
concepts that are specific for this article.

\subsection{Metamours}

We now formally define the most fundamental concept of this article,
namely metamours.

\begin{definition}
  Let $G$ be a graph.
  \begin{itemize}
  \item A vertex $v$ of the graph~$G$ is a \emph{metamour} of a
    vertex~$w$ of~$G$ if
    the distance of $v$ and $w$ on the graph~$G$ equals~$2$.
  \item The \emph{metamour graph}~$M$ of~$G$ is the graph with the
    same vertex set as~$G$ and an edge between the vertices~$v$ and
    $w$ of~$M$ whenever $v$ is a metamour of~$w$ in~$G$.
  \end{itemize}
\end{definition}

We can slightly reformulate the definition of metamours:
A vertex~$v$ having a different vertex~$w$ as metamour, i.e., having
distance~$2$ on a graph, is equivalent to saying that $v$ and $w$ are not 
adjacent and there is a
vertex~$u$ such that both $v$ and $w$ have an edge incident to this vertex~$u$,
i.e., $u$ is a common neighbor of $v$ and $w$.

Clearly, there is no edge in a graph between two vertices that are metamours of 
each other. This is
reflected in the relation between the metamour graph and the complement of a
graph, and put into writing as the following observation.

\begin{observation}
  \label{observation:mm-subgraph-complement}
  Let $G$ be a graph.
  Then the metamour graph of~$G$ is a subgraph of the
  complement of~$G$.
\end{observation}

The question whether the metamour graph equals the complement will appear
in many statements of this article. The following simply equivalence
is useful.

\begin{proposition}\label{pro:mm-eq-complement:diameter}
  Let $G$ be a connected graph with $n$ vertices.
  Then the following statements are equivalent:
  \begin{enumerate}[(a)]
  \item\label{it:mm-eq-complement:diameter:M-eq-compl}
    The metamour graph of~$G$ equals $\complement{G}$.
  \item\label{it:mm-eq-complement:diameter-2}
    The graph~$G$ has diameter~$2$ or $G=K_n$.
  \end{enumerate}
\end{proposition}

\subsection{Metamour-degree \& metamour-regularity}

Having the concept of metamours, it is natural to investigate the
number of metamours of a vertex. We formally define this ``degree'' and
related concepts below.

\begin{definition}
  \label{def:mm-degree}
  Let $G$ be a graph.
  \begin{itemize}
  \item  The \emph{metamour-degree} of a vertex of~$G$
    is the number of metamours of this vertex.
  \item The \emph{maximum metamour-degree}
    of the graph~$G$
    is the maximum over the metamour-degrees of its vertices.
  \item For $k\ge0$ the graph~$G$ is called
    \emph{$k$-metamour-regular} if every vertex of~$G$ has
    metamour-degree~$k$, i.e., has exactly $k$ metamours.
  \end{itemize}
\end{definition}

We finally have $k$-metamour-regularity at hand and can now start to
relate it to other existing terms. We begin with the following two
observations.

\begin{observation}
    \label{obs:kmr-iff-metamourgraphkr}
    Let $k \ge 0$, and let $G$ be a graph and $M$ its metamour graph.
    Then~$G$ is $k$-metamour-regular if and only if the
    metamour graph~$M$ is $k$-regular.
\end{observation}

The number of vertices with odd degree is even by the
handshaking lemma. Therefore, we get the following observation.

\begin{observation}
  Let $k\ge1$ be odd. Then the number of vertices of a
  $k$-metamour-regular graph is even.
\end{observation}

\begin{proposition}\label{pro:mm-eq-complement:degree}
  Let $G$ be a connected graph with $n$~vertices.
  Then the following statements are equivalent:
  \begin{enumerate}[(a)]
  \item\label{it:mm-eq-complement:degree:M-eq-compl}
    The metamour graph of~$G$ equals $\complement{G}$.
  \item\label{it:mm-eq-complement:degree}
    For every vertex of~$G$, the sum of its degree and its metamour-degree
    equals~$n-1$.
  \end{enumerate}
\end{proposition}
Note that if $k\ge0$ and $G$ is a connected $k$-metamour-regular graph with 
$n$~vertices, then \ref{it:mm-eq-complement:degree} states that
the graph~$G$ is $(n-1-k)$-regular. We use this in
Proposition~\ref{pro:regular-join-and-mm-graphs}.

\subsection{Joins of graphs}
\label{sec:joins}

Given two graphs, we already have defined the union of these graphs in
Section~\ref{sec:definitions}. A join of graphs is a variant of that.
We will introduce this concept now, see also
Harary~\cite[p.~21]{Harary:1969:graph-theory},
and then discuss a couple of simple
properties of joins, also in conjunction with metamour graphs.

\begin{definition}
  Let $G_1$ and $G_2$ be graphs with disjoint vertex sets~$V(G_1)$ and $V(G_2)$.
  The \emph{join of $G_1$ and $G_2$} is the graph denoted by
  $G_1 \join G_2$ with vertices
  $
    V(G_1) \cup V(G_2)
  $
  and edges
  \begin{equation*}
    E(G_1) \cup E(G_2) \cup
    \setm[\big]{\edge{g_1}{g_2}}{\text{$g_1 \in V(G_1)$ and $g_2 \in V(G_2)$}}.
  \end{equation*}
\end{definition}
Some graphs in Figures~\ref{fig:graphs-6} to \ref{fig:graphs-9} are
joins of complements of cycle graphs.
All of the joins of graphs in this paper are ``disjoint joins''. We
use the convention that if the vertex sets~$V(G_1)$ and $V(G_2)$ are not
disjoint, then we make them disjoint before the join.
We point out that the operator~$\join$ is associative
and commutative.

Let us get to know joins of graphs in form of a
supplement to Remark~\ref{rem:frequently-used-complements}. We have
\begin{equation*}
  K_{3,3,\dots,3} = \complement{C_3} \join \complement{C_3} \join \dots
  \join \complement{C_3}
\end{equation*}
for the complete multipartite graph $K_{3,3,\dots,3}$.

There are connections between joins of graphs and metamour graphs that
will appear frequently in the statements and results of this article.
We now present first such relations.

\begin{proposition}\label{pro:join-and-mm-graphs}
  Let $G$ be a connected graph and $M$ its metamour graph.
  Then the following statements are
  equivalent:
  \begin{enumerate}[(a)]
  \item\label{it:join-and-mm-graphs:M-eq-compl}
    The metamour graph~$M$ equals $\complement{G}$ and $\abs{\CC{M}}\ge2$.
  \item\label{it:join-and-mm-graphs:join-mindestens-2}
    The graph $G$ equals
    $G = \complement{M_1} \join \dots \join \complement{M_t}$
    with $\set{M_1,\dots,M_t}=\CC{M}$
    and $t\ge2$.
  \item\label{it:join-and-mm-graphs:join-of-2}
    There are graphs $G_1$ and $G_2$ with $G = G_1 \join G_2$.
  \end{enumerate}
\end{proposition}

\begin{proposition}\label{pro:regular-join-and-mm-graphs}
  Let $k\ge0$ and
  $G$ be a connected $k$-metamour-regular graph with $n$~vertices.
  Let $M$ be the metamour graph of~$G$.
  Then the following statements are
  equivalent:
  \begin{enumerate}[(a)]
  \item\label{it:pro:regular-join-and-mm-graphs:M-eq-compl}
    The metamour graph~$M$ equals $\complement{G}$.
  \item\label{it:pro:regular-join-and-mm-graphs:diameter-2}
    The graph~$G$ has diameter~$2$ or $G=K_n$.
  \item\label{it:pro:regular-join-and-mm-graphs:join}
    The graph $G$ equals
    $G = \complement{M_1} \join \dots \join \complement{M_t}$
    with $\set{M_1,\dots,M_t}=\CC{M}$.
  \item\label{it:pro:regular-join-and-mm-graphs:regular}
    The graph~$G$ is $(n-1-k)$-regular.
  \end{enumerate}
\end{proposition}

Note that we have $G=K_n$ in \ref{it:pro:regular-join-and-mm-graphs:diameter-2} 
if and only if $k=0$.

\section{Characterizations \& properties of metamour-regular graphs}
\label{sec:results}

It is now time to present the main results of this article and their
implications. In this section, we will do this in a formal manner
using the terminology introduced in Section~\ref{sec:formal-stuff}.
This section also includes brief sketches of the proofs of the main results.
The actual and complete proofs of the results follow later, from
Section~\ref{sec:proofs:k-mmr} on to Section~\ref{sec:proofs:2-mmr}.
Proof-wise the results on $k$-metamour-regular graphs for $k\in\set{0,1,2}$
build upon the result for arbitrary~$k\ge0$; this determines the order
of the sections containing the proofs. We will in this section, however, start with $k=0$, followed by
$k=1$ and $k=2$ and only deal with general~$k$ later on.

\subsection{\texorpdfstring{$0$}{0}-metamour-regular graphs}
\label{sec:results:0mr}

As a warm-up, we start with graphs in which no vertex has a
metamour. The following theorem is not very surprising; the only graphs
satisfying this property are complete graphs.

\begin{theorem}
\label{thm:0mr}
 Let $G$ be a connected graph with $n$ vertices. Then~$G$ is 
 $0$-metamour-regular if and only 
 if $G = K_n$.
\end{theorem}

An alternative point of view is that of the metamour graph. The
theorem simply implies that in the case of $0$-metamour-regularity,
the metamour graph is empty and also equals the complement of the
graph itself. The latter property will occur frequently later on which also
motivates its formulation in the following corollary.

\begin{corollary}
    \label{cor:0mmr-complement}
  A connected graph is $0$-metamour-regular
  if and only if
  its complement
  equals its metamour graph
  and this graph has no edges.
\end{corollary}

The characterization provided by Theorem~\ref{thm:0mr} makes it also
easy to count how many different $0$-metamour-regular graphs there are and 
leads to the
following corollary.

\begin{corollary}
  \label{cor:0mmr-counting}
  The number $m_{=0}(n)$ of unlabeled connected
  $0$-metamour-regular graphs with $n$ vertices is
  \begin{equation*}
    m_{=0}(n)=1.
  \end{equation*}
\end{corollary}
The Euler transform,
see Sloane and Plouffe~\cite{Sloane-Plouffe:1995:encyclopedia-int-seq},
of this sequence gives the
numbers~$m_{=0}'(n)$ of unlabeled but not necessarily connected
$0$-metamour-regular graphs with $n$ vertices. The number~$m_{=0}'(n)$
equals the partition function~$p(n)$, i.e., the number of integer partitions%
  \footlabel{footnote:partitions}{%
    An \emph{integer partition} of a positive integer~$n$ is
    a way of representing~$n$ as a sum of positive integers; the order
    of the summands is irrelevant. The \emph{parts} of a partition
    are the summands.}
of~$n$. The corresponding sequence starts with
  \begin{equation*}
    \begin{array}{c|cccccccccccccccc}
      n
      & 1 & 2 & 3 & 4
      & 5 & 6 & 7 & 8 & 9
      & 10 & 11 & 12 & 13 & 14
      & 15
      \\\hline
      m_{=0}'(n)
      & 1 & 2 & 3 & 5
      & 7 & 11 & 15 & 22 & 30
      & 42 & 56 & 77 & 101 & 135
      & 176
    \end{array}
  \end{equation*}
and is \oeis{A000041} in
The On-Line Encyclopedia of Integer Sequences~\cite{OEIS:2020}.

This completes the properties of $0$-metamour-regular
graphs that we bring here. We will, however, see in the following
sections how these properties behave in context of other graph
classes.

\subsection{\texorpdfstring{$1$}{1}-metamour-regular graphs}
\label{sec:results:1mr}

The next easiest case is that of graphs in which every vertex has exactly one
other vertex as metamour. As the metamour relation is symmetric,
these vertices always come in pairs. We write this fact down in the following
proposition.

\begin{proposition}
    \label{pro:1mr_evenNrVertices_perfectMatching}
    Let $G$ be a graph with $n$ vertices.
    Then the following statements hold:
    \begin{enumerate}[(a)]
    \item\label{pro:1mr_evenNrVertices_perfectMatching:it:atMostOneMatching}
    If every vertex of $G$ has at most one metamour, then
    the edges of the metamour graph of $G$ form a matching, i.e.,
    the vertices of~$G$ having exactly one metamour
    come in pairs such that 
    the two vertices of a pair are metamours of each other.
    \item\label{pro:1mr_evenNrVertices_perfectMatching:it:excatlyOnePerfectMatching}
    If $G$ is $1$-metamour-regular,
    then $n$ is even and the edges of the metamour 
    graph of~$G$ form a perfect matching.
  \end{enumerate}
\end{proposition}

By this connection of $1$-metamour-regular graphs to perfect matchings, we can
divine the underlying behavior. This leads to our main result of
this section, a characterization of
$1$-metamour-regular graphs; see the theorem below. It turns out that one 
exceptional case,
namely the graph~$P_4$ (Figure~\ref{fig:P4}),  occurs.

\begin{figure}
    \centering
    \begin{tikzpicture}    
    \begin{scope}[shift={(0, 0)}]
    \node [below] at (1.5,-0.5) {$P_4$};

    \node [vertex] (a4) at (0,0) {};
    \node [vertex] (b4) at (1,0) {};
    \node [vertex] (c4) at (2,0) {};
    \node [vertex] (d4) at (3,0) {};

    \draw [edge] (a4) -- (b4) -- (c4) -- (d4);

    \draw [mmedge] (a4) to [bend left] (c4);
    \draw [mmedge] (b4) to [bend right] (d4);
    \end{scope}
    \end{tikzpicture}
    \caption{The path graph~$P_4$ where each vertex has exactly $1$ metamour}
    \label{fig:P4}
\end{figure}

\begin{theorem}\label{thm:1mr}
Let $G$ be a connected graph with $n$ vertices. 
Then~$G$ is $1$-metamour-regular if and only
if $n \ge 4$ is even and either
\begin{enumerate}[(a)]
\item $G = P_4$ or
\item $G = K_n - \mu$ for some perfect matching $\mu$ of $K_n$
\end{enumerate}
holds.
\end{theorem}
When excluding $G=P_4$, then the graphs in the theorem are
exactly the cocktail party graphs~\cite{Weisstein:2020:cocktail-party-graph}.

Let us again view this from the angle of metamour graphs. As soon as
we exclude the exceptional case~$P_4$, the metamour graph and the complement
of a $1$-metamour-regular graph coincide; see the following corollary.

\begin{corollary}
  \label{cor:1mmr-complement}
  A connected graph with $n\ge5$ vertices is $1$-metamour-regular
  if and only if
  its complement
  equals its metamour graph
  and this graph is $1$-regular.
\end{corollary}
Note that a $1$-regular graph with $n$ vertices is a graph induced by 
a perfect matching of~$K_n$.
In view of Proposition~\ref{pro:regular-join-and-mm-graphs}, we can extend the two
equivalent statements.

As we have a characterization of $1$-metamour-regular graphs (provided by 
Theorem~\ref{thm:1mr})
available, we can determine the number of different graphs in this
class. Clearly, this is strongly related to the existence of a perfect
matching; details are provided below and also in
Section~\ref{sec:proofs:1-mmr}, where proofs are given.

\begin{corollary}
  \label{cor:1mmr-counting}
  The sequence of numbers $m_{=1}(n)$ of unlabeled connected
  $1$-metamour-regular graphs with $n$ vertices starts with
  \begin{equation*}
    \begin{array}{c|cccccccccc}
      n
      & 1 & 2 & 3 & 4
      & 5 & 6 & 7 & 8 & 9
      & 10
      \\\hline
      m_{=1}(n)
      & 0 & 0 & 0 & 2
      & 0 & 1 & 0 & 1 & 0
      & 1
    \end{array}
  \end{equation*}
  and we have
  \begin{equation*}
    m_{=1}(n) =
    \begin{cases}
      0 & \text{for odd $n$,} \\
      1 & \text{for even $n\ge6$.}
    \end{cases}
  \end{equation*}
\end{corollary}
The Euler transform, see~\cite{Sloane-Plouffe:1995:encyclopedia-int-seq},
of the sequence of numbers~$m_{=1}(2n)$
gives the
numbers~$m_{=1}'(2n)$ of unlabeled but not necessarily connected
$1$-metamour-regular graphs with $2n$ vertices. The sequence
of these numbers starts with
  \begin{equation*}
    \begin{array}{c|cccccccccccccccccccc}
      n
      & 1 & 2 & 3 & 4
      & 5 & 6 & 7 & 8 & 9
      & 10 & 11 & 12 & 13 & 14
      & 15 & 16 & 17 & 18
      \\\hline
      m_{=1}'(2n)
      & 0 & 2 & 1 & 4
      & 3 & 8 & 7 & 15 & 15
      & 27 & 29 & 48 & 53 & 82
      & 94 & 137 & 160 & 225
    \end{array}\;.
  \end{equation*}
This sequence also counts how often a part~$2$ appears
in all integer partitions%
    \footnote{For integer partitions,
    see footnote~\footref{footnote:partitions}
    on page~\pageref{footnote:partitions}.}
of $n+2$ with parts at least~$2$. The underlying bijection is formulated
as the following corollary.

\begin{corollary}
  \label{cor:1mmr-bijection:nnc}
  Let $n \ge 0$. Then the set of unlabeled
  $1$-metamour-regular graphs with $2n$ vertices is in bijective correspondence
  to the set of partitions of~$n+2$ with smallest part equal to~$2$
  and one part~$2$ of each partition marked.
\end{corollary}

\subsection{Graphs with maximum metamour-degree~\texorpdfstring{$1$}{1}}
\label{sec:results:at-most-1m}

Let us now slightly relax the metamour-regularity condition and
consider graphs in which every vertex of $G$ has at most one
metamour. In view of
Proposition~\ref{pro:1mr_evenNrVertices_perfectMatching}, matchings
play an important role again. Formally, the following theorem holds.

\begin{theorem}
    \label{thm:max1mm}
    Let $G$ be a connected graph with $n$ vertices. 
    Then the maximum metamour-degree of~$G$ is $1$
    if and only if either
    \begin{enumerate}[(a)]
    \item $G \in \set{K_1,K_2,P_4}$ or
    \item $n \ge 3$ and $G = K_n - \mu$ for some matching $\mu$ of $K_n$
  \end{enumerate}
  holds.
\end{theorem}

As in the sections above, the obtained characterization leads to
the following equivalent statements with respect to
metamour graph and complement.

\begin{corollary}
  \label{cor:max1mm-complement}
  A connected graph with $n\ge5$ vertices has the property that
  every vertex has at most one metamour
  if and only if
  its complement
  equals its metamour graph
  and this graph has maximum degree~$1$.
\end{corollary}
Note that graphs with maximum degree~$1$ and $n$ vertices are graphs induced 
by a (possibly empty) matching of $K_n$.
In view of Proposition~\ref{pro:mm-eq-complement:diameter}, we can extend the 
two
equivalent statements by a third saying that $G$ has diameter~$2$ or $G=K_n$.

Counting the graphs with maximum metamour-degree~$1$ relies on the
number of matchings; see the relevant proofs in
Section~\ref{sec:proofs:max-1mm} for details. We obtain the following
corollary.

\begin{corollary}
  \label{cor:max1mm-counting}
  The sequence of numbers $m_{\le1}(n)$ of unlabeled connected
  graphs with $n$ vertices where every vertex has at most one metamour
  starts with
  \begin{equation*}
    \begin{array}{c|cccccccccc}
      n
      & 1 & 2 & 3 & 4
      & 5 & 6 & 7 & 8 & 9
      & 10
      \\\hline
      m_{\le1}(n)
      & 1 & 1 & 2 & 4
      & 3 & 4 & 4 & 5 & 5
      & 6
    \end{array}
  \end{equation*}
  and we have
  \begin{equation*}
    m_{\le1}(n) = \floor[\big]{\tfrac{n}{2}}+1
  \end{equation*}
  for $n\ge5$.
\end{corollary}
The Euler transform, see~\cite{Sloane-Plouffe:1995:encyclopedia-int-seq},
gives the sequence of
numbers~$m_{\le1}'(n)$ of unlabeled but not necessarily connected
graphs with maximum metamour-degree~$1$ and $n$ vertices
which starts with
  \begin{equation*}
    \begin{array}{c|cccccccccccccc}
      n
      & 1 & 2 & 3 & 4
      & 5 & 6 & 7 & 8 & 9
      & 10 & 11 & 12 & 13
      \\\hline
      m_{\le1}'(n)
      & 1 & 2 & 4 & 9
      & 14 & 26 & 43 & 76 & 122
      & 203 & 322 & 523 & 814
    \end{array}\;.
  \end{equation*}

\subsection{\texorpdfstring{$2$}{2}-metamour-regular graphs}
\label{sec:results:2mr}

We now come to the most interesting graphs in this article, namely
graphs in which every vertex has exactly two other vertices as
metamours. Also in this case a characterization of the class of
graphs is possible.
We first consider Observation~\ref{obs:kmr-iff-metamourgraphkr} in
view of~$2$-metamour-regularity. 
As a graph is $2$-regular if and only if it is a union of cycles, 
the following observation is easy to verify.

\begin{observation}
    \label{obs:2mr_metamour-graph-consists-of-cycles}
    Let $G$ be a graph and $M$ its metamour graph. Then
    $G$ is $2$-metamour-regular if and only if
    every connected component of the metamour graph~$M$ is a cycle.
\end{observation}

We are ready to fully state the mentioned characterization formally as the 
theorem below. We comment this result and discuss
implications afterwards.
Note that Theorem~\ref{thm:2mriff} generalizes the main result of
Azimi and Farrokhi~\cite[Theorem~2.3]{Azimi-Farrokhi:2014:2-distance-graphs-paths-cycles}
which only deals with metamour graphs being connected.

\begin{theorem}\label{thm:2mriff}
    Let $G$ be a connected graph with $n$ vertices.
    Then $G$ is $2$-metamour-regular if and only if $n\ge 5$ and one of 
    \begin{enumerate}[(a)]
        \item\label{it:thm-2mmr:general}
        $G = \complement{C_{n_1}} \join \dots \join \complement{C_{n_t}}$ 
        with
        $n=n_1+\dots+n_t$ for some $t\ge1$
        and $n_i \ge 3$ for all $i\in\set{1,\dots,t}$,  
        \item\label{it:thm-2mmr:Cn}
        $G = C_n$, or      
        \item\label{it:thm-2mmr:exceptional}
        $\begin{aligned}[t]
          G \in \bigl\{\! &%
            \specialH{4,4}{a},
            \specialH{4,4}{b},
            \specialH{4,4}{c},
            \\ &%
            \specialH{7}{a},
            \specialH{7}{b},
            \specialH{4,3}{a},
            \specialH{4,3}{b},
            \specialH{4,3}{c},
            \specialH{4,3}{d},
            \\ &%
            \specialH{6}{a},
            \specialH{6}{b},
            \specialH{6}{c},
            \specialH{3,3}{a}, 
            \specialH{3,3}{b},
            \specialH{3,3}{c}, 
            \specialH{3,3}{d},
            \specialH{3,3}{e}            
        \bigr\}          
        \end{aligned}$ \\[0.5ex]
        with graphs defined by
        Figures~\ref{fig:graphs-6}, \ref{fig:graphs-7} and \ref{fig:graphs-8}
    \end{enumerate}
    holds.
\end{theorem}

\input{figures-2mmr}

A representation of every $2$-metamour-regular graph with at most $9$
vertices can be found in Figures~\ref{fig:graphs-5},
\ref{fig:graphs-6}, \ref{fig:graphs-7}, \ref{fig:graphs-8} and
\ref{fig:graphs-9}.
For $10$ vertices, all $2$-metamour-regular graphs---there are $6$ of them---are
$C_{10}$,
$\complement{C_{10}}$,
$\complement{C_7} \join \complement{C_3}$,
$\complement{C_6} \join \complement{C_4}$,
$\complement{C_5} \join \complement{C_5}$,
$\complement{C_4} \join \complement{C_3} \join \complement{C_3}$.

For rounding out Theorem~\ref{thm:2mriff}, we have collected a couple
of remarks and bring them now.

\begin{remark}
  \label{rem:thm-2mmr}
  \mbox{}
  \begin{enumerate}
  \item\label{it:thm-2mmr:C5-twice}
    The smallest possible $2$-metamour-regular graph has $5$
    vertices, and there is exactly one with five vertices, namely~$C_5$;
    see Figure~\ref{fig:graphs-5}. This graph is covered by
    Theorem~\ref{thm:2mriff}\ref{it:thm-2mmr:general} as well
    as~\ref{it:thm-2mmr:Cn} because $C_5 = \complement{C_5}$.
  \item\label{it:thm-2mmr:equivlent-join}
    Theorem~\ref{thm:2mriff}\ref{it:thm-2mmr:general}
    can be replaced by any other equivalent statement of
    Proposition~\ref{pro:regular-join-and-mm-graphs}.
  \item\label{it:thm-2mmr:t1}
    For $t=1$, Theorem~\ref{thm:2mriff}\ref{it:thm-2mmr:general}
    condenses to $G=\complement{C_n}$.
    Implicitly we get $n=n_1 \ge 5$.
  \item\label{it:thm-2mmr:Ci-is-Mi}
    The graphs~$C_{n_1}$, \dots, $C_{n_t}$ of
    Theorem~\ref{thm:2mriff}\ref{it:thm-2mmr:general} satisfy
    \begin{equation*}
      C_{n_i} = M_i
    \end{equation*}
    with $\set{M_1,\dots,M_t} = \CC{M}$. This means that the decomposition
    of the
    graph~$G = \complement{C_{n_1}} \join \dots \join
    \complement{C_{n_t}}$ reveals the metamour graph of~$G$ and vice
    versa.
  \item\label{it:thm-2mmr:mm-eq-complement}
    For Theorem~\ref{thm:2mriff}\ref{it:thm-2mmr:general}
    as well as for \ref{it:thm-2mmr:Cn} with $n=5$, every graph
    satisfies that its complement equals its metamour graph.
    For all other cases, this is
    not the case. A full formulation of this fact is stated as
    Corollary~\ref{cor:2mm:mm-ne-complement}.
  \end{enumerate}
\end{remark}

The proof of Theorem~\ref{thm:2mriff} is quite extensive and
we refer to Section~\ref{sec:proofs:2-mmr} for the complete proof;
at this point, we only sketch it.

\begin{proof}[Sketch of Proof of Theorem~\ref{thm:2mriff}]
  Let $G$ be $2$-metamour-regular. We apply
  Theorem~\ref{thm:kmmr:metamour-graph-not-connected}
  (to be presented in Section~\ref{sec:results:kmr}) with $k=2$,
  which leads us to one of three cases.
  The generic case is already covered by
  Theorem~\ref{thm:kmmr:metamour-graph-not-connected}
  combined with Observation~\ref{obs:2mr_metamour-graph-consists-of-cycles}.

  If the metamour graph of~$G$ is connected, then we first rule out
  that $G$ is a tree. If not, then the graph contains a cycle, and
  depending on whether~$G$ contains a cycle of length~$n$ or not, we
  get $G = \specialH{7}{b}$ or
  $G \in \set{\specialH{6}{a}, \specialH{6}{b}, \specialH{7}{a}}$,
  respectively. (This parts of the proof are formulated as
  Proposition~\ref{pro:2mmr:Cn-or-complementCn},
  Lemma~\ref{lem:2mmr:deg-two-options} and
  Lemma~\ref{lem:2mmr:contain-Cn}.)
  We prove these parts by studying the longest cycle
  in the graph~$G$ and step-by-step obtaining information between which
  vertices edges, non-edges and metamour relations need to be.
  Knowing enough,  $2$-metamour-regularity forces the graph to be bounded
  in the number of vertices. As this number is quite small, further case
  distinctions lead to the desired graphs.
  Furthermore, a separate investigation is needed when
  every vertex has degree $2$ or $n-3$ or a mixture of these to degrees;
  here $n$ is the number of vertices of~$G$. This leads to $G=C_n$ and
  $n$ odd, $G=\complement{C_n}$ and
  $G \in \set{C_5, \specialH{6}{c}}$, respectively. (These parts of
  the proof are formulated as Lemma~\ref{lemma:2mmr:all-deg-2},
  Lemma~\ref{lemma:2mmr:all-deg-nMinus3} and
  Lemma~\ref{lem:2mmr:deg-2-and-nm3}.)

  If the metamour graph of~$G$ is not connected, it consists
  of exactly two connected components
  (by Theorem~\ref{thm:kmmr:metamour-graph-not-connected}), and
  the graph is split respecting the connected components
  of the metamour graph of~$G$.
  Due to $2$-metamour-regularity,
  Theorem~\ref{thm:kmmr:metamour-graph-not-connected} guarantees that
  the number of vertices of each piece is a most~$2$. Studying each
  configuration separately lead to the graphs
  $G = C_n$ and $n$ even, or
  $G \in \set{
    \specialH{4,4}{a},
    \specialH{4,4}{b},
    \specialH{4,4}{c},
    \specialH{4,3}{a},
    \specialH{4,3}{b},
    \specialH{4,3}{c},
    \specialH{4,3}{d},
    \specialH{3,3}{a},
    \specialH{3,3}{b},
    \specialH{3,3}{c},
    \specialH{3,3}{d},
    \specialH{3,3}{e}
  }$. (This part of the proof is formulated as
  Proposition~\ref{pro:2mmr:metamour-graph-not-cn}.)
\end{proof}

The characterization provided by 
Theorem~\ref{thm:2mriff} has many implications. We start with
the following easy corollaries.

\begin{corollary}
    \label{cor:2mm:collect-non-standard-cases}
    Let $G$ be a connected graph with $n\ge9$ vertices. Then $G$ is
    $2$-metamour-regular if and only if $G$ is either $C_n$ or
    $\complement{C_{n_1}} \join \dots \join \complement{C_{n_t}}$ with
    $n=n_1+\dots+n_t$ for some~$t \ge 1$ and
    $n_i \ge 3$ for all $i\in\set{1,\dots,t}$.
\end{corollary}

\begin{corollary}
    \label{cor:2mm:2mm-iff-regular}
    Let $G$ be a connected graph with $n\ge9$ vertices. Then $G$ is
    $2$-metamour-regular if and only if $G$ is either $2$-regular or
    $(n-3)$-regular.
\end{corollary}

As before, we consider the relation of metamour graphs and
complement more closely; see the following corollary. Again, we 
feel the spirit of Proposition~\ref{pro:regular-join-and-mm-graphs}.

\begin{corollary}
  \label{cor:2mm:mm-ne-complement}
  Let $G$ be a connected $2$-metamour-regular graph with $n$ vertices.
  Then the following statements are equivalent:
  \begin{enumerate}[(a)]
  \item The metamour graph of~$G$ is a proper subgraph of $\complement{G}$.
  \item We have either $G=C_n$ and $n\ge6$
    (Theorem~\ref{thm:2mriff}\ref{it:thm-2mmr:Cn}) or $G$ is one of the graphs
    in Theorem~\ref{thm:2mriff}\ref{it:thm-2mmr:exceptional}.
  \item The graph~$G$ has diameter larger than~$2$.
  \end{enumerate}
\end{corollary}

Theorem~\ref{thm:2mriff} makes it also
possible to count how many different $2$-metamour-regular graphs with $n$ 
vertices there are.
We provide this in the following corollary.

\begin{corollary}
  \label{cor:2mmr:counting}
  The sequence of numbers $m_{=2}(n)$ of unlabeled connected
  $2$-metamour-regular graphs with $n$ vertices starts with
  \begin{equation*}
    \begin{array}{c|cccccccccccccccccccc}
      n
      & 1 & 2 & 3 & 4
      & 5 & 6 & 7 & 8 & 9
      & 10 & 11 & 12 & 13 & 14
      & 15 & 16 & 17 & 18 & 19
      & 20
      \\\hline
      \!m_{=2}(n)\!
      & 0 & 0 & 0 & 0
      & 1 & 11 & 9 & 7 & 5
      & 6 & 7 & 10 & 11 & 14
      & 18 & 22 & 26 & 34 & 40
      & 50
    \end{array}
  \end{equation*}
  and for $n\ge9$ we have
  \begin{equation*}
    m_{=2}(n) = p_3(n) + 1,
  \end{equation*}
  where $p_3(n)$ is the number of integer partitions%
  \footnote{For integer partitions,
    see footnote~\footref{footnote:partitions}
    on page~\pageref{footnote:partitions}.
    The function~$p_3(n)$ is \oeis{A008483} in~\cite{OEIS:2020}.}
  of~$n$ with parts at least~$3$.
\end{corollary}
The sequence of numbers~$m_{=2}(n)$ is \oeis{A334275} in
The On-Line Encyclopedia of Integer Sequences~\cite{OEIS:2020}.

We apply the Euler transform,
see Sloane and Plouffe~\cite{Sloane-Plouffe:1995:encyclopedia-int-seq},
on this sequence and obtain the
numbers~$m_{=2}'(n)$ of unlabeled but not necessarily connected
$2$-metamour-regular graphs with $n$ vertices. The sequence of these numbers
starts with
  \begin{equation*}
    \begin{array}{c|ccccccccccccccccccc}
      n
      & 1 & 2 & 3 & 4
      & 5 & 6 & 7 & 8 & 9
      & 10 & 11 & 12 & 13 & 14
      & 15 & 16 & 17 & 18
      \\\hline
      m_{=2}'(n)
      & 0 & 0 & 0 & 0
      & 1 & 11 & 9 & 7 & 5
      & 7 & 18 & 85 & 117 & 141
      & 143 & 179 & 277 & 667
    \end{array}\;.
  \end{equation*}

\subsection{\texorpdfstring{$k$}{k}-metamour-regular graphs}
\label{sec:results:kmr}

In this section, we present results that are valid for
graphs with maximum metamour-degree~$k$ and $k$-metamour-regular graphs
for any non-negative number~$k$.

We start with  Proposition~\ref{pro:join-of-complements-is-kmr} stating that 
the join of complements of $k$-regular graphs is a 
$k$-metamour-regular graph.

\begin{proposition}\label{pro:join-of-complements-is-kmr}
  Let $M$ be a graph having $t\ge2$ connected components
  $M_1$, \dots, $M_t$. Set
  \begin{equation*}
    G = \complement{M_1} \join \dots \join \complement{M_t}.
  \end{equation*}
  Then $G$ has metamour-graph~$M$.
  In particular, if $M$ is $k$-regular for some $k \ge 0$,
  then $G$ is $k$-metamour-regular.
\end{proposition}
We call this construction of a graph with given metamour graph
\emph{generic construction}. In particular this generic construction allows us 
to build $k$-metamour-regular graphs.
We will not investigate further options to construct $k$-metamour-regular 
graphs (as, for example, with circulant graphs), as the above construction suffices for the ultimate goal of this paper to 
characterize all $k$-metamour-regular graphs for $k \leq 2$.
Last in our discussion of Proposition~\ref{pro:join-of-complements-is-kmr},
we note that, as we have
$G = \complement{M_1} \join \dots \join \complement{M_t}$,
we can also extend Proposition~\ref{pro:join-of-complements-is-kmr}
by the equivalent statements of
Propositions~\ref{pro:join-and-mm-graphs}
and~\ref{pro:regular-join-and-mm-graphs}. 

Next we state the main structural result about graphs with maximum metamour-degree~$k$
as Theorem~\ref{thm:kmmr:metamour-graph-not-connected}.
A consequence of this main statement is that
every $k$-metamour-regular graph is the join of 
complements of $k$-regular graphs as in Proposition~\ref{pro:join-of-complements-is-kmr},
or has in its metamour graph only one or two connected
components; see Corollary~\ref{cor:kmmr:3cc-implies-join} for a full
formulation.

\begin{theorem}\label{thm:kmmr:metamour-graph-not-connected}
    Let $k \ge 0$. 
    Let $G$ be a connected graph with maximum metamour-degree~$k$ and
    $M$ its metamour graph. Then exactly one of the following
    statements is true:
    \begin{enumerate}[(a)]
      \item\label{it:general:mm-connected}
        The metamour graph~$M$ is connected.

      \item\label{it:general:general}
        The metamour graph~$M$ is not connected and the induced subgraph
        $G[V(M_i)]$ is connected for some $M_i \in \CC{M}$.
        
        In this case we have
        $G = \complement{M_1} \join \dots \join \complement{M_t}$ 
        with $\set{M_1,\dots,M_t} = \CC{M}$ and $t \ge 2$,
        and any other equivalent statement of
        Proposition~\ref{pro:join-and-mm-graphs}.
        
      \item\label{it:general:exceptional}
        The metamour graph~$M$ is not connected and
        no induced subgraph $G[V(M_i)]$ is connected for any $M_i \in \CC{M}$.

        In this case the metamour graph~$M$ has exactly
        two connected components and the following holds.
        Set $G^M = G[V(M_1)] \cup G[V(M_2)]$ with
        $\set{M_1,M_2} = \CC{M}$, i.e., $G^M$
        is the graph $G$ after deleting every edge 
        between two vertices that are from different connected components of 
        the metamour graph~$M$. Then we have:
        \begin{enumerate}[(i)]
            \item Every connected component of $G^M$
            has 
            at most $k$ vertices.
            \label{it:at-most-k-vert}
            \item
            If $G$ is $k$-metamour-regular,
            then every connected component of $G^M$ is a regular 
            graph.            
            \label{it:kmr-implies-regular-conn-comp}
            \item Every connected component $G_i$ of $G^M$ satisfies 
            $\complement{G_i} =  M[V(G_i)]$. 
            \label{it:complement-graph-is-metamour-graph}
            \item If two different connected components of $G^M$
              are adjacent in~$G$,
              then these connected components are completely adjacent in~$G$.%
            \label{it:one-egdge-implies-all-edges}
            \item If two vertices of different connected components $G_i$ and 
            $G_j$ of 
            $G^M$ have a common neighbor in~$G$, then every vertex of $G_i$ is 
            a metamour of every vertex of $G_j$.%
            \label{it:metamour-components-contain-metamours}
            \item  
            If a connected component $G_i$ of $G^M$ is adjacent in~$G$ to
            another connected component $G_j$
            consisting of $k-d$ 
            vertices for some $d \ge 0$,
            then the neighbors (in~$G$) of vertices of $G_i$
            are in at most $d+2$ (including $G_j$) connected components.%
            \label{it:not-too-many-components}
        \end{enumerate}
    \end{enumerate}
\end{theorem}

The complete and extensive proof of
Theorem~\ref{thm:kmmr:metamour-graph-not-connected}
can be found in Section~\ref{sec:proofs:k-mmr};
we again only sketch it at this point.

\begin{proof}[Sketch of Proof of Theorem~\ref{thm:kmmr:metamour-graph-not-connected}]
  Suppose the metamour graph of~$G$ is not connected. We split
  the graph respecting the connected components of the metamour graph of~$G$.

  First, we lift adjacency of some vertices of different components of
  this split to all vertices and likewise, metamour relations between
  vertices of different components. By this structural property, we
  show on the one hand that the maximum metamour-degree~$k$ implies
  that each of the components has at most~$k$ vertices. On the other
  hand, by investigating an alternating behavior on shortest paths
  between different components, we deduce that the metamour graph
  of~$G$ has exactly two connected components.

  Before gluing everything together, two more properties need to be
  derived: We show that within each of the components, the non-edges
  correspond exactly to metamour relations, and we bound the number of
  neighboring components of a component. This provides enough structure
  completing the proof.
\end{proof}

Let us discuss the three outcomes of
Theorem~\ref{thm:kmmr:metamour-graph-not-connected} in view of the
characterizations provided in Sections~\ref{sec:results:0mr}
to~\ref{sec:results:2mr}. 
Toward this end note that in case~\ref{it:general:general},
the graph $G$ is obtained 
by the generic construction.

For $0$-metamour-regular graphs due to Theorem~\ref{thm:0mr} there is no graph 
that is not obtained by the generic construction.
So only~\ref{it:general:general}
happens, except if the graph consists of only one vertex in which case
a degenerated~\ref{it:general:mm-connected} happens.
For $1$-metamour-regular graphs we know that 
there is only one graph not obtained by the generic construction by
Theorem~\ref{thm:1mr}.
This exceptional case is associated to~\ref{it:general:exceptional}, otherwise
we are in~\ref{it:general:general}.
Finally Theorem~\ref{thm:2mriff} states that beside the generic case
associated to~\ref{it:general:general},
there is only the class with graphs~$C_n$ and 
17 exceptional cases of $2$-metamour-regular graphs associated
to~\ref{it:general:mm-connected} and~\ref{it:general:exceptional}.

At last in this section, we bring and discuss the full formulation of a
statement mentioned earlier.

\begin{corollary}
  \label{cor:kmmr:3cc-implies-join}
    Let $k \ge 0$. 
    Let $G$ be a connected graph with maximum metamour-degree~$k$ and
    $M$ its metamour graph. Let $\set{M_1,\dots,M_t} = \CC{M}$. 
    If $t\ge3$, then we have
    \begin{equation*}
      G = \complement{M_1} \join \dots \join \complement{M_t}.
    \end{equation*}
\end{corollary}

By this corollary every $k$-metamour-regular graph that 
has at least three connected components in its metamour graph
is a join of 
complements of $k$-regular graphs and therefore
can be built by the generic construction.
As a consequence, it is only possible that a $2$-metamour-regular graph is not 
obtained by the generic construction if its metamour graph has at most two 
connected components. 

\section{Proofs regarding foundations}
\label{sec:proof:foundations}

We start by proving Proposition~\ref{pro:mm-eq-complement:diameter} which relates
the metamour graph, the complement and the diameter of a graph.

\begin{proof}[Proof of Proposition~\ref{pro:mm-eq-complement:diameter}]
  Suppose~\ref{it:mm-eq-complement:diameter:M-eq-compl} holds.
  Let $v$ and $w$ be vertices in~$G$. If $v=w$, then their distance
  is~$0$. If $v$ and $w$ are adjacent in~$G$, then their distance
  is~$1$. Otherwise, there is no edge~$\edge{v}{w}$ in $G$.
  Therefore, this edge is in~$\complement{G}=M$, where $M$ is the metamour 
  graph of~$G$. This implies
  that~$v$ and $w$ are metamours. Thus, the distance between~$v$ and
  $w$ is~$2$. Consequently, no distance in~$G$ is larger than~$2$, which 
  implies that the diameter of $G$ is at most $2$.
  As a result, either the diameter of $G$ is exactly $2$, or it is at most $1$.
  If the diameter is equal to $1$, then
  all vertices of $G$ are pairwise adjacent and therefore $G = K_n$.
  Furthermore, there are at least two vertices at distance $1$ and hence $n\ge 2$.
  If the diameter is $0$, then $G = K_1$.

  Now suppose \ref{it:mm-eq-complement:diameter-2} holds. If $G=K_n$, then its 
  diameter is either equal to~$0$ if $n=1$ or equal to $1$ if $n \ge 2$.
  Therefore, in this case the diameter of $G$ is at most $2$.
  Now let~$\edge{v}{w}$ be an edge in $\complement{G}$. Then $v \ne w$ and
  these vertices are not adjacent in~$G$, so their distance is at
  least~$2$.  As the diameter is at most~$2$, the distance of~$v$
  and~$w$ is at most~$2$. Consequently their distance is exactly~$2$
  implying that they are metamours. So the edge $\edge{v}{w}$ is in $M$, and 
  hence the 
  complement of $G$ is a subgraph of the metamour graph of~$G$. Due to 
  Observation~\ref{observation:mm-subgraph-complement}, the metamour graph of 
  $G$ is a subgraph of the complement of~$G$, hence the metamour graph of $G$
  and $\complement{G}$ coincide,
  so we have shown \ref{it:mm-eq-complement:diameter:M-eq-compl}.
\end{proof}

Next we prove Proposition~\ref{pro:mm-eq-complement:degree} which relates the metamour 
graph, the complement, and the degree and metamour-degree of the vertices of a 
graph. 

\begin{proof}[Proof of Proposition~\ref{pro:mm-eq-complement:degree}]
  We use that
  for a graph~$G$ with $n$~vertices, the sum of the degrees of a vertex
  in~$G$ and in the complement~$\complement{G}$ always equals
  $n-1$.

  If \ref{it:mm-eq-complement:degree:M-eq-compl} holds, then
  the degree of a vertex in the metamour graph equals the degree in the
  complement~$\complement{G}$. This yields \ref{it:mm-eq-complement:degree}
  by using the statement at the beginning of this proof.
  
  If~\ref{it:mm-eq-complement:degree} holds,
  then the sum of the degree and the metamour-degree of a vertex equals
  the sum of the degree in $G$ and the degree in $\complement{G}$ 
  of this vertex by the statement
  at the beginning of this proof. Therefore, the metamour-degree of
  a vertex is equal to the degree in $\complement{G}$ of this vertex.
  Due to Observation~\ref{observation:mm-subgraph-complement},
  the metamour graph of~$G$ is a subgraph of~$\complement{G}$, and therefore  
  the metamour graph of~$G$ equals $\complement{G}$.
\end{proof}

Finally we prove Propositions~\ref{pro:join-and-mm-graphs} 
and~\ref{pro:regular-join-and-mm-graphs} that relate the metamour graph and 
joins.

\begin{proof}[Proof of Proposition~\ref{pro:join-and-mm-graphs}]
  We start by showing that \ref{it:join-and-mm-graphs:M-eq-compl} implies \ref{it:join-and-mm-graphs:join-mindestens-2}.
  In~$M$ there are no edges between its different connected components.
  Therefore,
  in the complement~$\complement{M}=G$, 
  there are all possible edges between the vertices of different components.
  This is equivalent to the definition of the join of graphs;
  the individual graphs in~$\CC{M}$ are complemented, and
  consequently \ref{it:join-and-mm-graphs:join-mindestens-2} follows.

  For proving that \ref{it:join-and-mm-graphs:join-mindestens-2} implies \ref{it:join-and-mm-graphs:join-of-2}, we simply set $G_1=\complement{M_1}$
  and $G_2=\complement{M_2} \join \dots \join \complement{M_t}$.
  By the definition of the operator~$\join$, \ref{it:join-and-mm-graphs:join-of-2}
  follows.
  
  We now show that \ref{it:join-and-mm-graphs:join-of-2} implies \ref{it:join-and-mm-graphs:M-eq-compl}.
  Every pair of vertices of $G_1$ has a common neighbor in 
  $G_2$. This implies that the vertices of this pair are metamours
  if and only if they are not adjacent in~$G_1$.
  The same holds for any pair of vertices of $G_2$ by symmetry
  or due to commutativity of the operator~$\join$.
  As a consequence of this and because $G = G_1 \join G_2$
  and the definition of the operator~$\join$,
  the metamour graph~$M$ and the complement of $G$ coincide.
  As every possible edge from a vertex of~$G_1$ to a vertex of~$G_2$
  exists in~$G$, this complement~$\complement{G}$ has at least
  two connected components, hence~$t\ge2$.
\end{proof}

\begin{proof}[Proof of Proposition~\ref{pro:regular-join-and-mm-graphs}]
  \ref{it:pro:regular-join-and-mm-graphs:M-eq-compl} and
  \ref{it:pro:regular-join-and-mm-graphs:diameter-2} are equivalent by
  Proposition~\ref{pro:mm-eq-complement:diameter}.
  
  If $\abs{\CC{M}}=t=1$, then \ref{it:pro:regular-join-and-mm-graphs:M-eq-compl} and
  \ref{it:pro:regular-join-and-mm-graphs:join} are trivially equivalent.
  If $\abs{\CC{M}}=t\ge2$, then this equivalence is part of
  Proposition~\ref{pro:join-and-mm-graphs}.

  Finally, the equivalence of~\ref{it:pro:regular-join-and-mm-graphs:M-eq-compl}
  and~\ref{it:pro:regular-join-and-mm-graphs:regular}
  follows from Proposition~\ref{pro:mm-eq-complement:degree}.
\end{proof}

At this point we have shown all results form Section~\ref{sec:formal-stuff} and 
therefore have laid the foundations of the subsequent results.

\section{Proofs regarding \texorpdfstring{$k$}{k}-metamour-regular graphs}
\label{sec:proofs:k-mmr}

Next we give the proofs of results
from Section~\ref{sec:results:kmr}
concerning graphs with maximum metamour-degree~$k$ and
$k$-metamour-regular graphs
that are valid for arbitrary $k \ge 0$. 

\begin{proof}[Proof of Proposition~\ref{pro:join-of-complements-is-kmr}]
  Let $v$ be a vertex of $G$. Then $v$ is in $\complement{M_i}$ and
  therefore in $M_i$ for some $i\in\set{1,\dots,t}$.
    
    Let $u \neq v$ be a vertex of $G$ and $j\in\set{1,\dots,t}$ such that
    $u$ is in $\complement{M_j}$.
    If $j \neq i$,
    then $u$ and $v$ are adjacent in $G$ by 
    construction. Therefore, they are not metamours.
    If $j = i$, then any vertex not in $M_i$, i.e., in any of 
    $\complement{M_1}$, \dots, $\complement{M_{i-1}}$, 
    $\complement{M_{i+1}}$, \dots, $\complement{M_t}$, is 
    a common neighbor of $u$ and $v$. Therefore, $u$ and $v$ are metamours 
    if and only if they are not adjacent in $\complement{M_i}$, and
    this is the case if and only if they are adjacent in $M_i$.

    Summarized, we have that $u$ is a metamour of $v$ if and only if
    $u$ is in $M_i$ and adjacent to~$v$ in~$M_i$.
    This yields that the metamour graph
    of~$G$ is~$M=M_1 \cup \dots \cup M_t$ which was to show.

    The $k$-metamour-regularity follows directly from
    Observation~\ref{obs:kmr-iff-metamourgraphkr}.
\end{proof}

\resetbeh
\begin{proof}[Proof of Theorem~\ref{thm:kmmr:metamour-graph-not-connected}]
  If the metamour graph~$M$ is connected, then
  we are in case~\ref{it:general:mm-connected} and nothing is to show.
  So suppose that the metamour graph is not connected.
  
  We partition the vertices of~$M$ (and therefore the vertices of~$G$)
  into two parts $V' \uplus V^\ast$ such that the vertex set of
  each connected component of~$M$ is a subset of either~$V'$ or $V^\ast$ and 
  such that neither~$V'$ nor~$V^\ast$ is empty,
  i.e., we partition by the connected components~$\CC{M}$
  of the metamour graph~$M$.
  As $M$ is not connected, it consists of at least two connected components,
  and therefore such a set-partition of the vertices of~$M$ is always possible.
  We now split up the graph~$G$ into the two
  subgraphs~$G' = G[V']$ and $G^\ast = G[V^\ast]$.
  
  Rephrased, we obtain $G'$ and $G^\ast$ from~$G$ by cutting it (by deleting edges)
  in two, but respecting and not cutting the connected components of its metamour graph~$M$.
  Note that the formulation is symmetric with respect to~$G'$ and $G^\ast$,
  therefore, we might switch the two without loss of generality during
  the proof. This also implies that in the statements of the following claims
  we may switch the roles of $G'$ and $G^\ast$.

    \begin{beh}\label{beh:kmmr:one-edge-implies-all-edges}
      If $G'_1\in\CC{G'}$ is adjacent in $G$ to
      $G^\ast_1\in\CC{G^\ast}$, then
      $G'_1$ is completely adjacent in $G$ to $G^\ast_1$.
    \end{beh}
    \begin{proofbeh}{beh:kmmr:one-edge-implies-all-edges}
      Suppose $u' \in V(G_1')$ and $v^\ast \in V(G^\ast_1)$ are adjacent
      in~$G$.
      Let $u \in V(G_1')$ and $v \in V(G^\ast_1)$.
      We have to prove that $\edge{u}{v} \in E(G)$.

        There is a path $\pi_{u',u}=(u_1, u_2, \dots, u_r)$ from 
        $u_1=u'$ to $u_r=u$ in $G'_1$ because $G'_1$ is connected.
        Furthermore, there is a path $\pi_{v^\ast,v}=(v_1, v_2, \dots, v_s)$
        from $v_1=v^\ast$ to $v_s=v$ in $G^\ast_1$ because $G^\ast_1$ is 
        connected.
        
        We use induction to prove 
        that $\edge{u_1}{v_\ell} \in E(G)$ for all $\ell\in\set{1,\dots,s}$.
        Indeed, this 
        is true for $\ell = 1$ by assumption. So assume $\edge{u_1}{v_\ell} \in E(G)$.
        We have $\edge{v_\ell}{v_{\ell+1}} \in E(G)$ because this is an edge of the 
        path $\pi_{v^\ast,v}$. If $\edge{u_1}{v_{\ell+1}} \not \in E(G)$, then 
        $u_1$ and $v_{\ell+1}$ are metamours. But $u_1$ has 
        all its metamours in $G'$, a contradiction. Hence, 
        $\edge{u_1}{v_{\ell+1}} \in E(G)$, which finishes the induction.
        In particular, we have proven that $\edge{u_1}{v_s} \in E(G)$.
        
        Now we prove that $\edge{u_\ell}{v_s} \in E(G)$ holds for all
        $\ell \in \set{1,\dots,r}$ by induction.
        This holds for $\ell = 1$ by the above. Now 
        assume $\edge{u_\ell}{v_s} \in E(G)$. We have $\edge{u_\ell}{u_{\ell+1}} \in 
        E(G)$ because this edge is a part of the path $\pi_{u',u}$. If 
        $\edge{u_{\ell+1}}{v_s} \not \in E(G)$, then 
        $u_{\ell+1}$ and $v_s$ are metamours, a contradiction since every 
        metamour of $u_{\ell+1}$ is in $G'$. Therefore, 
        $\edge{u_{\ell+1}}{v_s} \in E(G)$ holds and the induction is completed.
        As a result, we have $\edge{u}{v} = \edge{u_r}{v_s} \in E(G)$. 
    \end{proofbeh}

    \begin{beh}\label{beh:kmmr:one-graph-connected}
      Let $\CC{M} = \set{M_1,\dots,M_t}$.     
      If the graph $G'$ is connected, then
      $G = G' \join G^\ast = \complement{M_1} \join \dots \join \complement{M_t}$
      with $t \ge 2$.
    \end{beh}
    \begin{proofbeh}{beh:kmmr:one-graph-connected}
        The graph $G$ is connected, so every connected component of 
        $G^\ast$ is adjacent in $G$ to $G'$.
        By~\behref{beh:kmmr:one-edge-implies-all-edges} this 
        implies 
        that $G'$ is completely adjacent in $G$ to $G^\ast$, i.e.,
        all possible edges between $G'$ and $G^\ast$ exist. 
        Therefore, $G = G' \join G^\ast$ by the definition of the
        join of graphs.

        By Proposition~\ref{pro:join-and-mm-graphs}, the full decomposition
        into the components~$\CC{M}$ follows.
    \end{proofbeh}
    As a consequence of~\behref{beh:kmmr:one-graph-connected}, we are finished 
    with the proof in the case that $G[V(M_i)]$ is connected
    for some~$M_i\in\CC{M}$ because statement~\ref{it:general:general} follows
    by setting~$G'=G[V(M_i)]$.
    
    So from now on we consider the case that every $G[V(M_i)]$ with
    $M_i\in\CC{M}$ has at least two connected components. This is the
    set-up of statement~\ref{it:general:exceptional}.
    
    \begin{beh}\label{beh:kmmr:metamour-components-are-metamours}
      Suppose we are in the set-up of~\ref{it:general:exceptional}.
      Let $G'_1\in\CC{G'}$ and $\set{G^\ast_1,G^\ast_2} \subseteq \CC{G^\ast}$.
      If both $G^\ast_1$ and $G^\ast_2$ are adjacent in $G$ to $G'_1$,
      then every vertex of $G^\ast_1$ 
        is a metamour of every vertex of $G^\ast_2$.
    \end{beh}
    \begin{proofbeh}{beh:kmmr:metamour-components-are-metamours}
      As both $G^\ast_1$ and $G^\ast_2$ are adjacent in $G$ to $G'_1$,
      they are both completely adjacent in~$G$ to~$G'_1$ due 
      to~\behref{beh:kmmr:one-edge-implies-all-edges}.
      Furthermore, $G^\ast_1$ is not adjacent in~$G$ to $G^\ast_2$, i.e.,
        no vertex of $G^\ast_1$ is adjacent in $G$ to 
        any vertex of $G^\ast_2$, because they are in different connected 
        components of $G^\ast$. Hence, every vertex of $G^\ast_1$ 
        is a metamour of every vertex of~$G^\ast_2$.
    \end{proofbeh}

    \begin{beh}\label{beh:kmmr:connected-components-at-most-k-vertices}
      Suppose we are in the set-up of~\ref{it:general:exceptional}.
      Then every connected component of $G'$
      has at most $k$ vertices and this connected component's vertex set
      is a subset of the vertex set of one
        connected component of the metamour graph~$M$.
    \end{beh}
    \begin{proofbeh}{beh:kmmr:connected-components-at-most-k-vertices}
      Let $G'_1 \in \CC{G'}$.
      As $G'$ is not connected but the graph~$G$ is connected,
      there is a path $\pi$ from a vertex
      of $G'_1$ to some vertex in some connected component of $\CC{G'}$
      other than~$G'_1$. By construction of $G'$ and $G^\ast$, the path
      $\pi$ splits from start to end into vertices of $G'_1$, followed
      by vertices of some $G^\ast_1 \in \CC{G^\ast}$,
      followed by some vertices of $G'_2 \in \CC{G'}$, and remaining vertices.
      Therefore, we have connected components~$G'_2$ and $G^\ast_1$
      such that at least one vertex of $G'_1$ is connected
      to some vertex of $G^\ast_1$ and at least one 
      vertex of $G'_2$ is connected
      to some vertex of $G^\ast_1$.
      
      Then, due to~\behref{beh:kmmr:metamour-components-are-metamours} every vertex
        of 
        $G'_1$ is a metamour of every vertex of $G'_2$.
        From this, we now deduce two statements.
                
        First, if we assume that $G'_1$ contains at least $k+1$ vertices, 
        then every vertex of $G'_2$ has at least $k+1$ metamours, a 
        contradiction to $k$ being the maximum metamour-degree of~$G$.
        Therefore, $G'_1$ contains at most $k$ vertices.

        Second, every vertex of $G'_1$ is adjacent in the
        metamour graph~$M$ to every vertex of $G'_2$, so
        $G'_1$ is completely adjacent in~$M$ to~$G'_2$.
        Therefore,
        all these vertices are in the same connected component of the
        metamour graph~$M$.  In particular, this is true for the set
        of vertices of $G'_1$ as claimed, and so the proof is complete.
    \end{proofbeh}

    \begin{beh}\label{beh:kmmr:shortest-path-alternating}
      Suppose we are in the set-up of~\ref{it:general:exceptional}.
        Let $v_1$ and $v_2$ be 
        two vertices of different connected components of $G'$. 
        Then every shortest path 
        from $v_1$ to $v_2$ in $G$ consists of vertices
        alternating between $G'$ and $G^\ast$. 
    \end{beh}
    \begin{proofbeh}{beh:kmmr:shortest-path-alternating}
        Let $\set{G'_1,G'_2} \subseteq \CC{G'}$ such that
        $v_1 \in G'_1$ and $v_2 \in G'_2$.
        Let $\pi=(u_1, u_2, \dots, u_r)$ be a shortest path from
        $u_1=v_1$ to $u_r=v_2$ in $G$.
        Note that $u_1$ and $u_r$ are both from $G'$ but from different connected 
        components. Hence, $\pi$ consists of at least two vertices from 
        $G'$ and at least one vertex from $G^\ast$. 
        
        Assume that the vertices of $\pi$ are not
        alternating between $G'$ and
        $G^\ast$. Then,
        without loss of generality (by reversing the enumeration
        of the vertices in the path~$\pi$),
        there exist indices~$i<j$
        and graphs $\set{\widetilde{G},\widehat{G}} = \set{G', G^\ast}$
        with the following properties:
        Every vertex of the subpath
        $\pi_{i,j}=(u_i, u_{i+1}, \dots, u_j)$
        of~$\pi$ is of $\widetilde{G}$, and
        the vertex $u_{j+1}$ exists and is in $\widehat{G}$.

        As $\pi_{i,j}$ is a path, all of its vertices
        are of the same connected component of $\widetilde{G}$.
        As $u_j$ is adjacent to $u_{j+1}$ and
        due to~\behref{beh:kmmr:one-edge-implies-all-edges},
        the vertex~$u_{j+1}$ is adjacent to every vertex of $\pi_{i,j}$,
        in particular, adjacent to~$u_i$. But then
        $u_1$, \dots, $u_i$, $u_{j+1}$, \dots, $u_r$ is a 
        shorter path between~$v_1$ and~$v_2$, a contradiction.
        Hence, our initial assumption 
        was wrong, and the vertices of $\pi$ are alternating
        between $G'$ and~$G^\ast$.
    \end{proofbeh}

    \begin{beh}\label{beh:kmmr:not-connected-metamour-graph-2-components}
      Suppose we are in the set-up of~\ref{it:general:exceptional}.
        Then the metamour graph~$M$ has exactly two connected components.
    \end{beh}
    \begin{proofbeh}{beh:kmmr:not-connected-metamour-graph-2-components}
        The metamour graph~$M$ is not connected, therefore it has at least 
        two connected components.
        Assume it has at least three components. 
        Then, without loss of generality (by switching $G'$ and $G^\ast$),
        the graph $G'$ contains 
        vertices of at least two different connected components
        of~$M$.
        Let $v_1$ and $v_2$ be two vertices of $G'$ that are in 
        different connected components of~$M$. If follows 
        from~\behref{beh:kmmr:connected-components-at-most-k-vertices} that 
        $v_1$ and $v_2$ are in different connected components of $G'$. 
        Let $\set{G'_1,G'_2} \subseteq \CC{G'}$ such that
        $v_1 \in G'_1$ and $v_2 \in G'_2$.
        Let $\pi=(u_1,\dots,u_r)$ be a shortest path from $u_1=v_1$ to $u_r=v_2$ in $G$.
        
        Now consider two vertices $u_{2i - 1}$ and $u_{2i+1}$ for $i \ge 1$ 
        of $\pi$. Both $u_{2i - 1}$ and $u_{2i+1}$ are from $G'$ 
        because $\pi$ 
        consists of alternating vertices from $G'$ and $G^\ast$ due 
        to~\behref{beh:kmmr:shortest-path-alternating} and $u_1$ is from 
        $G'$. 
        If both $u_{2i-1}$ and $u_{2i+1}$ are from the same connected 
        component of $G'$, then they are in the same connected component of the 
        metamour graph~$M$
        by~\behref{beh:kmmr:connected-components-at-most-k-vertices}.
        If $u_{2i-1}$ and $u_{2i+1}$ are in different connected components of 
        $G'$, then they are metamours because they have the common neighbor 
        $u_{2i}$ and they are not adjacent. Therefore, they are in the same 
        connected component of $M$ as well.
        Hence, in any case $u_{2i - 1}$ and $u_{2i+1}$ are in the same connected 
        component of~$M$.
        By induction this implies that $u_1=v_1$ is in the same connected 
        component of~$M$ as $u_r = v_2$, a contradiction to  
        $v_1$ and $v_2$ being from different connected components of~$M$. 
        Hence, our assumption was wrong and the metamour graph consists of exactly two 
        connected components.
    \end{proofbeh}

    \begin{beh}\label{beh:kmmr:connected-component-is-regular}
      Suppose we are in the set-up of~\ref{it:general:exceptional}.
      Then every connected component $G'_1$ of $G'$
      satisfies $\complement{G'_1} = M[V(G'_1)]$.
        If $G$ is $k$-metamour-regular, then the
        connected component is a regular graph.
    \end{beh}
    \begin{proofbeh}{beh:kmmr:connected-component-is-regular}    
      Let $G'_1$ be a connected component of $G'$.
      As $G$ is connected, there is an edge from
      a vertex $v_1$ of $G'_1$ to $G^\ast$ and this is
        extended to every vertex of $G'_1$
        by~\behref{beh:kmmr:one-edge-implies-all-edges}.
        Therefore, two different vertices of $G'_1$ are metamours if and only 
        if they are not adjacent in $G'$. Restricting this
        to the subgraph~$G'_1$ yields the first statement.

      Furthermore, by construction of
        $G'$ and $G^\ast$, every metamour of~$v_1$ is in $G'$.
        Let $v'$ be such a metamour, and suppose that $v'$
        is not in $G'_1$.
        The vertices~$v_1$ and $v'$ have a common neighbor $u$ that has to be 
        in $G^\ast$ as $v_1$ and $v'$
        are in different connected components of~$G'$.
        The vertex~$u$ is completely adjacent in $G$ to
        $G'_1$ because 
        of~\behref{beh:kmmr:one-edge-implies-all-edges}.
        Hence, $v'$ is a
        metamour of every vertex of $G'_1$.
        As a consequence, every vertex of $G'_1$ has the same number 
        of metamours outside of $G'_1$,
        i.e., in $G'$ but not in $G'_1$.
        If no such pair of vertices $v_1$ of $G'_1$ and
        $v'$ not of $G'_1$ that are metamours exists,
        then every vertex of $G'_1$
        still has the same number of metamours outside of $G'_1$,
        namely zero.

        Now let us assume that $G$ is $k$-metamour-regular.
        As every vertex of $G'_1$ has the same number of metamours outside of 
        $G'_1$, this implies that every vertex of $G'_1$ must also have the 
        same number of metamours inside $G'_1$.
        We combine this with the results of first paragraph and conclude that
        every vertex of $G'_1$ 
        is adjacent to the same number of vertices of $G'_1$, and 
        hence $G'_1$ is a regular graph.
    \end{proofbeh}

    \begin{beh}\label{beh:kmmr:not-adjacent-to-too-many-components}
      Suppose we are in the set-up of~\ref{it:general:exceptional}.
        If a connected component $G'_1\in\CC{G'}$ is adjacent in~$G$ to
        a connected component $G^\ast_1\in\CC{G^\ast}$
        consisting of $k-d$ vertices for some $d \ge 0$,
        then the neighbors (in~$G$) of vertices of $G'_1$ that are in $G^\ast$
        are in at most $d+2$ connected components of $G^\ast$
        (including $G^\ast_1$).
    \end{beh}
    \begin{proofbeh}{beh:kmmr:not-adjacent-to-too-many-components}
      Let $\mathcal{G}^\ast \subseteq \CC{G^\ast}$ be such that a
      connected components of $G^\ast$ is in $\mathcal{G}^\ast$
      if and only if it is adjacent in~$G$ to~$G'_1$.
      $G^\ast_1$ consists of $k-d$ vertices and
      there is some vertex $v'$ of $G'_1$ that is 
      adjacent to some vertex of $G^\ast_1$.
      
      We have to prove that $\abs{\mathcal{G}^\ast} \le d + 2$
      in order to finish the proof.
      So assume $\abs{\mathcal{G}^\ast} > d + 2$. Let $v^\ast$
      be a vertex of some connected component $G^\ast_2\in\mathcal{G}^\ast$
      other than $G^\ast_1$.
        Then $v^\ast$ is adjacent to $v'$ due 
        to~\behref{beh:kmmr:one-edge-implies-all-edges}. Because 
        of~\behref{beh:kmmr:metamour-components-are-metamours}, every vertex
        in any connected component in $\mathcal{G}^\ast$
        except $G^\ast_2$ is a metamour of $v^\ast$.
        The component $G^\ast_1$ contains $k-d$ vertices and
        there are at least $d+1$ other components each containing
        at least one vertex. In total, $v^\ast$
        has at least $(k - d) + (d+1) = k+1$ 
        metamours, a contradiction to $k$ being the maximum metamour-degree of~$G$.
        Therefore, our assumption was wrong and
        $\abs{\mathcal{G}^\ast} \le d + 2$ holds.
    \end{proofbeh}
    
    Now we are able to collect everything we have proven so far and finish the 
    proof of statement~\ref{it:general:exceptional}.
    Due 
    to~\behref{beh:kmmr:not-connected-metamour-graph-2-components}, the metamour 
    graph~$M$ of $G$ consists of exactly two connected components, and
    consequently the connected 
    components of $G^M$ coincide with the union of the
    connected components of $G'$ and $G^\ast$.
    Then~\behref{beh:kmmr:connected-components-at-most-k-vertices}
    implies~\ref{it:at-most-k-vert}, 
    \behref{beh:kmmr:connected-component-is-regular}
    implies~\ref{it:kmr-implies-regular-conn-comp}, 
    \behref{beh:kmmr:connected-component-is-regular}
    implies~\ref{it:complement-graph-is-metamour-graph}, 
    \behref{beh:kmmr:one-edge-implies-all-edges}
    implies~\ref{it:one-egdge-implies-all-edges}, 
    \behref{beh:kmmr:metamour-components-are-metamours}
    implies~\ref{it:metamour-components-contain-metamours}
    and~\behref{beh:kmmr:not-adjacent-to-too-many-components}
    implies~\ref{it:not-too-many-components}.
    This completes the proof.
\end{proof}

\begin{proof}[Proof of Corollary~\ref{cor:kmmr:3cc-implies-join}]
  As the metamour graph consists of at least~$3$ connected components,
  we cannot land in the cases~\ref{it:general:mm-connected}
  and~\ref{it:general:exceptional} of
  Theorem~\ref{thm:kmmr:metamour-graph-not-connected}. But then the
  statement of the corollary follows from~\ref{it:general:general}.
\end{proof}

Now we have proven everything we need to know about
graphs with maximum metamour-degree~$k$ and
$k$-regular-metamour graphs for general~$k$ and can use this knowledge
to derive the results we need 
in order to characterize all $k$-regular-metamour graphs for $k\in\set{0,1,2}$.

\section{Proofs regarding \texorpdfstring{$0$}{0}-metamour-regular graphs}
\label{sec:proofs:0-mmr}

We are now ready to prove all results concerning $0$-metamour-regular graphs 
from Section~\ref{sec:results:0mr}. 
In order to do so we first need the following lemma.

\begin{lemma}
\label{lem:noMetamour_adjacentToAll}
Let $G$ be a connected graph. If a vertex has no metamour, then it is 
adjacent 
to all other vertices of $G$.
\end{lemma}
\begin{proof}
    Let $n$ be the number of vertices of $G$. 
    Clearly the statement is true for 
    $n = 1$ as no other vertices are present and for $n = 2$ as the two vertices are adjacent due to connectedness.
    So let $n \ge 3$, and let $v$ be a vertex of $G$ that 
    has no metamour.
    
    Assume there is a vertex $w \neq v \in V(G)$ such that 
    $\edge{v}{w} \not \in E(G)$. 
    $G$ has a spanning tree $T$ because $G$ is connected.
    Let $v = u_1$, $u_2$, \dots, $u_r = w$ be the vertices on the unique 
    path from $v$ to $w$ in $T$.
    Then $\edge{u_i}{u_{i+1}} \in E(G)$ for all $i \in \set{1,\dots,r-1}$, so 
    due 
    to our assumption $r \ge 3$ holds. 
    In particular, $\edge{u_1}{u_2} \in E(G)$.
    If $\edge{u_1}{u_3} \not \in E(G)$, then both $u_1$ and $u_3$ are 
    adjacent 
    to $u_2$, but not adjacent to each other and therefore would be metamours. 
    But $u_1 = v$ does not have a metamour, hence 
    $\edge{u_1}{u_3} \in E(G)$.
    By induction $\edge{u_1}{u_i} \in E(G)$ for all $i\in\set{1,\dots,r}$ and 
    thus 
    $\edge{v}{w} = \edge{u_1}{u_r} \in E(G)$, a contradiction.
\end{proof}

Now we are able to prove Theorem~\ref{thm:0mr} that provides a characterization 
of $0$-metamour-regular graphs.

\begin{proof}[Proof of Theorem~\ref{thm:0mr}]
If $G$ is $0$-metamour-regular then every vertex of $G$ has no metamour and 
hence $G = K_n$ due to 
Lemma~\ref{lem:noMetamour_adjacentToAll}. Furthermore, $K_n$ is $0$-metamour-regular
as every vertex is adjacent to all other vertices.
\end{proof}

Next we prove the corollaries which yield an alternative 
characterization of $0$-metamour-regular graphs and allow to count 
$0$-metamour-regular graphs.

\begin{proof}[Proof of Corollary~\ref{cor:0mmr-complement}]
    Due to Theorem~\ref{thm:0mr}, a connected graph with $n$ vertices is 
    $0$-metamour-regular if and only if is equal to $K_n$. This is the case if 
    and only if its complement has no edges. In this case the complement 
    also equals the metamour graph.
\end{proof}

\begin{proof}[Proof of Corollary~\ref{cor:0mmr-counting}]
  This is an immediate and easy consequence of the
  characterization provided by Theorem~\ref{thm:0mr}.
\end{proof}

\section{Proofs regarding \texorpdfstring{$1$}{1}-metamour-regular graphs}
\label{sec:proofs:1-mmr}
In this section we present the proofs of the results from 
Section~\ref{sec:results:1mr}. They lead to a 
characterization of $1$-metamour-regular graphs.

\begin{proof}[Proof of Proposition~\ref{pro:1mr_evenNrVertices_perfectMatching}]
    Whenever a vertex $v \in V(G)$ is a metamour of a vertex $w \neq v 
    \in V(G)$ then also $w$ is a metamour of $v$. 
    Therefore, supposing that $v$ has exactly one metamour,
    so has~$w$ and
    the vertices $v$ and $w$ form a pair such that 
    the two vertices of the pair are metamours of each other.
    This also leads to an edge from $v$ to $w$ in the metamour graph of~$G$.
    As every vertex has at most one metamour,
    the edge from $v$ to $w$ is isolated in the metamour graph of~$G$, so $v$ and $w$ have no other 
adjacent
    vertices in the metamour graph of~$G$. Therefore, the edges of the metamour 
    graph form a matching, which 
    yields~\ref{pro:1mr_evenNrVertices_perfectMatching:it:atMostOneMatching}.

    Suppose now additionally that $G$ is $1$-metamour-regular.
    Then we can partition the vertices of $G$ into pairs of
    metamours. 
    Hence, $n$ is 
    even and the edges of the metamour graph form a perfect matching, 
    so~\ref{pro:1mr_evenNrVertices_perfectMatching:it:excatlyOnePerfectMatching}
     holds.
\end{proof}

For proving Theorem~\ref{thm:1mr}, we need some auxiliary results.
We start by showing that the graphs mentioned in the theorem
are indeed $1$-metamour-regular.

\begin{proposition}\label{proposition:P4-1mmr}
    The graph~$P_4$ depicted in Figure~\ref{fig:P4} is
    $1$-metamour-regular. 
\end{proposition}
\begin{proof}
    This is checked easily.
\end{proof}

The following proposition is slightly more general than needed in the proof of
Theorem~\ref{thm:1mr} and will be used later on.

\begin{proposition}\label{proposition:Kn-minus-M-le1mmr}
  Let $n \ge 3$. In the graph $G = K_n - \mu$ with a
  matching~$\mu$ of $K_n$, every vertex has at most one metamour.
\end{proposition}
\begin{proof}
    Let $v$ be an arbitrary vertex of $G$. Suppose $v$ is not incident to any
    edge in~$\mu$, then $v$ is adjacent to all other vertices. Thus, $v$ has
    no metamour.

    Now suppose that $v$ is incident to some edge in~$\mu$, and let
    the vertex~$v'$ be the other vertex incident to this edge.
    Then 
    clearly $\edge{v}{v'} \not \in E(G)$,
    so both $v$ as well as $v'$ have to be adjacent to 
    all other vertices of $G$ by construction of~$G$. 
    Due to the assumption $n \ge 3$, there is 
    at least one other vertex besides $v$ and $v'$, and this vertex is a common
    neighbor of them. Hence, $v$ and $v'$ are metamours of each other.
    Both $v$ and $v'$ do not have any other metamour because they are adjacent 
    to all other vertices. As a result, $v$ has exactly one metamour.
\end{proof}

\begin{proposition}\label{proposition:Kn-minus-M-1mmr}
    Let $n \ge 4$ be even. The graph $G = K_n - \mu$ with a
    perfect matching~$\mu$ of $K_n$ is $1$-metamour-regular.
\end{proposition}
\begin{proof}
    As the matching~$\mu$ is perfect, every vertex~$v$ of $G$ is incident to
    one edge in~$\mu$. Thus, by the proof of Proposition~\ref{proposition:Kn-minus-M-le1mmr},
    every~$v$ has exactly one metamour.
\end{proof}

We are now ready for proving Theorem~\ref{thm:1mr}.

\begin{proof}[Proof of Theorem~\ref{thm:1mr}]
  The one direction of the equivalence follows directly from 
  Proposition~\ref{proposition:P4-1mmr} and 
  Proposition~\ref{proposition:Kn-minus-M-1mmr}, so only the other direction is 
  left to prove.

  Suppose we have a graph~$G$ with $n$ vertices that is
  $1$-metamour-regular.
  Due to 
  Proposition~\ref{pro:1mr_evenNrVertices_perfectMatching}\ref{pro:1mr_evenNrVertices_perfectMatching:it:excatlyOnePerfectMatching},
   $n$ is even,
  and the set of edges of the metamour graph of $G$ forms a perfect matching.
  In particular, each connected component of the metamour graph consists of two 
  adjacent vertices.
  
  If the metamour graph is connected, then it consists of only two adjacent vertices 
  and $n=2$. 
  This can be ruled out easily, so we have~$n\ge4$ and the metamour graph is 
  not connected.
  Now we can use Theorem~\ref{thm:kmmr:metamour-graph-not-connected} and 
  see that one of the two cases~\ref{it:general:general} 
  and~\ref{it:general:exceptional} applies.
  
  In the first case~\ref{it:general:general} we have $G = \complement{M_1} 
  \join \dots \join \complement{M_t}$ 
  with
  $n=\abs{M_1}+\dots+\abs{M_t}$ for some~$t 
  \ge 
  2$,
  where $M_i$ is a connected $1$-regular graph
  for all $i\in\set{1,\dots,t}$. 
  The only connected $1$-regular graph is $P_2$, therefore $\abs{V(M_i)} = 2$ and 
  $M_i = P_2$ for all 
  $i\in\set{1,\dots,t}$. Hence, we have
  $G = \complement{P_2} 
  \join \dots \join \complement{P_2}$, which means nothing else than 
  $G = K_n - \mu$ for a
  perfect matching~$\mu$ of $K_n$.
  
  In the second case~\ref{it:general:exceptional} the metamour graph of $G$ 
  consists of two connected 
  components, so the metamour graph consists of $n=4$ vertices with two edges 
  that form a perfect matching. 
  It is 
  easy to see that $G = P_4$ or
  $G = C_4 = K_4 - \mu$ for some perfect matching $\mu$ of $K_4$ are the only two 
  possibilities in this case.
  
  As a result, we obtain in any case $G=P_4$ or $G = K_n - \mu$ for a
  perfect matching~$\mu$ of $K_n$, which is the desired result.
\end{proof}

To finish this section we prove the three corollaries of Theorem~\ref{thm:1mr}.

\begin{proof}[Proof of Corollary~\ref{cor:1mmr-complement}]
    Due to the characterization of $1$-metamour-regular graphs of 
    Theorem~\ref{thm:1mr}, we know that a connected graph with $n \ge 5$ 
    vertices is $1$-metamour-regular if and only if it is equal to $K_n - \mu$ 
    for some perfect matching $\mu$ of $K_n$. This is the case if and only if 
    the complement is the graph induced by~$\mu$. Furthermore, a graph is induced by a 
    perfect matching if and only if it is $1$-regular.
    To summarize, a connected graph with $n \ge 5$ vertices is 
    $1$-metamour-regular if and 
    only if its complement is a $1$-regular graph. 
    In this case the complement also equals the metamour graph, which implies 
    the desired result.
\end{proof}

\begin{proof}[Proof of Corollary~\ref{cor:1mmr-counting}]
  We use the characterization provided by Theorem~\ref{thm:1mr}.
  So let us consider $1$-metamour-regular graphs.
  Such a graph has at least~$n\ge4$ vertices, and $n$ is even.
  Every perfect matching~$\mu$ of~$K_n$ results in the same unlabeled graph~$K_n-\mu$;
  this brings to account $1$. For $n=4$, there is additionally
  the graph~$P_4$.
  In total, this gives the claimed numbers.
\end{proof}

\begin{proof}[Proof of Corollary~\ref{cor:1mmr-bijection:nnc}]
  Let $G$ be an unlabeled graph with~$n$ pairs of vertices that each are
  metamours. We first construct a pair~$(\lambda_1+\dots+\lambda_t,s)$,
  where $\lambda_1+\dots+\lambda_t$ is
  a partition of~$n$ with $\lambda_i\ge2$ for all $i\in\set{1,\dots,t}$
  and $s$ is a non-negative integer bounded by~$r_\lambda$ which is defined to be
  the number of $i\in\set{1,\dots,t}$ with $\lambda_i=2$.
  
  Let $\set{G_1, \dots, G_t} = \CC{G}$, set
  $\lambda_i=\abs{V(G_i)}/2$ for all $i\in\set{1,\dots,t}$, and let us assume
  that $\lambda_1 \ge \dots \ge \lambda_t$. Then
  $n=\lambda_1+\dots+\lambda_t$, so this is a partition of~$n$.
  As there is no graph~$G_i$ with only $1$~metamour-pair, 
  $\lambda_i\ge2$ for all $i\in\set{1,\dots,t}$.
  We define~$s$ to be the number of $i\in\set{1,\dots,t}$ with $G_i=P_4$.
  We clearly have $s \le r_\lambda$.

  Conversely, let a pair~$(\lambda_1+\dots+\lambda_t,s)$ as above be given.
  For every $i\in\set{1,\dots,t}$ with $\lambda_i\ge3$ there is
  exactly one choice for a $1$-metamour-regular graph~$G_i$ with $2\lambda_i$ 
  vertices by
  Theorem~\ref{thm:1mr}. Now consider parts~$2$ of $\lambda_1+\dots+\lambda_t$.
  We choose any (the graphs are unlabeled)
  $s$ indices and set $G_i=P_4$.
  Then we set~$G_i=C_4$ for the remaining $r_\lambda-s$ indices.
  The graph~$G=G_1 \cup \dots \cup G_t$ is then fully determined.
  Thus, we have a found a bijective correspondence.

  We still need to related our partition of~$n$ to the partition of~$n+2$
  of Corollary~\ref{cor:1mmr-bijection:nnc}. A partition of~$n+2$ is
  either $n+2=(n+2)$, $n\ge1$, in which case no additional part~$2$ appears,
  or $n+2=\lambda_1+\dots+\lambda_t+2$ for a partition
  $n=\lambda_1+\dots+\lambda_t$. Here one additional part~$2$ appears.
  Therefore, every pair~$(\lambda_1+\dots+\lambda_t,s)$ from above
  maps bijectively to a partition $\lambda_1+\dots+\lambda_t+2$ of $n+2$
  together with a marker of one of the $r_\lambda+1$ parts~$2$ in this partition
  that is uniquely determined by $s$ (by some fixed rule that is not needed
  to be specified explicitly).
  This completes the proof of Corollary~\ref{cor:1mmr-bijection:nnc}.
\end{proof}

\section{Proofs regarding graphs with maximum 
metamour-degree~\texorpdfstring{$1$}{1}}
\label{sec:proofs:max-1mm}

Next we prove the results of Section~\ref{sec:results:at-most-1m} on graphs 
with 
maximum metamour-degree~$1$. 
We start with the proof of the characterization of these graphs.

\begin{proof}[Proof of Theorem~\ref{thm:max1mm}]
  It is easy to see that in the graphs $K_1$ and $K_2$ no vertex has any
  metamour, so the condition that each vertex has at most one metamour
  is satisfied.
  Furthermore, 
  by Propositions~\ref{proposition:P4-1mmr}
  and~\ref{proposition:Kn-minus-M-le1mmr},
  every vertex has indeed at most one metamour in the remaining specified 
  graphs. Therefore, one
  direction of the equivalence is proven, and we can focus on the other
  direction.

  So, let $G$ be a graph in which every vertex has at most one metamour.
  Due to Theorems~\ref{thm:0mr} and~\ref{thm:1mr}, it is enough to restrict
  ourselves to graphs~$G$, where  at
    least one vertex of $G$ has no metamour and at least one vertex of $G$ has 
    exactly one metamour.
  We will show that~$n\ge3$ and that $G = K_n - \mu$ for some matching 
    $\mu$ that is not perfect and contains at least one edge.

    Let $V_0 \subseteq V(G)$ and $V_1 \subseteq V(G)$ be the set of vertices of 
    $G$ that have no and exactly one metamour, respectively, and
    let $v \in V_0$.
    Due to Lemma~\ref{lem:noMetamour_adjacentToAll}, every vertex in $V_0$,
    in particular $v$, is adjacent to all other vertices.
    Furthermore, by
    Proposition~\ref{pro:1mr_evenNrVertices_perfectMatching}, 
    the vertices in $V_1$ induce a matching~$\mu$
    in both the metamour graph and the complement of $G$.
    This matching~$\mu$ contains at least one edge because $V_1$ is not empty,
    and $\mu$ is not perfect because $V_0$ is not empty.
    Furthermore, 
    this implies that $V_1$ contains at least two vertices and in
    total that $n \ge 3$. 

    Let $w$ and $w'$ be two vertices in $V_1$ that are not metamours. Since 
    $v$ is a common neighbor of both $w$ and $w'$, this 
    implies that $\edge{w}{w'} \in E(G)$. Hence, all possible edges except those 
    in $\mu$ are present in $G$ and therefore $G = K_n - \mu$.
\end{proof}

Next we prove the two corollaries of Theorem~\ref{thm:max1mm}.

\begin{proof}[Proof of Corollary~\ref{cor:max1mm-complement}]
    Due to Theorem~\ref{thm:max1mm}, in a connected graph $G$ with $n\ge 5$ 
    vertices every vertex has at most one metamour if and only if $G = K_n - 
    \mu$ for some matching $\mu$ of $K_n$. 
    This is the case if and only if 
    the complement is the graph induced by~$\mu$. Furthermore, a graph is 
    induced by a 
    matching if and only if it has maximum degree $1$.
    To summarize, a connected graph with $n \ge 5$ vertices has maximum 
    metamour-degree~$1$ if and 
    only if its complement is a graph with maximum degree $1$.
    In this case the complement also equals the metamour graph, which implies 
    the desired result.
\end{proof}

\begin{proof}[Proof of Corollary~\ref{cor:max1mm-counting}]
  We use the characterization provided by Theorem~\ref{thm:max1mm}.
  So let us consider graphs with maximum metamour-degree~$1$.
  For $n\in\set{1,2}$, we only have $K_1$ and $K_2$, so $m_{\le1}(n)=1$
  in both cases.

  So let $n\ge3$.
  Every perfect matching~$\mu$ of~$K_n$ having the same number of edges
  results in the same graph~$K_n-\mu$. A matching can contain at most
  $\floor{n/2}$ edges and each choice in $\set{0,\dots,\floor{n/2}}$
  for the number of edges is possible.
  This brings to account $\floor{n/2}+1$. For $n=4$, there is additionally
  the graph~$P_4$.
  In total, this gives the claimed numbers.
\end{proof}

\section{Proofs regarding \texorpdfstring{$2$}{2}-metamour-regular graphs}
\label{sec:proofs:2-mmr}

This section is devoted to the proofs concerning $2$-metamour-regular graphs 
from Section~\ref{sec:results:2mr}. It is a long way to obtain the final 
characterization of $2$-metamour-regular graphs of Theorem~\ref{thm:2mriff}, so 
we have outsourced the key parts of the proof into several lemmas and 
propositions. 

For the proofs of Lemma~\ref{lem:2mmr:contain-Cn}, Lemma~\ref{lem:2mmr:deg-two-options}, 
Lemma~\ref{lem:2mmr:deg-2-and-nm3} and 
Proposition~\ref{pro:2mmr:metamour-graph-not-cn} we provide many figures. 
Every proof consists of a series of steps, and in  each 
of the steps vertices and 
edges of a graph are analyzed: It is determined whether edges
are present or not and which vertices are metamours of each other. The figures 
of the actual situations show subgraphs of the
graph (and additional assumptions) in the following way: Between
two vertices there is either
an edge \tikz{
  \node [vertex] (v0) at (0,0) {};
  \node [vertex] (v1) at (1,0) {};
  \draw [edge] (v0) -- (v1); }
or a non-edge \tikz{
  \node [vertex] (v0) at (0,0) {};
  \node [vertex] (v1) at (1,0) {};
  \draw [nonedge] (v0) -- (v1); }
or nothing \tikz{
  \node [vertex] (v0) at (0,0) {};
  \node [vertex] (v1) at (1,0) {}; }
drawn. 
If nothing is
drawn, then it is not (yet) clear whether the edge is present or not.
A metamour relation \tikz{
  \node [vertex] (v0) at (0,0) {};
  \node [vertex] (v1) at (1,0) {};
  \draw [nonedge] (v0) -- (v1);
  \draw [mmedge] (v0) -- (v1); }
might be indicated at a non-edge. 

Note that we frequently use the particular graphs defined by
Figures~\ref{fig:graphs-6}, \ref{fig:graphs-7} and \ref{fig:graphs-8}.

\subsection{Graphs with connected metamour graph}

The proof of the characterization of $2$-metamour-regular graphs in Theorem~\ref{thm:2mriff}
is split into two main parts, which represent whether
the metamour graph of $G$ is connected or not in order to apply the 
corresponding case of Theorem~\ref{thm:kmmr:metamour-graph-not-connected}. 

If the metamour graph of a graph with $n$ vertices is connected, then 
according to 
Observation~\ref{obs:2mr_metamour-graph-consists-of-cycles} 
the metamour graph equals $C_n$.
Here we make a further distinction between graphs that do and that do not 
contain a cycle of length $n$ as a subgraph.
First, we characterize all $2$-metamour-regular graphs whose metamour graph is 
connected and that do not contain a cycle of length~$n$. 

\begin{lemma}\label{lem:2mmr:contain-Cn}
  Let $G$ be a connected $2$-metamour-regular graph with $n$ vertices
  \begin{itemize}
  \item whose metamour graph equals the~$C_n$,
  \item that is not a tree, and
  \item that does not contain a cycle of length~$n$.
\end{itemize}
  Then
  \begin{equation*}
    G \in \set{
          \specialH{6}{a},
          \specialH{6}{b},
          \specialH{7}{a}}.
  \end{equation*}
\end{lemma}

\resetbeh
\begin{proof}
  As $G$ is not a tree,
  let $\gamma = (v_1, v_2, \dots, v_r, v_1)$, $v_i\in V(G)$ for
  $i \in \set{1,\dots,r}$, be a longest cycle in $G$.
  In all the figures accompanying the proof, the longest cycle is marked
  by \tikz[baseline]{\draw [cycleedge] (0,0) arc (150:30:0.5);}.  
  By assumption, we have $r < n$.
  For proving the lemma, we have to show that $G \in \set{
      \specialH{6}{a},
      \specialH{6}{b},
      \specialH{7}{a}}$. 
 
  As a cycle has length at least~$3$, we have $r\ge3$ for the length
  of the cycle~$\gamma$. As we also have $n>r$, we may assume~$n \ge 4$.
  
  We start by showing the following claims.

  \begin{beh}\label{beh:2mmr:vertex-adjacent-consec}
    A vertex~$u$ in~$G$ that is not in the cycle~$\gamma$
    is adjacent to at most one vertex in~$\gamma$.

    If $u$ is adjacent to a vertex~$v$ in~$\gamma$, then $u$ is a metamour
    of each neighbor of~$v$ in~$\gamma$.
  \end{beh}
  \begin{proofbeh}{beh:2mmr:vertex-adjacent-consec}
    Let $u \in V(G)$ be a vertex not in~$\gamma$.
    We assume that $u$ is adjacent to
    $v_1$ (without loss of generality by renumbering)
    and some $v_j$ in the cycle~$\gamma$.
    We first show that $v_1$ and $v_j$ are not
    two consecutive vertices in~$\gamma$.
    So let us assume that they are, i.e., $j=2$
    (see Figure~\ref{fig:2mmr:vertex-adjacent-consec}(a)) or
    $j=r$ which works analogously).
    Then $(v_1, u, v_2, \dots, v_r, v_1)$ would be a longer cycle
    which is a contradiction to $\gamma$ being a longest cycle.
    \begin{figure}
      \centering
      \begin{tikzpicture}
        \begin{scope}[shift={(0,0)}]
          \node at (1,-2.5) {(a) consecutive vertices};
          
          \node [vertex, label=left:{$v_1$}] (v0) at (20:1) {};
          \node [vertex, label=left:{$v_2$}] (v1) at (-20:1) {};
          \node [vertex, label=left:{$v_r$}] (vm1) at (60:1) {};
          \node [vertex, label=left:{$v_3$}] (v2) at (-60:1) {};
          \node [vertex, label=right:{$u$}] (u) at (2,0) {};
          
          \draw [edge] (vm1) -- (v0) -- (v1) -- (v2);
          \draw [edge] (v0) -- (u) -- (v1);

          \begin{pgfonlayer}{back}
          \draw [cycleedge] (vm1) -- (v0) -- (u) -- (v1) -- (v2);
          \draw [cycleedge] (vm1) arc (60:300:1);
          \end{pgfonlayer}
        \end{scope}
        \begin{scope}[shift={(5,0)}]
          \node at (0,-2.5) {(b) $r=4$};
          
          \node [vertex, label=above:{$v_1$}] (v1) at (1,0) {};
          \node [vertex, label=below:{$v_2$}] (v2) at (0,-1) {};
          \node [vertex, label=left:{$v_3$}] (v3) at (-1,0) {};
          \node [vertex, label=above:{$v_4$}] (v4) at (0,1) {};
          \node [vertex, label=right:{$u$}] (w) at (2,0) {};

          \draw [edge] (w) -- (v1) -- (v2) -- (v3) -- (v4) -- (v1);
          \draw [nonedge] (w) to [bend right] (v4);
          \draw [nonedge] (w) to [bend left] (v2);
          \draw [edge] (v2) -- (v4);

          \draw [edge] plot[smooth, tension=1.5]
          coordinates {(w) (0,-1.75) (v3)};

          \begin{pgfonlayer}{back}
          \draw [cycleedge] (w) -- (v1) -- (v2) -- (v4) -- (v3);
          \draw [cycleedge] plot[smooth, tension=1.5]
          coordinates {(w) (0,-1.75) (v3)};
          \end{pgfonlayer}
        \end{scope}
      \end{tikzpicture}
      
      \caption{Subgraphs of the situations
        in the proof of~\behref{beh:2mmr:vertex-adjacent-consec}}
      \label{fig:2mmr:vertex-adjacent-consec}
    \end{figure}
    Hence, $v_1$ and $v_j$ are not consecutive vertices
    in~$\gamma$. Then $r\ge4$ as there need to be at least one vertex
    between~$v_1$ and~$v_j$ on the cycle on each side.

    If $r=4$, then $v_j=v_3$ and we are in the situation shown in
    Figure~\ref{fig:2mmr:vertex-adjacent-consec}(b). There,
    $(u,v_1,v_2,v_4,v_3,u)$ is a $5$-cycle which contradicts that
    the longest cycle is of length~$4$.
    Therefore, $r=4$ cannot hold.

    If $r>4$, then $u$ is a metamour of $v_{2}$, $v_{r}$, $v_{j-1}$ and 
    $v_{j+1}$, 
    because it has a common neighbor ($v_1$ or $v_j$) with these vertices and 
    is not 
    adjacent to them. 
    At
    least one of~$v_{j-1}$ and~$v_{j+1}$ is different from~$v_2$ and $v_r$,
    so $\abs{\set{v_2, v_r, v_{j-1}, v_{j+1}}} \geq 3$. This contradicts
    the $2$-metamour-regularity of~$G$.

    Therefore, we have shown that $u$ is adjacent to at most one
    vertex in~$\gamma$. Now suppose $u$ is adjacent to a vertex~$v$ in~$\gamma$.
    Then $u$ is not adjacent to any neighbor of~$v$ in~$\gamma$ and
    therefore a metamour of every such neighbor.
  \end{proofbeh}

  \begin{beh}\label{beh:2mmr:exist-w}
    There exists a vertex $w$ in $G$ but not in~$\gamma$ that is adjacent to
    (without loss of generality)~$v_1$, but not to any other~$v_j$,
    $j\neq1$, in $\gamma$.
  \end{beh}
  \begin{proofbeh}{beh:2mmr:exist-w}
    As $r<n$, there exists a vertex~$w'$ not in the cycle~$\gamma$. The
    graph~$G$ is connected, so there is a path from a vertex
    of~$\gamma$ to~$w'$. Therefore, there is also a vertex~$w$ not in $\gamma$
    which is adjacent to a vertex~$v_i$.
    By renumbering, we can assume without loss of generality that $i=1$.

    As the vertex~$w$ is adjacent to $v_1$, $w$ is not adjacent to
    any other~$v_j$ by~\behref{beh:2mmr:vertex-adjacent-consec}.
  \end{proofbeh}

  \begin{figure}
    \centering
    \begin{tikzpicture}
      \begin{scope}[shift={(0,0)}]
        \node [vertex, label=above:{$v_1$}] (v1) at (0:1) {};
        \node [vertex, label=below:{$v_2$}] (v2) at (-60:1) {};
        \node [vertex, label=above:{$v_{r}$}] (vr) at (60:1) {};
        \node [vertex, label=right:{$w$}] (w) at (2,0) {};
        
        \draw [edge] (w) -- (v1) -- (v2);
        \draw [edge] (v1) -- (vr);
        \draw [nonedge] (w) to [bend right] (vr);
        \draw [mmedge] (w) to [bend right] (vr);
        \draw [nonedge] (w) to [bend left] (v2);
        \draw [mmedge] (w) to [bend left] (v2);

        \begin{pgfonlayer}{back}
        \draw [cycleedge] (v2) -- (v1) -- (vr) arc (60:300:1);
        \end{pgfonlayer}
      \end{scope}
    \end{tikzpicture}

    \caption{Subgraph of the situation between~\behref{beh:2mmr:exist-w}
      and~\behref{beh:2mmr:edge-v2-vr}}
    \label{fig:2mmr:exist-w}
  \end{figure}
  At this point, we assume to have a vertex~$w$ as
  in~\behref{beh:2mmr:exist-w}; the situation is shown in
  Figure~\ref{fig:2mmr:exist-w}.

  \begin{beh}\label{beh:2mmr:edge-v2-vr}
    The graph~$G$ contains the edge $\edge{v_2}{v_r}$.
  \end{beh}
  \begin{proofbeh}{beh:2mmr:edge-v2-vr}
    \begin{figure}
      \centering
      \begin{tikzpicture}
        \begin{scope}[shift={(0,0)}]
          \node [vertex, label=above:{$v_1$}] (v1) at (0:1) {};
          \node [vertex, label=below:{$v_2$}] (v2) at (-60:1) {};
          \node [vertex, label=above:{$v_{r}$}] (vr) at (60:1) {};
          \node [vertex, label=right:{$w$}] (w) at (2,0) {};

          \draw [edge] (w) -- (v1) -- (v2);
          \draw [edge] (v1) -- (vr);
          \draw [nonedge] (w) to [bend right] (vr);
          \draw [mmedge] (w) to [bend right] (vr);
          \draw [nonedge] (w) to [bend left] (v2);
          \draw [mmedge] (w) to [bend left] (v2);
          \draw [nonedge] (v2) -- (vr);
          \draw [mmedge] (v2) -- (vr);

          \begin{pgfonlayer}{back}
          \draw [cycleedge] (v2) -- (v1) -- (vr) arc (60:300:1);
          \end{pgfonlayer}
        \end{scope}
      \end{tikzpicture}

      \caption{Subgraph of the situation
        in the proof of~\behref{beh:2mmr:edge-v2-vr}}
      \label{fig:2mmr:edge-v2-vr}
    \end{figure}
    Assume that there is no edge between
    $v_2$ and $v_r$; see Figure~\ref{fig:2mmr:edge-v2-vr}.
    Then $v_2$ and $v_r$ are metamours of each other,
    and consequently $(w, v_2, v_r, w)$ forms a $3$-cycle
    in the metamour graph of~$G$. This contradicts that
    the metamour graph of~$G$ is $C_n$ and~$n\ge4$.
  \end{proofbeh}

  \begin{figure}
    \centering
    \begin{tikzpicture}
      \begin{scope}[shift={(0,0)}]
        \node [vertex, label=above:{$v_1$}] (v1) at (0:1) {};
        \node [vertex, label=below:{$v_2$}] (v2) at (-60:1) {};
        \node [vertex, label=above:{$v_{r}$}] (vr) at (60:1) {};
        \node [vertex, label=right:{$w$}] (w) at (2,0) {};
        
        \draw [edge] (w) -- (v1) -- (v2);
        \draw [edge] (v1) -- (vr);
        \draw [edge] (v2) -- (vr);
        \draw [nonedge] (w) to [bend right] (vr);
        \draw [mmedge] (w) to [bend right] (vr);
        \draw [nonedge] (w) to [bend left] (v2);
        \draw [mmedge] (w) to [bend left] (v2);

        \begin{pgfonlayer}{back}
        \draw [cycleedge] (v2) -- (v1) -- (vr) arc (60:300:1);
        \end{pgfonlayer}
      \end{scope}
    \end{tikzpicture}

    \caption{Subgraph of the situation between~\behref{beh:2mmr:edge-v2-vr}
      and~\behref{beh:2mmr:r-is-3}}
    \label{fig:2mmr:after-edge-v2-vr}
  \end{figure}
  At this point, we have the situation shown in
  Figure~\ref{fig:2mmr:after-edge-v2-vr}.
  In the next steps we will rule out possible values of $r$.

  \begin{beh}\label{beh:2mmr:r-is-3}
    If $r=3$, then $G=\specialH{6}{a}$.
  \end{beh}
  
  \begin{proofbeh}{beh:2mmr:r-is-3}
    Our initial situation is shown in Figure~\ref{fig:2mmr:r-is-3}(a).
    
    Suppose there is an additional vertex~$v_1'$ of~$G$ adjacent to~$v_1$;
    see Figure~\ref{fig:2mmr:r-is-3}(b).
    Then by~\behref{beh:2mmr:vertex-adjacent-consec}, $v_1'$ has
    metamours~$v_2$ and $v_3$. Therefore, $(w,v_2,v_1',v_3,w)$
    is a $4$-cycle in the metamour graph of~$G$. As this cycle
    does not cover~$v_1$, we have a contradiction
    to the metamour graph being the single cycle~$C_n$ for $n>r=3$.
    Therefore, there is no additional vertex adjacent to~$v_1$.

    \begin{figure}
      \centering
      \begin{tikzpicture}
        \begin{scope}[shift={(0,0)}]
          \node at (0,-2) {(a) initial situation};

          \node [vertex, label=above:{$v_1$}] (v1) at (0:1) {};
          \node [vertex, label=below:{$v_2$}] (v2) at (-120:1) {};
          \node [vertex, label=above:{$v_3$}] (v3) at (120:1) {};
          \node [vertex, label=right:{$w$}] (w) at (2,0) {};
          
          \draw [edge] (w) -- (v1) -- (v2) -- (v3) -- (v1);
          \draw [nonedge] (w) to [bend right] (v3);
          \draw [mmedge] (w) to [bend right] (v3);
          \draw [nonedge] (w) to [bend left] (v2);
          \draw [mmedge] (w) to [bend left] (v2);

          \begin{pgfonlayer}{back}
          \draw [cycleedge] (v2) -- (v1) -- (v3) -- (v2);
          \end{pgfonlayer}
        \end{scope}
        \begin{scope}[shift={(5,0)}]
          \node at (0,-2) {(b) additional~$v_1'$};

          \node [vertex, label=above:{$v_1$}] (v1) at (0:1) {};
          \node [vertex, label=above:{$v_1'$}] (v1p) at (0,0) {};
          \node [vertex, label=below:{$v_2$}] (v2) at (-120:1) {};
          \node [vertex, label=above:{$v_3$}] (v3) at (120:1) {};
          \node [vertex, label=right:{$w$}] (w) at (2,0) {};
          
          \draw [edge] (w) -- (v1) -- (v2) -- (v3) -- (v1);
          \draw [nonedge] (w) to [bend right] (v3);
          \draw [mmedge] (w) to [bend right] (v3);
          \draw [nonedge] (w) to [bend left] (v2);
          \draw [mmedge] (w) to [bend left] (v2);
          
          \draw [edge] (v1) -- (v1p);
          \draw [nonedge] (v2) -- (v1p) -- (v3);
          \draw [mmedge] (v2) -- (v1p) -- (v3);

          \begin{pgfonlayer}{back}
          \draw [cycleedge] (v2) -- (v1) -- (v3) -- (v2);
          \end{pgfonlayer}
        \end{scope}
        \begin{scope}[shift={(0,-5)}]
          \node at (0,-3) {(c) additional~$v_3'$};
          
          \node [vertex, label=above:{$v_1$}] (v1) at (0:1) {};
          \node [vertex, label=below:{$v_2$}] (v2) at (-120:1) {};
          \node [vertex, label=above:{$v_3$}] (v3) at (120:1) {};
          \node [vertex, label=right:{$w$}] (w) at (2,0) {};
          \node [vertex, label=above:{$v_3'$}] (v3p) at (120:2) {};
          
          \draw [edge] (w) -- (v1) -- (v2) -- (v3) -- (v1);
          \draw [nonedge] (w) to [bend right] (v3);
          \draw [mmedge] (w) to [bend right] (v3);
          \draw [nonedge] (w) to [bend left] (v2);
          \draw [mmedge] (w) to [bend left] (v2);

          \draw [edge] (v3) -- (v3p);
          \draw [nonedge] (v3p) to [bend right] (v2);
          \draw [mmedge] (v3p) to [bend right] (v2);
          \draw [nonedge] (v3p) to [bend left] (v1);
          \draw [mmedge] (v3p) to [bend left] (v1);

          \begin{pgfonlayer}{back}
          \draw [cycleedge] (v2) -- (v1) -- (v3) -- (v2);
          \end{pgfonlayer}
        \end{scope}
        \begin{scope}[shift={(5,-5)}]
          \node at (0,-3) {(d) additional~$v_2'$ and $v_3'$};

          \node [vertex, label=above:{$v_1$}] (v1) at (0:1) {};
          \node [vertex, label=below:{$v_2$}] (v2) at (-120:1) {};
          \node [vertex, label=above:{$v_3$}] (v3) at (120:1) {};
          \node [vertex, label=right:{$w$}] (w) at (2,0) {};
          \node [vertex, label=below:{$v_2'$}] (v2p) at (-120:2) {};
          \node [vertex, label=above:{$v_3'$}] (v3p) at (120:2) {};
          
          \draw [edge] (w) -- (v1) -- (v2) -- (v3) -- (v1);
          \draw [nonedge] (w) to [bend right] (v3);
          \draw [mmedge] (w) to [bend right] (v3);
          \draw [nonedge] (w) to [bend left] (v2);
          \draw [mmedge] (w) to [bend left] (v2);

          \draw [edge] (v3) -- (v3p);
          \draw [nonedge] (v3p) to [bend right] (v2);
          \draw [mmedge] (v3p) to [bend right] (v2);
          \draw [nonedge] (v3p) to [bend left] (v1);
          \draw [mmedge] (v3p) to [bend left] (v1);
          \draw [edge] (v2) -- (v2p);
          \draw [nonedge] (v2p) to [bend right] (v1);
          \draw [mmedge] (v2p) to [bend right] (v1);
          \draw [nonedge] (v2p) to [bend left] (v3);
          \draw [mmedge] (v2p) to [bend left] (v3);

          \begin{pgfonlayer}{back}
          \draw [cycleedge] (v2) -- (v1) -- (v3) -- (v2);
          \end{pgfonlayer}
        \end{scope}
      \end{tikzpicture}
      
      \caption{Subgraph of the situation
        in the proof of~\behref{beh:2mmr:r-is-3}}
      \label{fig:2mmr:r-is-3}
    \end{figure}

    At this point, we know that $w$ is a metamour of both~$v_2$ and~$v_3$;
    see again Figure~\ref{fig:2mmr:r-is-3}(a).
    We now look for the second metamour of~$v_2$ and~$v_3$, respectively.
    As we ruled out an additional vertex adjacent to~$v_1$,
    there need to be an additional vertex adjacent to~$v_2$ or to~$v_3$.
    
    Without loss of generality (by symmetry),
    suppose there is an additional vertex~$v_3'$ of~$G$ adjacent to~$v_3$;
    see Figure~\ref{fig:2mmr:r-is-3}(c).
    Then by~\behref{beh:2mmr:vertex-adjacent-consec}, $v_3'$ has
    metamours~$v_1$ and~$v_2$. Therefore, these two vertices
    are the two metamours of~$v_3'$.
    There cannot be an additional vertex~$v_3''$ of~$G$ adjacent to~$v_3$, 
    because due to the same arguments as for $v_3'$ this vertex would be a
    metamour of $v_2$, hence $v_2$ would have three metamours, and this 
    contradicts 
    the $2$-metamour-regularity of~$G$.
    
    Suppose there is no additional vertex adjacent to~$v_2$. Then, in
    order to close the metamour cycle containing
    $(v_1, v_3', v_2, w, v_3)$, there needs to be a path from $v_3'$ to
    $w$. This implies the existence of a cycle longer than~$r=3$,
    therefore cannot be. So there is an additional vertex adjacent
    to~$v_2'$; the situation is shown in Figure~\ref{fig:2mmr:r-is-3}(d).

    By the same argument as above, $v_1$ and $v_3$ are the two
    metamours of~$v_2'$. Therefore, $(w,v_2,v_3',v_1,v_2',v_3,w)$
    is a $6$-cycle in the metamour graph of~$G$ and $n=6$.
    This is the graph $G=\specialH{6}{a}$. We can only add additional edges
    between the vertices~$w$, $v_2'$ and $v_3'$, but this
    would lead to a cycle of length larger than~$3$. So there are no
    other edges present.
    There cannot be any additional vertex because this vertex would need to
    be in a different cycle in the metamour graph, contradicting
    that the metamour graph is the~$C_n$.
  \end{proofbeh}

  As a consequence of \behref{beh:2mmr:r-is-3}, the proof is finished for 
  $r=3$, 
  because then $G = \specialH{6}{a}$.
  What is left to consider is the case~$r\ge4$
  and consequently $n\ge5$. 
  The situation is again as in
  Figure~\ref{fig:2mmr:after-edge-v2-vr}.

  \begin{beh}\label{beh:2mmr:v1-not-adj-vi}
    The only vertices of~$\gamma$ that are adjacent to $v_1$ are $v_2$ and 
    $v_r$.
    In particular, $v_1$ is metamour of~$v_3$ and of~$v_{r-1}$.
  \end{beh}
  \begin{proofbeh}{beh:2mmr:v1-not-adj-vi}
    Suppose there is a vertex~$v_i$ with $i\in\set{3,\dots,r-1}$ adjacent 
    to~$v_1$. 
    Then, as $w$ is not adjacent to~$v_i$
    by~\behref{beh:2mmr:vertex-adjacent-consec} or~\behref{beh:2mmr:exist-w}, $v_i$ is a third metamour
    of~$w$. This contradicts the $2$-metamour-regularity of~$G$.

    As $v_1$ has distance~$2$ on the cycle~$\gamma$ to both~$v_3$ 
    and~$v_{r-1}$, and is not
    adjacent to these vertices, the vertices~$v_3$ and~$v_{r-1}$ are
    metamours of~$v_1$.
  \end{proofbeh}

  \begin{figure}
    \centering
    \begin{tikzpicture}
      \begin{scope}[shift={(0,0)}]
        \node [vertex, label=75:{$v_1$}] (v1) at (0:1) {};
        \node [vertex, label=below:{$v_2$}] (v2) at (-50:1) {};
        \node [vertex, label=below:{$v_3$}] (v3) at (-100:1) {};
        \node [vertex, label=above:{$v_{r-1}$}] (vrm1) at (100:1) {};
        \node [vertex, label=above:{$v_{r}$}] (vr) at (50:1) {};
        \node [vertex, label=right:{$w$}] (w) at (2,0) {};
        
        \draw [edge] (w) -- (v1) -- (v2) -- (v3);
        \draw [edge] (v1) -- (vr) -- (vrm1);
        \draw [edge] (v2) -- (vr);
        \draw [nonedge] (w) to [bend right] (vr);
        \draw [mmedge] (w) to [bend right] (vr);
        \draw [nonedge] (w) to [bend left] (v2);
        \draw [mmedge] (w) to [bend left] (v2);
        \draw [nonedge] (w) to [bend left=60] (v3);
        \draw [nonedge] (w) to [bend right=60] (vrm1);
        \draw [nonedge] (v1) -- (vrm1);
        \draw [mmedge] (v1) -- (vrm1);
        \draw [nonedge] (v1) -- (v3);
        \draw [mmedge] (v1) -- (v3);

        \begin{pgfonlayer}{back}
        \draw [cycleedge] (v3) -- (v2) -- (v1) -- (vr) -- (vrm1) arc 
        (100:260:1);
        \end{pgfonlayer}
      \end{scope}
    \end{tikzpicture}

    \caption{Subgraph of the situation between~\behref{beh:2mmr:v1-not-adj-vi}
      and~\behref{beh:2mmr:r-is-4}}
    \label{fig:2mmr:v1-not-adj-vi}
  \end{figure}
  At this point, we have the situation shown in
  Figure~\ref{fig:2mmr:v1-not-adj-vi}. Note, that it is still possible that 
  $v_3 = v_{r-1}$.

  \begin{beh}\label{beh:2mmr:r-is-4}
    We cannot have $r=4$.
  \end{beh}
  \begin{proofbeh}{beh:2mmr:r-is-4}
    As $r=4$, we have $v_3=v_{r-1}$. This situation is shown in
    Figure~\ref{fig:2mmr:r-is-4}.
    \begin{figure}
      \centering
      \begin{tikzpicture}
        \begin{scope}[shift={(0,0)}]
          \node [vertex, label=above:{$v_1$}] (v1) at (0:1) {};
          \node [vertex, label=below:{$v_2$}] (v2) at (-90:1) {};
          \node [vertex, label=below:{$v_3$}] (v3) at (-180:1) {};
          \node [vertex, label=above:{$v_4$}] (v4) at (90:1) {};
          \node [vertex, label=right:{$w$}] (w) at (2,0) {};
          
          \draw [edge] (w) -- (v1) -- (v2) -- (v3) -- (v4) -- (v1);
          \draw [edge] (v2) -- (v4);
          \draw [nonedge] (w) to [bend right] (v4);
          \draw [mmedge] (w) to [bend right] (v4);
          \draw [nonedge] (w) to [bend left] (v2);
          \draw [mmedge] (w) to [bend left] (v2);
          \draw [nonedge] (v1) -- (v3);
          \draw [mmedge] (v1) -- (v3);
          \draw [nonedge] plot[smooth, tension=1.5]
          coordinates {(w) (0,-1.75) (v3)};

          \begin{pgfonlayer}{back}
          \draw [cycleedge] (v1) -- (v2) -- (v3) -- (v4) -- (v1);
          \end{pgfonlayer}
        \end{scope}
      \end{tikzpicture}

      \caption{Subgraph of the situation
        in the proof of~\behref{beh:2mmr:r-is-4}}
      \label{fig:2mmr:r-is-4}
    \end{figure}

    Suppose there is an additional vertex~$v_3'$ of~$G$ adjacent to~$v_3$.
    Then by~\behref{beh:2mmr:vertex-adjacent-consec}, $v_3'$ has
    metamours~$v_2$ and~$v_4$. Therefore, $(w,v_2,v_3',v_4,w)$
    is a $4$-cycle in the metamour graph of~$G$. This is a contradiction
    to the metamour graph being~$C_n$ and $n\ge6$,
    so there is no additional vertex adjacent to~$v_3$.
    This implies that we cannot have a vertex at distance~$1$ from~$v_3$
    other than~$v_2$ and~$v_4$.

    Now suppose there is an additional vertex~$v_2'$ of~$G$ adjacent to~$v_2$.
    Again by~\behref{beh:2mmr:vertex-adjacent-consec}, $v_2'$ has
    metamours~$v_1$ and~$v_3$. Therefore, $(v_2',v_1,v_3,v_2')$
    is a $3$-cycle in the metamour graph of~$G$. This is again a contradiction
    to the metamour graph being~$C_n$ and $n\ge6$,
    so there is no additional vertex adjacent to~$v_2$ either.
    Likewise, by symmetry, there is also no additional vertex adjacent to~$v_4$.

    As $v_2$ and $v_4$ are the only neighbors of~$v_3$, we cannot have
    a vertex at distance~$2$ from~$v_3$ other than~$v_1$. This means that
    there is no second metamour of~$v_3$ which 
    contradicts the $2$-metamour-regularity of~$G$.
  \end{proofbeh}

  At this point, we can assume that~$r\ge5$ as the case~$r=4$ was excluded
  by \behref{beh:2mmr:r-is-4},
  and consequently also $n\ge6$.
  The situation is still as in Figure~\ref{fig:2mmr:v1-not-adj-vi}.
  
  \begin{beh}\label{beh:2mmr:metamours-of-v2}
    We have~$r\le6$. Specifically,
    either $r=5$,
    or $r=6$ and there is an edge $\edge{v_2}{v_5}$ in~$G$.
    In the second case, the two metamours of the vertex~$v_2$
    are $w$ and $v_4$.
  \end{beh}
  \begin{proofbeh}{beh:2mmr:metamours-of-v2}
    As $r\ge5$, the two metamours of~$v_1$ are on the cycle~$\gamma$,
    namely the distinct vertices~$v_3$ and~$v_{r-1}$;
    see Figure~\ref{fig:2mmr:metamours-of-v2}(a).

    We now consider
    the neighbors of~$v_2$. Suppose $v_2$ is adjacent to some~$v_i$ with
    $i \not\in \set{1,3,r-1,r}$. As the vertex~$v_1$ is not connected
    to~$v_{i}$ by \behref{beh:2mmr:v1-not-adj-vi},
    the vertex~$v_{i}$ is a metamour of~$v_1$
    different from $v_3$ and $v_{r-1}$.
    This contradicts the $2$-metamour-regularity of~$G$.
    Furthermore, $v_2$ is adjacent to~$v_1$, $v_3$ and $v_r$.
    This implies that the neighborhood of $v_2$ on~$\gamma$ is determined up to 
    $v_{r-1}$. We will now distinguish whether $v_{r-1}$ is or is not 
    in this neighborhood.

    \begin{figure}
      \centering
      \begin{tikzpicture}
        \begin{scope}[shift={(0,0)}]
          \node at (0.5,-2.25) {(a) initial situation};

          \node [vertex, label=75:{$v_1$}] (v1) at (0:1) {};
          \node [vertex, label=below:{$v_2$}] (v2) at (-50:1) {};
          \node [vertex, label=below:{$v_3$}] (v3) at (-100:1) {};
          \node [vertex, label=below:{$v_4$}] (v4) at (-150:1) {};
          \node [vertex, label=above:{$v_{r-2}$}] (vrm2) at (150:1) {};
          \node [vertex, label=above:{$v_{r-1}$}] (vrm1) at (100:1) {};
          \node [vertex, label=above:{$v_{r}$}] (vr) at (50:1) {};
          \node [vertex, label=right:{$w$}] (w) at (2,0) {};
          
          \draw [edge] (w) -- (v1) -- (v2) -- (v3) -- (v4);
          \draw [edge] (v1) -- (vr) -- (vrm1) -- (vrm2);
          \draw [edge] (v2) -- (vr);
          \draw [nonedge] (w) to [bend right] (vr);
          \draw [mmedge] (w) to [bend right] (vr);
          \draw [nonedge] (w) to [bend left] (v2);
          \draw [mmedge] (w) to [bend left] (v2);
          \draw [nonedge] (v1) -- (v4);
          \draw [nonedge] (v1) -- (vrm2);
          \draw [nonedge] (v1) -- (vrm1);
          \draw [mmedge] (v1) -- (vrm1);
          \draw [nonedge] (v1) -- (v3);
          \draw [mmedge] (v1) -- (v3);
          
          \draw [nonedge] (w) to [bend right=60] (vrm1);
          \draw [nonedge] (w) to [bend left=60] (v3);
          \draw [nonedge] (w) to [bend left=45] (0,-1.75);
          \draw [nonedge] (0,-1.75) to [bend left=30] (v4);
          \draw [nonedge] (w) to [bend right=45] (0,1.75);
          \draw [nonedge] (0,1.75) to [bend right=30] (vrm2);

          \begin{pgfonlayer}{back}
          \draw [cycleedge] (v4) -- (v3) -- (v2) -- (v1)
                             -- (vr) -- (vrm1) -- (vrm2) arc (150:210:1);
          \end{pgfonlayer}
        \end{scope}
        \begin{scope}[shift={(5,0)}]
        \node at (0.5,-2.25) {(b) with $\edge{v_2}{v_{r-1}} \not\in E(G)$};
        
        \node [vertex, label=75:{$v_1$}] (v1) at (0:1) {};
        \node [vertex, label=below:{$v_2$}] (v2) at (-50:1) {};
        \node [vertex, label=below:{$v_3$}] (v3) at (-100:1) {};
        \node [vertex, label=above:{$v_4$}] (v4) at (100:1) {};
        \node [vertex, label=above:{$v_5$}] (v5) at (50:1) {};
        \node [vertex, label=right:{$w$}] (w) at (2,0) {};
        
        \draw [edge] (w) -- (v1) -- (v2) -- (v3) -- (v4);
        \draw [edge] (v1) -- (v5) -- (v4);
        \draw [edge] (v2) -- (v5);
        \draw [nonedge] (w) to [bend right] (v5);
        \draw [mmedge] (w) to [bend right] (v5);
        \draw [nonedge] (w) to [bend left] (v2);
        \draw [mmedge] (w) to [bend left] (v2);
        \draw [nonedge] (v1) -- (v4);
        \draw [mmedge] (v1) -- (v4);
        \draw [nonedge] (v1) -- (v3);
        \draw [mmedge] (v1) -- (v3);
        \draw [nonedge] (v2) -- (v4);
        \draw [mmedge] (v2) -- (v4);
        
        \draw [nonedge] (w) to [bend right=60] (v4);
        \draw [nonedge] (w) to [bend left=60] (v3);

        \begin{pgfonlayer}{back}
        \draw [cycleedge] (v3) -- (v2) -- (v1) -- (v5) -- (v4) -- (v3);
        \end{pgfonlayer}
        \end{scope}        
        \begin{scope}[shift={(10,0)}]
          \node at (0.5,-2.25) {(c) with $\edge{v_2}{v_{r-1}} \in E(G)$};

          \node [vertex, label=75:{$v_1$}] (v1) at (0:1) {};
          \node [vertex, label=below:{$v_2$}] (v2) at (-50:1) {};
          \node [vertex, label=below:{$v_3$}] (v3) at (-100:1) {};
          \node [vertex, label=below:{$v_4$}] (v4) at (-180:1) {};
          \node [vertex, label=above:{$v_5$}] (v5) at (100:1) {};
          \node [vertex, label=above:{$v_6$}] (v6) at (50:1) {};
          \node [vertex, label=right:{$w$}] (w) at (2,0) {};
          
          \draw [edge] (w) -- (v1) -- (v2) -- (v3) -- (v4);
          \draw [edge] (v1) -- (v6) -- (v5) -- (v4);
          \draw [edge] (v2) -- (v6);
          \draw [nonedge] (w) to [bend right] (v6);
          \draw [mmedge] (w) to [bend right] (v6);
          \draw [nonedge] (w) to [bend left] (v2);
          \draw [mmedge] (w) to [bend left] (v2);
          \draw [nonedge] (v1) -- (v4);
          \draw [nonedge] (v1) -- (v5);
          \draw [mmedge] (v1) -- (v5);
          \draw [nonedge] (v1) -- (v3);
          \draw [mmedge] (v1) -- (v3);
          \draw [nonedge] (v2) -- (v4);
          \draw [mmedge] (v2) -- (v4);
          \draw [edge] (v2) -- (v5);

          \draw [nonedge] (w) to [bend right=60] (v5);
          \draw [nonedge] (w) to [bend left=60] (v3);
          \draw [nonedge] (w) to [bend left=45] (0,-1.75);
          \draw [nonedge] (0,-1.75) to [bend left=30] (v4);

          \begin{pgfonlayer}{back}
          \draw [cycleedge] (v3) -- (v2) -- (v1) -- (v6) -- (v5) -- (v4) -- (v3);
          \end{pgfonlayer}
        \end{scope}
      \end{tikzpicture}

      \caption{Subgraphs of the situations
        in the proof of~\behref{beh:2mmr:metamours-of-v2}}
      \label{fig:2mmr:metamours-of-v2}
    \end{figure}
    Suppose
    $\edge{v_2}{v_{r-1}} \not\in E(G)$. If $v_{r-1} \ne v_4$, then
    $\edge{v_2}{v_4} \not\in E(G)$ because of  what is shown in
    the previous paragraph. But then,
    the metamours of~$v_2$ would be $w$, $v_{r-1}$ and $v_4$.
    This contradicts the $2$-metamour-regularity of $G$
    and implies that $v_{r-1}=v_4$ and $r=5$; see
    Figure~\ref{fig:2mmr:metamours-of-v2}(b).
       
    Suppose
    $\edge{v_2}{v_{r-1}} \in E(G)$. We again distinguish between two cases. 
    If $r \ge 6$, then $w$, $v_4$ and $v_{r-2}$ are metamours of~$v_2$. In
    this case, the $2$-metamour-regularity of~$G$ implies that
    $v_{r-2}=v_4$ and therefore $r=6$; 
    see Figure~\ref{fig:2mmr:metamours-of-v2}(c). 
    If
    $r < 6$, then by the findings so far, we must have $r=5$, and therefore we 
    are also done 
    in this case.
  \end{proofbeh}

  By~\behref{beh:2mmr:metamours-of-v2} we are left with the two
  cases~$r=5$ and $r=6$. One possible situation for $r=5$ and the situation for 
  $r=6$ are shown in
  Figure~\ref{fig:2mmr:metamours-of-v2}(b) and (c), respectively, and we will 
  deal with
  these two situations now.
  
  \begin{beh}\label{beh:2mmr:r-is-5}
    If $r=5$, then $G\in\set{\specialH{6}{b}, \specialH{7}{a}}$.
  \end{beh}

  \begin{proofbeh}{beh:2mmr:r-is-5}
    The full situation for  $r=5$ is shown in Figure~\ref{fig:2mmr:r-is-5}(a).
    \begin{figure}
      \centering
      \begin{tikzpicture}
        \begin{scope}[shift={(0,0)}]
          \node at (0.5,-2) {(a) initial situation};
          
          \node [vertex, label=75:{$v_1$}] (v1) at (0:1) {};
          \node [vertex, label=below:{$v_2$}] (v2) at (-50:1) {};
          \node [vertex, label=below:{$v_3$}] (v3) at (-100:1) {};
          \node [vertex, label=above:{$v_4$}] (v4) at (100:1) {};
          \node [vertex, label=above:{$v_5$}] (v5) at (50:1) {};
          \node [vertex, label=right:{$w$}] (w) at (2,0) {};
          
          \draw [edge] (w) -- (v1) -- (v2) -- (v3) -- (v4);
          \draw [edge] (v1) -- (v5) -- (v4);
          \draw [edge] (v2) -- (v5);
          \draw [nonedge] (w) to [bend right] (v5);
          \draw [mmedge] (w) to [bend right] (v5);
          \draw [nonedge] (w) to [bend left] (v2);
          \draw [mmedge] (w) to [bend left] (v2);
          \draw [nonedge] (v1) -- (v4);
          \draw [mmedge] (v1) -- (v4);
          \draw [nonedge] (v1) -- (v3);
          \draw [mmedge] (v1) -- (v3);
          
          \draw [nonedge] (w) to [bend right=60] (v4);
          \draw [nonedge] (w) to [bend left=60] (v3);

          \begin{pgfonlayer}{back}
          \draw [cycleedge] (v3) -- (v2) -- (v1) -- (v5) -- (v4) -- (v3);
          \end{pgfonlayer}
        \end{scope}
        \begin{scope}[shift={(5,0)}]
          \node at (0.5,-2) {(b) with none of two edges};

          \node [vertex, label=75:{$v_1$}] (v1) at (0:1) {};
          \node [vertex, label=below:{$v_2$}] (v2) at (-50:1) {};
          \node [vertex, label=below:{$v_3$}] (v3) at (-100:1) {};
          \node [vertex, label=above:{$v_4$}] (v4) at (100:1) {};
          \node [vertex, label=above:{$v_5$}] (v5) at (50:1) {};
          \node [vertex, label=right:{$w$}] (w) at (2,0) {};
          
          \draw [edge] (w) -- (v1) -- (v2) -- (v3) -- (v4);
          \draw [edge] (v1) -- (v5) -- (v4);
          \draw [edge] (v2) -- (v5);
          \draw [nonedge] (w) to [bend right] (v5);
          \draw [mmedge] (w) to [bend right] (v5);
          \draw [nonedge] (w) to [bend left] (v2);
          \draw [mmedge] (w) to [bend left] (v2);
          \draw [nonedge] (v1) -- (v4);
          \draw [mmedge] (v1) -- (v4);
          \draw [nonedge] (v1) -- (v3);
          \draw [mmedge] (v1) -- (v3);
          \draw [nonedge] (v2) -- (v4);
          \draw [mmedge] (v2) -- (v4);
          
          \draw [nonedge] (w) to [bend right=60] (v4);
          \draw [nonedge] (w) to [bend left=60] (v3);

          \draw [nonedge] (v3) -- (v5);
          \draw [mmedge] (v3) -- (v5);

          \begin{pgfonlayer}{back}
          \draw [cycleedge] (v3) -- (v2) -- (v1) -- (v5) -- (v4) -- (v3);
          \end{pgfonlayer}
        \end{scope}
        \begin{scope}[shift={(10,0)}]
          \node at (0.5,-2) {(c) with one of two edges};
          
          \node [vertex, label=75:{$v_1$}] (v1) at (0:1) {};
          \node [vertex, label=below:{$v_2$}] (v2) at (-50:1) {};
          \node [vertex, label=below:{$v_3$}] (v3) at (-100:1) {};
          \node [vertex, label=above:{$v_4$}] (v4) at (100:1) {};
          \node [vertex, label=above:{$v_5$}] (v5) at (50:1) {};
          \node [vertex, label=right:{$w$}] (w) at (2,0) {};
          \node [vertex, label=above:{$v_4'$}] (v4p) at (115:2) {};
          
          \draw [edge] (w) -- (v1) -- (v2) -- (v3) -- (v4);
          \draw [edge] (v1) -- (v5) -- (v4);
          \draw [edge] (v2) -- (v5);
          \draw [nonedge] (w) to [bend right] (v5);
          \draw [mmedge] (w) to [bend right] (v5);
          \draw [nonedge] (w) to [bend left] (v2);
          \draw [mmedge] (w) to [bend left] (v2);
          \draw [nonedge] (v1) -- (v4);
          \draw [mmedge] (v1) -- (v4);
          \draw [nonedge] (v1) -- (v3);
          \draw [mmedge] (v1) -- (v3);
          \draw [nonedge] (v2) -- (v4);
          \draw [mmedge] (v2) -- (v4);
          
          \draw [nonedge] (w) to [bend right=60] (v4);
          \draw [nonedge] (w) to [bend left=60] (v3);

          \draw [edge] (v3) -- (v5);
          \draw [edge] (v4) -- (v4p);
          \draw [nonedge] (v3) to [bend left=15] (v4p);
          \draw [mmedge] (v3) to [bend left=15] (v4p);
          \draw [nonedge] (v5) to [bend right] (v4p);
          \draw [mmedge] (v5) to [bend right] (v4p);  

          \begin{pgfonlayer}{back}
          \draw [cycleedge] (v3) -- (v2) -- (v1) -- (v5) -- (v4) -- (v3);
          \end{pgfonlayer}
        \end{scope}
      \end{tikzpicture}

      \caption{Subgraphs of the situations
        in the proof of~\behref{beh:2mmr:r-is-5}}
      \label{fig:2mmr:r-is-5}
    \end{figure}
    
    Clearly the situation is symmetric in the potential edges $\edge{v_2}{v_4}$ 
    and $\edge{v_3}{v_5}$, so we have to consider the three cases that both, 
    one and none of these two edges are present.

    First let us assume that neither $\edge{v_2}{v_4}$ nor
    $\edge{v_3}{v_5}$ is an edge; see 
    Figure~\ref{fig:2mmr:r-is-5}(b).
    Then $v_2$ and $v_4$ as well as $v_3$ and $v_5$ are metamours, 
    so we have the $6$-cycle
    $(w,v_2,v_4,v_1,v_3,v_5,w)$ in the metamour graph of~$G$. This
    is the graph $G=\specialH{6}{b}$.
    There cannot be any additional vertex because this vertex would need to
    be in a different cycle in the metamour graph contradicting
    that the metamour graph is the~$C_n$.
    There also cannot be any additional edges because all edges and non-edges 
    are already determined.

    Next let us assume that there is exactly one of the edges $\edge{v_2}{v_4}$ 
    and $\edge{v_3}{v_5}$ present in~$G$, without loss of generality let 
    $\edge{v_3}{v_5} \in E(G)$; see 
    Figure~\ref{fig:2mmr:r-is-5}(c).
    At this point we know that $v_3$ and $v_1$ as well as
    $v_5$ and $w$ are metamours, and we are looking for the second metamours of
    $v_3$ and $v_5$. As the
    vertices~$w$, $v_1$ and $v_4$ already have two metamours each,
    there need to be additional vertices for these metamours.

    Statement~\behref{beh:2mmr:vertex-adjacent-consec} implies
    that there is no additional vertex of~$G$ adjacent to~$v_5$
    as~$v_1$ has already the two metamours~$v_3$ and~$v_4$.
    Likewise, by symmetry,
    there is no additional vertex adjacent to~$v_2$. Moreover,
    by the same argument, there is
    also no additional vertex adjacent to~$v_3$ as~$v_4$
    has the two metamours~$v_1$ and~$v_2$.

    Therefore, there need to be an additional vertex~$v_4'$ adjacent
    to~$v_4$. By~\behref{beh:2mmr:vertex-adjacent-consec}, $v_4'$ has
    metamours~$v_3$ and~$v_5$. This gives the $7$-cycle
    $(w,v_2,v_4,v_1,v_3,v_4',v_5,w)$ in the metamour graph of~$G$ and
    the graph $G=\specialH{7}{a}$.
    There cannot be any additional vertex because this vertex would need to
    be in a different cycle in the metamour graph contradicting
    that the metamour graph is the~$C_n$.
    There also cannot be any additional edges because all edges and non-edges 
    are already determined.
    
    At last, let us consider the case that both of the edges $\edge{v_2}{v_4}$ 
    and $\edge{v_3}{v_5}$ are present in~$G$.
    We already know that $w$ is a metamour of $v_2$ and are now
    searching for the second metamour of $v_2$. 
    There does not exist a vertex $v_1'$ adjacent to $v_1$ in 
    $G$ but 
    not in $\gamma$, because this would induce a $C_4$ in the 
    metamour graph by the same arguments as in the proof 
    of~\behref{beh:2mmr:r-is-3}.
    Furthermore, there cannot be a vertex $v_2'$ in $G$ but not in 
    $\gamma$ that is adjacent to $v_2$, due to the fact that 
    this vertex would be a third metamour of $v_1$, a contradiction.
    By symmetry, there is no vertex of $G$ without~$\gamma$ adjacent to~$v_5$.
    If there would be a vertex~$v_3'$ in $G$ but not in $\gamma$ which is 
    adjacent to $v_3$, then due 
    to~\behref{beh:2mmr:vertex-adjacent-consec}, 
    this vertex would have $v_2$, $v_4$ and~$v_5$ as a metamour, a 
    contradiction to the $2$-metamour-regularity of $G$.
    Again by symmetry, there is no vertex of $G$ without~$\gamma$ adjacent to~$v_4$.
    Therefore, $v_2$ cannot have a second metamour in $G$ and this case 
    cannot happen.
  \end{proofbeh}
    
   Statement~\behref{beh:2mmr:r-is-5} finalizes the proof for $r=5$.    
   Hence, $r=6$ is the only remaining value for 
    $r$ we have to consider.

  \begin{beh}\label{beh:2mmr:r-is-6}
    We cannot have $r=6$.
  \end{beh}
  \begin{proofbeh}{beh:2mmr:r-is-6}
    As $r=6$, there is an edge $\edge{v_2}{v_5}$ in $G$
    by~\behref{beh:2mmr:metamours-of-v2}. The initial situation is shown in
    Figure~\ref{fig:2mmr:r-is-6}(a).
    \begin{figure}
      \centering
      \begin{tikzpicture}
        \begin{scope}[shift={(0,0)}]
          \node at (0.5,-2.5) {(a) initial situation};
          
          \node [vertex, label=75:{$v_1$}] (v1) at (0:1) {};
          \node [vertex, label=below:{$v_2$}] (v2) at (-50:1) {};
          \node [vertex, label=below:{$v_3$}] (v3) at (-100:1) {};
          \node [vertex, label=below:{$v_4$}] (v4) at (-180:1) {};
          \node [vertex, label=above:{$v_5$}] (v5) at (100:1) {};
          \node [vertex, label=above:{$v_6$}] (v6) at (50:1) {};
          \node [vertex, label=right:{$w$}] (w) at (2,0) {};
          
          \draw [edge] (w) -- (v1) -- (v2) -- (v3) -- (v4);
          \draw [edge] (v1) -- (v6) -- (v5) -- (v4);
          \draw [edge] (v2) -- (v6);
          \draw [nonedge] (w) to [bend right] (v6);
          \draw [mmedge] (w) to [bend right] (v6);
          \draw [nonedge] (w) to [bend left] (v2);
          \draw [mmedge] (w) to [bend left] (v2);
          \draw [nonedge] (v1) -- (v4);
          \draw [nonedge] (v1) -- (v5);
          \draw [mmedge] (v1) -- (v5);
          \draw [nonedge] (v1) -- (v3);
          \draw [mmedge] (v1) -- (v3);
          \draw [nonedge] (v2) -- (v4);
          \draw [mmedge] (v2) -- (v4);
          \draw [edge] (v2) -- (v5);

          \draw [nonedge] (w) to [bend right=60] (v5);
          \draw [nonedge] (w) to [bend left=60] (v3);
          \draw [nonedge] (w) to [bend left=45] (0,-1.75);
          \draw [nonedge] (0,-1.75) to [bend left=30] (v4);

          \begin{pgfonlayer}{back}
          \draw [cycleedge] (v3) -- (v2) -- (v1) -- (v6) -- (v5) -- (v4) -- (v3);
          \end{pgfonlayer}
        \end{scope}
        \node at (3.5,0) {$\Longrightarrow$};
        \begin{scope}[shift={(5.5,0)}]
          \node at (0.5,-2.5) {(b) more edges found};
          
          \node [vertex, label=75:{$v_1$}] (v1) at (0:1) {};
          \node [vertex, label=below:{$v_2$}] (v2) at (-50:1) {};
          \node [vertex, label=below:{$v_3$}] (v3) at (-100:1) {};
          \node [vertex, label=below:{$v_4$}] (v4) at (-180:1) {};
          \node [vertex, label=above:{$v_5$}] (v5) at (100:1) {};
          \node [vertex, label=above:{$v_6$}] (v6) at (50:1) {};
          \node [vertex, label=right:{$w$}] (w) at (2,0) {};
          
          \draw [edge] (w) -- (v1) -- (v2) -- (v3) -- (v4);
          \draw [edge] (v1) -- (v6) -- (v5) -- (v4);
          \draw [edge] (v2) -- (v6);
          \draw [nonedge] (w) to [bend right] (v6);
          \draw [mmedge] (w) to [bend right] (v6);
          \draw [nonedge] (w) to [bend left] (v2);
          \draw [mmedge] (w) to [bend left] (v2);
          \draw [nonedge] (v1) -- (v4); 
          \draw [nonedge] (v1) -- (v5);
          \draw [mmedge] (v1) -- (v5);
          \draw [nonedge] (v1) -- (v3);
          \draw [mmedge] (v1) -- (v3);
          \draw [nonedge] (v2) -- (v4);
          \draw [mmedge] (v2) -- (v4);
          \draw [edge] (v2) -- (v5);

          \draw [edge] (v4) -- (v6);
          \draw [edge] (v3) -- (v5);

          \draw [nonedge] (w) to [bend right=60] (v5);
          \draw [nonedge] (w) to [bend left=60] (v3);
          \draw [nonedge] (w) to [bend left=45] (0,-1.75);
          \draw [nonedge] (0,-1.75) to [bend left=30] (v4);

          \begin{pgfonlayer}{back}
          \draw [cycleedge] (v3) -- (v2) -- (v1) -- (v6) -- (v5) -- (v4) -- (v3);
          \end{pgfonlayer}
        \end{scope}
      \end{tikzpicture}

      \caption{Subgraphs of the situations
        in the proof of~\behref{beh:2mmr:r-is-6}}
      \label{fig:2mmr:r-is-6}
    \end{figure}

    Suppose $v_3$ and $v_5$ are not adjacent. Then
    $(v_1,v_3,v_5,v_1)$ is a $3$-cycle in the metamour graph
    of~$G$. This contradicts that the metamour graph is~$C_n$
    and $n>r=6$,
    so we can assume $\edge{v_3}{v_5} \in E(G)$.
    Likewise, suppose that $v_4$ and $v_6$ are not adjacent. Then
    $(w,v_2,v_4,v_6,w)$ is a $4$-cycle in the metamour graph
    of~$G$. This contradicts that the metamour graph is~$C_n$
    and $n>r=6$,
    so we can assume $\edge{v_4}{v_6} \in E(G)$.
    The current situation is shown in Figure~\ref{fig:2mmr:r-is-6}(b).

    Statement~\behref{beh:2mmr:vertex-adjacent-consec} implies
    that there is no additional vertex of~$G$ adjacent to~$v_2$
    as~$v_1$ has already the two metamours~$v_3$ and~$v_5$.
    By symmetry, there
    is also no additional vertex adjacent to~$v_6$.
    By the same argumentation as above, there is no additional vertex adjacent
    to~$v_1$ as well as to~$v_3$ because of vertex~$v_2$
    and its metamours. Moreover, we slightly vary the argumentation
    to show that there cannot be an additional vertex adjacent
    to~$v_5$. Suppose there is an additional vertex~$v_5'$ of~$G$ adjacent
    to~$v_5$. Then, $v_5'$ is not adjacent to~$v_2$ as we have shown above,
    so $v_5'$ is as well a metamour of~$v_2$. This
    contradicts the $2$-metamour-regularity of~$G$ again.

    The vertex~$v_4$ has $v_2$ as metamour. We are now searching for
    its second metamour. It cannot be~$w$ or~$v_1$ as these vertices
    have already two other metamours each. It cannot be any of $v_3$,
    $v_5$ or $v_6$ either as all of them are adjacent
    to~$v_4$. Moreover, the second metamour of~$v_4$ cannot be
    adjacent to~$v_3$, $v_5$ or $v_6$, as we above ruled additional
    neighbors to these vertices out. Therefore, there has to be
    an additional vertex~$v_4'$ adjacent to~$v_4$.
    By~\behref{beh:2mmr:vertex-adjacent-consec}, this vertex~$v_4'$ has
    metamours~$v_3$ and~$v_5$. This results
    in the $4$-cycle $(v_1,v_3,v_4',v_5,v_1)$ in the metamour graph of~$G$
    and contradicts our assumption that this graph is~$C_n$ and $n>r=6$.
  \end{proofbeh}
  
  We have now completed the proof of Lemma~\ref{lem:2mmr:contain-Cn} as
  in all cases we were able to 
  show that $G \in \set{
      \specialH{6}{a},
      \specialH{6}{b},
      \specialH{7}{a}}$
    holds. 
\end{proof}

After characterizing all $2$-metamour-regular graphs whose metamour graph is 
connected and that do not contain a cycle of length~$n$, we can now focus on 
$2$-metamour-regular graphs whose metamour graph is connected and that contain  
a cycle of length~$n$. Here, we make a further distinction depending on the 
degree of the vertices and begin with the following
lemma.

\begin{lemma}\label{lem:2mmr:deg-two-options}
    Let $G$ be a connected $2$-metamour-regular graph with $n$ vertices
    \begin{itemize}
    \item whose metamour graph equals the~$C_n$,
    \item that contains a cycle of length~$n$, and
    \item that has a vertex of degree larger than $2$ and smaller than $n-3$.
    \end{itemize}
    Then
    \begin{equation*}
      G = \specialH{7}{b}.
    \end{equation*}
\end{lemma}

\resetbeh
\begin{proof}
    Let $\gamma$ be a cycle of length~$n$ in $G$.
    First, we introduce some notation. Let $v$ be a vertex of $G$, and let $u$ 
    and $u'$ be the two metamours of $v$. We explore the vertices on the
    cycle~$\gamma$ starting with~$v$: The set of vertices on both sides
    of~$v$ strictly before $u$ and $u'$ are called the \emph{fellows} of $v$.
    The remaining set of vertices strictly between $u$ and $u'$ on 
    $\gamma$ is called the \emph{opponents} of $v$;
    see Figure~\ref{fig:2mmr:fellows-opponents}.
    In other words for each vertex $v$ of~$G$ the set of vertices of $G$ can 
    be 
    partitioned into $v$, its fellows, its 
    metamours and its opponents.

    \begin{figure}
      \centering
      \begin{tikzpicture}
        \begin{scope}[shift={(0,0)}]
          \draw [edge] (0,0) circle (\cradius);
          
          \cvertex{90}{}{v1}
          \cvertex{210}{}{vp}
          \cvertex{330}{}{vq}

          \cvertices{105}{195}{}{}{6}
          \cvertices{225}{315}{}{}{6}
          \cvertices{-15}{75}{}{}{6}

          \draw [nonedge] (v1) to [out=90+180, in=30] (vp);
          \draw [mmedge] (v1) to [out=90+180, in=30] (vp);
          \draw [nonedge] (v1) to [out=90+180, in=150] (vq);
          \draw [mmedge] (v1) to [out=90+180, in=150] (vq);

          \node [label={[anchor=south]90:{$v$}}] at (90:\cradius) {};
          
          \node [label={[anchor=south, rotate=150-90]150:{fellows}}] at (150:\cradius) {};
          \node [label={[anchor=south, rotate=30-90]30:{fellows}}] at (30:\cradius) {};
          \node [label={[anchor=north]270:{opponents}}] at (270:\cradius) {};

          \node [label={[anchor=north, rotate=210-270, 
          align=center]210:{meta-\\[-1ex] mour}}] at (vp) {};
          \node [label={[anchor=north, rotate=-30-270, align=center]-30:{meta-\\[-1ex] mour}}] at (-30:\cradius) {};
        \end{scope}
      \end{tikzpicture}
      \caption{Fellows and opponents of a vertex~$v$
        in the proof of Lemma~\ref{lem:2mmr:deg-two-options}}
      \label{fig:2mmr:fellows-opponents}
    \end{figure}
    
    We start with the following claims.
    
    \begin{beh}\label{beh:2mmr:adj-to-fellows}
        Every vertex of $G$ is adjacent to each of its fellows.
    \end{beh}
    \begin{proofbeh}{beh:2mmr:adj-to-fellows}
      Let $v_1$ be a vertex of~$G$ and $\gamma=(v_1,\dots,v_n,v_1)$.
      Suppose~$v_p$ is the vertex with smallest index~$p$ that is not
      adjacent to~$v_1$. We have to show that~$v_p$ is a
      metamour of~$v_1$. The index~$p$ exists because~$v_1$
      is not adjacent to its metamours.
      Moreover, this index satisfies~$p>2$
      as $v_2$ is adjacent to $v_1$ because they are consecutive
      vertices on~$\gamma$. Thus, $v_1$ and $v_p$ have $v_{p-1}$ as common
      neighbor and are therefore metamours.

      By symmetry, the vertex~$v_q$ with largest index~$q$ that is not
      adjacent to~$v_1$, is also a metamour of~$v_1$.
      Note that as $v_1$ has exactly two metamours,
      $v_p$ and $v_q$ are these metamours, so $v_1$ is adjacent to each of 
      its fellows.
    \end{proofbeh}

    \begin{beh}\label{beh:2mmr:adj-to-all-or-non-opponents}
        Every vertex of $G$ is either adjacent to each of its 
        opponents, or not adjacent to any of its opponents.
    \end{beh}
    \begin{proofbeh}{beh:2mmr:adj-to-all-or-non-opponents}
        It is enough to show that if a vertex~$v_1$ of~$G$
        is adjacent to at least one opponent of~$v_1$, then 
        it is adjacent to every opponent of~$v_1$.
        Let $\gamma=(v_1, \dots, v_n,v_1)$, and  
        let $W$ be a subset of the opponents of~$v_1$ that consists of
        consecutive vertices of~$\gamma$, say from $v_{i}$ to $v_{j}$ for some
        $i \le j$,
        such that each of these vertices is adjacent to~$v_1$, and $W$ is 
        maximal (with respect to inclusion)
        with this property.
        Note that the set~$W$ is not empty because of our assumption.

        Clearly none of the vertices in~$W$ is a metamour of~$v_1$. However
        $v_{i-1}$ and $v_{j+1}$ are metamours of~$v_1$ because of their common
        neighbors~$v_{i}$ and $v_{j}$, and the maximality of~$W$.
        Therefore, as $v_1$ has exactly two metamours, $W$ equals the set of 
        opponents of~$v_1$ which was to show.
    \end{proofbeh}    
    
    Now we are ready to start with the heart of the proof of 
    Lemma~\ref{lem:2mmr:deg-two-options}.
    Suppose $v_1$ is a vertex of~$G$ with $2 < \deg(v_1) < n-3$.
    In order to complete the proof
    we have to show that $G=\specialH{7}{b}$.
    
    Let $\gamma=(v_1, \dots, v_n,v_1)$ be a cycle of length~$n$, and let 
    $v_p$ and $v_q$ be the metamours of $v_1$ with $p < q$.
    In the following claims we will derive several properties of $G$. 
    
    \begin{beh}\label{beh:2mmr:situation-v1}
        The vertex $v_1$ is adjacent to its fellows $v_{2}$, \dots, $v_{p-1}$, 
        $v_{q+1}$, \dots, $v_n$ and not adjacent to any
        metamour or opponent $v_{p}$, \dots, $v_{q}$.
        Furthermore, $p + 1 < q$ holds, i.e., there exists at least
        one opponent of~$v_1$.
    \end{beh}
    \begin{proofbeh}{beh:2mmr:situation-v1}
        Clearly $v_1$ is not adjacent to its metamours $v_p$ and $v_q$. 
        Furthermore, $v_1$ 
        is adjacent to all its fellows 
        $v_{2}$, \dots, $v_{p-1}$, $v_{q+1}$, \dots, $v_n$ 
        by~\behref{beh:2mmr:adj-to-fellows}.
        This together with $\deg(v_1) < n-3$ implies that $v_1$ has an opponent 
        to which it is not 
        adjacent, so $p + 1 < q$. 
        Then by~\behref{beh:2mmr:adj-to-all-or-non-opponents}, $v_1$ is not 
        adjacent 
        to any of its opponents. 
    \end{proofbeh}    

    \begin{figure}
      \centering
      \begin{tikzpicture}
        \begin{scope}[shift={(0,0)}]
          \draw [edge] (0,0) circle (\cradius);
          
          \def\wvone{90}
          \def\wvtwo{105}
          \def\wvpmone{195}
          \def\wvp{210}
          \def\wvppone{225}
          \def\wvqmone{315}
          \def\wvq{330}
          \def\wvqpone{345}
          \def\wvqptwo{360}
          \def\wvn{435}

          \cvertex{\wvone}{$v_1$}{vone}
          \cvertices{\wvtwo}{\wvpmone}{$v_2$}{$v_{p-1}$}{6}
          \node [svertex] (vpmone) at (\wvpmone:\cradius) {};
          \cvertex{\wvp}{$v_p$}{vp}
          \cvertices{\wvppone}{\wvqmone}{$v_{p+1}$}{$v_{q-1}$}{6}
          \node [svertex] (vppone) at (\wvppone:\cradius) {};
          \node [svertex] (vqmone) at (\wvqmone:\cradius) {};
          \cvertex{\wvq}{$v_q$}{vq}
          \cvertex{\wvqpone}{$v_{q+1}$}{vqpone}
          \cvertices{\wvqptwo}{\wvn}{$v_{q+2}$}{$v_n$}{5}
          \node [svertex] (vqptwo) at (\wvqptwo:\cradius) {};

          \cedgevv{nonedge}{vone}{vp}{}
          \cedgevv{mmedge}{vone}{vp}{}
          \cedgevv{nonedge}{vone}{vq}{}
          \cedgevv{mmedge}{vone}{vq}{}
          
          \cedgevw{nonedge}{vone}{\wvppone/2+\wvqmone/2}{}
          \cedgevw{edge}{vone}{\wvtwo/2+\wvpmone/2}{}
          \cedgevw{edge}{vone}{\wvqptwo/2+\wvn/2}{}
          \cedgevv{edge}{vone}{vqpone}{}
        \end{scope}
      \end{tikzpicture}
      \caption{Subgraph of the situation between~\behref{beh:2mmr:situation-v1}
        and~\behref{beh:2mmr:no-edges-fellows-opponents}}
      \label{fig:2mmr:situation-v1}
    \end{figure}

    Now $\deg(v_1) > 2$ together with~\behref{beh:2mmr:situation-v1} imply 
    that $v_1$ 
    has at least one fellow 
    different from $v_2$ and $v_n$.    
    Without loss of generality (by renumbering the vertices in the opposite 
    direction of rotation along $\gamma$) assume that $v_{q + 1}$ is a fellow  
    of $v_1$ different from $v_n$, so in other words we assume $q + 1 < n$.
    The situation is shown in Figure~\ref{fig:2mmr:situation-v1}.

    We will now prove several claims about edges, non-edges and metamours of 
    $G$.
    
    \begin{beh}\label{beh:2mmr:no-edges-fellows-opponents}
      No opponent $v_{p+1}$, \dots, $v_{q-1}$ is adjacent to any
      fellow $v_{2}$, \dots, $v_{p-1}$, $v_{q+1}$, \dots, $v_n$.
    \end{beh}
    \begin{proofbeh}{beh:2mmr:no-edges-fellows-opponents}
        Assume that $v_j$ is adjacent to $v_i$ for some 
        $j \in \set{p+1, \dots, q-1}$ and some $i \in \set{2, \dots, p-1} 
        \cup 
        \set{q+1, \dots, n}$. Then $v_j$ and $v_1$ have the common neighbor 
        $v_i$ because of~\behref{beh:2mmr:situation-v1}.
        Furthermore, $v_j$ and $v_1$ are 
        not adjacent by~\behref{beh:2mmr:situation-v1}, so $v_j$ and $v_1$ are 
        metamours.
        This is a contradiction to $v_p$ and $v_q$ being the only metamours 
        of $v_1$, therefore our assumption was wrong.
    \end{proofbeh}
    
    \begin{figure}
      \centering
      \begin{tikzpicture}
        \begin{scope}[shift={(0,0)}]
          \node at (0,-2.5) {(a) between~\behref{beh:2mmr:no-edges-fellows-opponents}
            and \behref{beh:2mmr:neighbors-of-metamours-are-metamours}};

          \draw [edge] (0,0) circle (\cradius);

          \def\wvone{90}
          \def\wvtwo{105}
          \def\wvpmone{195}
          \def\wvp{210}
          \def\wvppone{225}
          \def\wvqmone{315}
          \def\wvq{330}
          \def\wvqpone{345}
          \def\wvqptwo{360}
          \def\wvn{435}

          \cvertex{\wvone}{$v_1$}{vone}
          \cvertices{\wvtwo}{\wvpmone}{$v_2$}{$v_{p-1}$}{6}
          \node [svertex] (vpmone) at (\wvpmone:\cradius) {};
          \cvertex{\wvp}{$v_p$}{vp}
          \cvertices{\wvppone}{\wvqmone}{$v_{p+1}$}{$v_{q-1}$}{6}
          \node [svertex] (vppone) at (\wvppone:\cradius) {};
          \node [svertex] (vqmone) at (\wvqmone:\cradius) {};
          \cvertex{\wvq}{$v_q$}{vq}
          \cvertex{\wvqpone}{$v_{q+1}$}{vqpone}
          \cvertices{\wvqptwo}{\wvn}{$v_{q+2}$}{$v_n$}{5}
          \node [svertex] (vqptwo) at (\wvqptwo:\cradius) {};

          \cedgevv{nonedge}{vone}{vp}{}
          \cedgevv{mmedge}{vone}{vp}{}
          \cedgevv{nonedge}{vone}{vq}{}
          \cedgevv{mmedge}{vone}{vq}{}
          
          \cedgevw{nonedge}{vone}{\wvppone/2+\wvqmone/2}{}
          \cedgevw{edge}{vone}{\wvtwo/2+\wvpmone/2}{}
          \cedgevw{edge}{vone}{\wvqptwo/2+\wvn/2}{}
          \cedgevv{edge}{vone}{vqpone}{}

          \cedgeww{nonedge}{\wvppone/2+\wvqmone/2}{\wvtwo/2+\wvpmone/2}{}
          \cedgeww{nonedge}{\wvppone/2+\wvqmone/2}{\wvqptwo/2+\wvn/2}{}
          \cedgewv{nonedge}{\wvppone/2+\wvqmone/2}{vqpone}{}
        \end{scope}
        \begin{scope}[shift={(6,0)}]
          \node at (0,-2.5) {(b) between~\behref{beh:2mmr:neighbors-of-metamours-are-metamours}
            and~\behref{beh:2mmr:second-metamour-vl}};

          \draw [edge] (0,0) circle (\cradius);
          
          \def\wvone{90}
          \def\wvtwo{105}
          \def\wvpmone{195}
          \def\wvp{210}
          \def\wvppone{225}
          \def\wvqmone{315}
          \def\wvq{330}
          \def\wvqpone{345}
          \def\wvqptwo{360}
          \def\wvn{435}

          \cvertex{\wvone}{$v_1$}{vone}
          \cvertices{\wvtwo}{\wvpmone}{$v_2$}{$v_{p-1}$}{6}
          \node [svertex] (vpmone) at (\wvpmone:\cradius) {};
          \cvertex{\wvp}{$v_p$}{vp}
          \cvertices{\wvppone}{\wvqmone}{$v_{p+1}$}{$v_{q-1}$}{6}
          \node [svertex] (vppone) at (\wvppone:\cradius) {};
          \node [svertex] (vqmone) at (\wvqmone:\cradius) {};
          \cvertex{\wvq}{$v_q$}{vq}
          \cvertex{\wvqpone}{$v_{q+1}$}{vqpone}
          \cvertices{\wvqptwo}{\wvn}{$v_{q+2}$}{$v_n$}{5}
          \node [svertex] (vqptwo) at (\wvqptwo:\cradius) {};

          \cedgevv{nonedge}{vone}{vp}{}
          \cedgevv{mmedge}{vone}{vp}{}
          \cedgevv{nonedge}{vone}{vq}{}
          \cedgevv{mmedge}{vone}{vq}{}
          
          \cedgevw{nonedge}{vone}{\wvppone/2+\wvqmone/2}{}
          \cedgevw{edge}{vone}{\wvtwo/2+\wvpmone/2}{}
          \cedgevw{edge}{vone}{\wvqptwo/2+\wvn/2}{}
          \cedgevv{edge}{vone}{vqpone}{}

          \cedgeww{nonedge}{\wvppone/2+\wvqmone/2}{\wvtwo/2+\wvpmone/2}{}
          \cedgeww{nonedge}{\wvppone/2+\wvqmone/2}{\wvqptwo/2+\wvn/2}{}
          \cedgewv{nonedge}{\wvppone/2+\wvqmone/2}{vqpone}{}

          \cedgevv{nonedge}{vpmone}{vppone}{looseness=2}
          \cedgevv{mmedge}{vpmone}{vppone}{looseness=2}
          \cedgevv{nonedge}{vqmone}{vqpone}{looseness=2}
          \cedgevv{mmedge}{vqmone}{vqpone}{looseness=2}
        \end{scope}
      \end{tikzpicture}
      \caption{Subgraphs of the situations
        between~\behref{beh:2mmr:no-edges-fellows-opponents},
        \behref{beh:2mmr:neighbors-of-metamours-are-metamours}, 
        and~\behref{beh:2mmr:second-metamour-vl}}
      \label{fig:2mmr:no-edges-fellows-opponents}
    \end{figure}

    The known edges and non-edges at this moment are
    shown in Figure~\ref{fig:2mmr:no-edges-fellows-opponents}(a).
    
    \begin{beh}\label{beh:2mmr:neighbors-of-metamours-are-metamours}
        The vertices $v_{q-1}$ and $v_{q+1}$ are metamours of each other. 
        Also the vertices $v_{p-1}$ and $v_{p+1}$ are metamours of each other.
    \end{beh}
    \begin{proofbeh}{beh:2mmr:neighbors-of-metamours-are-metamours}
        The vertices $v_{q-1}$ and $v_{q+1}$ have the common neighbor 
        $v_{q}$ and are not adjacent due 
        to~\behref{beh:2mmr:no-edges-fellows-opponents}, so they are metamours. 
        Also $v_{p-1}$ and $v_{p+1}$ are metamours because they have $v_p$ as  
        a common neighbor and are not adjacent because 
        of~\behref{beh:2mmr:no-edges-fellows-opponents}.
    \end{proofbeh}
    
    Now we are in the situation shown
    in Figure~\ref{fig:2mmr:no-edges-fellows-opponents}(b).
   
    \begin{beh}\label{beh:2mmr:second-metamour-vl}
        The vertices $v_{q}$ and $v_{q + 2}$ are metamours 
        of each other.
    \end{beh}
    \begin{figure}
      \centering
      \begin{tikzpicture}
        \begin{scope}[shift={(0,0)},
          spy using outlines={circle, magnification=2, connect spies}]

          \draw [edge] (0,0) circle (\cradius);
          
          \def\wvone{90}
          \def\wvtwo{105}
          \def\wvpmone{195}
          \def\wvp{210}
          \def\wvppone{225}
          \def\wvqmone{315}
          \def\wvq{330}
          \def\wvqpone{345}
          \def\wvqptwo{360}
          \def\wvn{435}

          \cvertex{\wvone}{$v_1$}{vone}
          \cvertices{\wvtwo}{\wvpmone}{$v_2$}{$v_{p-1}$}{6}
          \node [svertex] (vpmone) at (\wvpmone:\cradius) {};
          \cvertex{\wvp}{$v_p$}{vp}
          \cvertices{\wvppone}{\wvqmone}{$v_{p+1}$}{$v_{q-1}$}{6}
          \node [svertex] (vppone) at (\wvppone:\cradius) {};
          \node [svertex] (vqmone) at (\wvqmone:\cradius) {};
          \cvertex{\wvq}{$v_q$}{vq}
          \cvertex{\wvqpone}{$v_{q+1}$}{vqpone}
          \cvertices{\wvqptwo}{\wvn}{$v_{q+2}$}{$v_n$}{5}
          \node [svertex] (vqptwo) at (\wvqptwo:\cradius) {};

          \cedgevv{nonedge}{vone}{vp}{}
          \cedgevv{mmedge}{vone}{vp}{}
          \cedgevv{nonedge}{vone}{vq}{}
          \cedgevv{mmedge}{vone}{vq}{}
          
          \cedgevw{nonedge}{vone}{\wvppone/2+\wvqmone/2}{}
          \cedgevw{edge}{vone}{\wvtwo/2+\wvpmone/2}{}
          \cedgevw{edge}{vone}{\wvqptwo/2+\wvn/2}{}
          \cedgevv{edge}{vone}{vqpone}{}

          \cedgeww{nonedge}{\wvppone/2+\wvqmone/2}{\wvtwo/2+\wvpmone/2}{}
          \cedgeww{nonedge}{\wvppone/2+\wvqmone/2}{\wvqptwo/2+\wvn/2}{}
          \cedgewv{nonedge}{\wvppone/2+\wvqmone/2}{vqpone}{}

          \cedgevv{nonedge}{vpmone}{vppone}{looseness=2}
          \cedgevv{mmedge}{vpmone}{vppone}{looseness=2}
          \cedgevv{nonedge}{vqmone}{vqpone}{looseness=2}
          \cedgevv{mmedge}{vqmone}{vqpone}{looseness=2}

          \cedgevv{edge}{vq}{vqptwo}{looseness=2}
          \cedgevv{nonedge}{vqmone}{vqptwo}{looseness=2}
          \cedgevv{mmedge}{vqmone}{vqptwo}{looseness=2}
          \cedgevv{edge}{vqmone}{vone}{}

          \spy [blue, size=4cm] on (337.5:\cinner) in node [fill=white] at (6,0);
        \end{scope}
      \end{tikzpicture}
      \caption{Subgraph of the situation
        in the proof of~\behref{beh:2mmr:second-metamour-vl}}
      \label{fig:2mmr:second-metamour-vl}
    \end{figure}
    \begin{proofbeh}{beh:2mmr:second-metamour-vl}
        This proof is accompanied by Figure~\ref{fig:2mmr:second-metamour-vl}.
        Assume $v_{q}$ and $v_{q+2}$ are adjacent. 
        Then $v_{q-1}$ and $v_{q+2}$ have the common neighbor $v_q$ 
        and are not adjacent because 
        of~\behref{beh:2mmr:no-edges-fellows-opponents}. Hence, $v_{q+2}$ is 
        a metamour of $v_{q-1}$. Due 
        to~\behref{beh:2mmr:neighbors-of-metamours-are-metamours},
        $v_{q + 1}$ is the second metamour of $v_{q-1}$.
        Both metamours are consecutive vertices on the cycle~$\gamma$,
        therefore, every other vertex except~$v_{q-1}$ is a fellow
        of~$v_{q-1}$, thus adjacent to~$v_{q-1}$
        by~\behref{beh:2mmr:adj-to-fellows}. In particular, $v_1$ is
        adjacent to~$v_{q-1}$ which
        contradicts~\behref{beh:2mmr:situation-v1}.

        Therefore, $v_{q}$ and $v_{q+2}$ are not adjacent and because
        of their common neighbor~$v_{q+1}$, metamours.
    \end{proofbeh}

    \begin{beh}\label{beh:2mmr:neighbors-vl}
      The vertex $v_q$ is adjacent to $v_2$, \dots, $v_{p-1}$,
      $v_p$, $v_{p+1}$, \dots, $v_{q-1}$.
    \end{beh}
    \begin{proofbeh}{beh:2mmr:neighbors-vl}
        The two metamours of $v_q$ are $v_1$ and $v_{q+2}$ because 
        of~\behref{beh:2mmr:second-metamour-vl}.
        This implies that $v_{2}$, \dots, $v_{q-1}$ are fellows of $v_q$ 
        and therefore adjacent to 
        $v_q$ because of~\behref{beh:2mmr:adj-to-fellows}.
    \end{proofbeh}    
            
    \begin{figure}
      \centering
      \begin{tikzpicture}
        \begin{scope}[shift={(0,0)},
          spy using outlines={circle, magnification=2, connect spies}]

          \draw [edge] (0,0) circle (\cradius);
          
          \def\wvone{90}
          \def\wvtwo{105}
          \def\wvpmone{195}
          \def\wvp{210}
          \def\wvppone{225}
          \def\wvqmone{315}
          \def\wvq{330}
          \def\wvqpone{345}
          \def\wvqptwo{360}
          \def\wvn{435}

          \cvertex{\wvone}{$v_1$}{vone}
          \cvertices{\wvtwo}{\wvpmone}{$v_2$}{$v_{p-1}$}{6}
          \node [svertex] (vpmone) at (\wvpmone:\cradius) {};
          \cvertex{\wvp}{$v_p$}{vp}
          \cvertices{\wvppone}{\wvqmone}{$v_{p+1}$}{$v_{q-1}$}{6}
          \node [svertex] (vppone) at (\wvppone:\cradius) {};
          \node [svertex] (vqmone) at (\wvqmone:\cradius) {};
          \cvertex{\wvq}{$v_q$}{vq}
          \cvertex{\wvqpone}{$v_{q+1}$}{vqpone}
          \cvertices{\wvqptwo}{\wvn}{$v_{q+2}$}{$v_n$}{5}
          \node [svertex] (vqptwo) at (\wvqptwo:\cradius) {};

          \cedgevv{nonedge}{vone}{vp}{}
          \cedgevv{mmedge}{vone}{vp}{}
          \cedgevv{nonedge}{vone}{vq}{}
          \cedgevv{mmedge}{vone}{vq}{}
          
          \cedgevw{nonedge}{vone}{\wvppone/2+\wvqmone/2}{}
          \cedgevw{edge}{vone}{\wvtwo/2+\wvpmone/2}{}
          \cedgevw{edge}{vone}{\wvqptwo/2+\wvn/2}{}
          \cedgevv{edge}{vone}{vqpone}{}

          \cedgeww{nonedge}{\wvppone/2+\wvqmone/2}{\wvtwo/2+\wvpmone/2}{}
          \cedgeww{nonedge}{\wvppone/2+\wvqmone/2}{\wvqptwo/2+\wvn/2}{}
          \cedgewv{nonedge}{\wvppone/2+\wvqmone/2}{vqpone}{}

          \cedgevv{nonedge}{vpmone}{vppone}{looseness=2}
          \cedgevv{mmedge}{vpmone}{vppone}{looseness=2}
          \cedgevv{nonedge}{vqmone}{vqpone}{looseness=2}
          \cedgevv{mmedge}{vqmone}{vqpone}{looseness=2}

          \cedgevv{nonedge}{vq}{vqptwo}{looseness=2}
          \cedgevv{mmedge}{vq}{vqptwo}{looseness=2}

          \cedgevw{edge}{vq}{\wvtwo/2+\wvpmone/2}{}
          \cedgevv{edge}{vq}{vp}{}
          \cedgevw{edge}{vq}{\wvppone/2+\wvqmone/2}{}

          \spy [blue, size=4cm] on (337.5:\cinner) in node [fill=white] at (6,0);
        \end{scope}
      \end{tikzpicture}
      \caption{Subgraph of the situation
        between~\behref{beh:2mmr:neighbors-vl} and
        \behref{beh:2mmr:only-one-fellow-one-side}}
      \label{fig:2mmr:2mmr:neighbors-vl}
    \end{figure}

    Figure~\ref{fig:2mmr:2mmr:neighbors-vl} shows the current situation.

    \begin{beh}\label{beh:2mmr:only-one-fellow-one-side}
        The vertices $v_{q-1}$ and $v_{p -1}$ are metamours of each other.
        Furthermore, $v_{p-1} = v_2$ holds, so there is exactly one fellow of 
        $v_1$ on 
        the cycle $\gamma$ between $v_1$ and $v_p$.
    \end{beh}
    \begin{proofbeh}{beh:2mmr:only-one-fellow-one-side}
        The vertex $v_{q-1}$ is not adjacent to any of $v_{2}$, \dots, 
        $v_{p-1}$ due to~\behref{beh:2mmr:no-edges-fellows-opponents}. 
        Furthermore, $v_{q-1}$
        has the common neighbor $v_q$ with each of these vertices because 
        of~\behref{beh:2mmr:neighbors-vl}. So every vertex $v_{2}$, \dots, 
        $v_{p-1}$ is a metamour of $v_{q-1}$. 
        This implies $\abs{\set{ v_{2}, \dots, v_{p-1}}} \le 1$ because 
        $v_{q-1}$ 
        also has $v_{q+1}$ as metamour and has in total exactly two 
        metamours.
        Moreover, as $v_2$ is adjacent to~$v_1$, $v_1$ and $v_2$ are not
        metamours, thus $v_2$ and $v_p$ cannot coincide.
        This implies $p > 2$ has to hold.
        In consequence, we obtain $p = 3$ implying $v_{p-1}=v_2$ has to hold.
    \end{proofbeh}

    \begin{figure}[t]
      \centering
      \begin{tikzpicture}
        \begin{scope}[shift={(0,0)}]
          \node at (0,-3.5) {(a)
          between~\behref{beh:2mmr:only-one-fellow-one-side}
            and \behref{beh:2mmr:exactly-one-opponent}};
      
          \draw [edge] (0,0) circle (\cradius);

          \def\wvone{90}
          \def\wvtwo{150}
          \def\wvpmone{150}
          \def\wvp{210}
          \def\wvppone{225}
          \def\wvqmone{315}
          \def\wvq{330}
          \def\wvqpone{345}
          \def\wvqptwo{360}
          \def\wvn{435}

          \cvertex{\wvone}{$v_1$}{vone}
          \cvertex{\wvpmone}{$v_{p-1}=v_2$}{vpmone}
          \cvertex{\wvp}{$v_p=v_3$}{vp}
          \cvertices{\wvppone}{\wvqmone}{$v_{p+1}=v_4$}{$v_{q-1}$}{6}
          \node [svertex] (vppone) at (\wvppone:\cradius) {};
          \node [svertex] (vqmone) at (\wvqmone:\cradius) {};
          \cvertex{\wvq}{$v_q$}{vq}
          \cvertex{\wvqpone}{$v_{q+1}$}{vqpone}
          \cvertices{\wvqptwo}{\wvn}{$v_{q+2}$}{$v_n$}{5}
          \node [svertex] (vqptwo) at (\wvqptwo:\cradius) {};

          \cedgevv{nonedge}{vone}{vp}{}
          \cedgevv{mmedge}{vone}{vp}{}
          \cedgevv{nonedge}{vone}{vq}{}
          \cedgevv{mmedge}{vone}{vq}{}
          
          \cedgevw{nonedge}{vone}{\wvppone/2+\wvqmone/2}{}
          \cedgevw{edge}{vone}{\wvqptwo/2+\wvn/2}{}
          \cedgevv{edge}{vone}{vqpone}{}

          \cedgewv{nonedge}{\wvppone/2+\wvqmone/2}{vpmone}{}
          \cedgeww{nonedge}{\wvppone/2+\wvqmone/2}{\wvqptwo/2+\wvn/2}{}
          \cedgewv{nonedge}{\wvppone/2+\wvqmone/2}{vqpone}{}

          \cedgevv{nonedge}{vpmone}{vppone}{looseness=1.5}
          \cedgevv{mmedge}{vpmone}{vppone}{looseness=1.5}
          \cedgevv{nonedge}{vqmone}{vqpone}{looseness=1.5}
          \cedgevv{mmedge}{vqmone}{vqpone}{looseness=1.5}

          \cedgevv{nonedge}{vq}{vqptwo}{looseness=1.5}
          \cedgevv{mmedge}{vq}{vqptwo}{looseness=1.5}

          \cedgevv{edge}{vq}{vpmone}{}
          \cedgevv{edge}{vq}{vp}{}
          \cedgevw{edge}{vq}{\wvppone/2+\wvqmone/2}{}
          \cedgevv{nonedge}{vpmone}{vqmone}{}
          \cedgevv{mmedge}{vpmone}{vqmone}{}
        \end{scope}
        \begin{scope}[shift={(7,0)}]
          \node at (0,-3.5) {(b) between~\behref{beh:2mmr:exactly-one-opponent}
            and \behref{beh:2mmr:second-metamour-vp}};

          \draw [edge] (0,0) circle (\cradius);

          \def\wvone{90}
          \def\wvtwo{150}
          \def\wvpmone{150}
          \def\wvp{210}
          \def\wvppone{270}
          \def\wvqmone{270}
          \def\wvq{330}
          \def\wvqpone{345}
          \def\wvqptwo{360}
          \def\wvn{435}

          \cvertex{\wvone}{$v_1$}{vone}
          \cvertex{\wvpmone}{$v_{p-1}=v_2$}{vpmone}
          \cvertex{\wvp}{$v_p=v_3$}{vp}
          \cvertex{\wvppone}{\rotatebox{90}{$\begin{aligned}&v_{q-1}\\[-1ex]=&v_{p+1}\\[-1ex]=&v_4\end{aligned}$}}{vppone}
          \cvertex{\wvqmone}{}{vqmone}
          \node [svertex] (vqmone) at (\wvqmone:\cradius) {};
          \cvertex{\wvq}{$v_q=v_5$}{vq}
          \cvertex{\wvqpone}{$v_{q+1}=v_6$}{vqpone}
          \cvertices{\wvqptwo}{\wvn}{$v_{q+2}=v_7$}{$v_n$}{5}
          \node [svertex] (vqptwo) at (\wvqptwo:\cradius) {};

          \cedgevv{nonedge}{vone}{vp}{}
          \cedgevv{mmedge}{vone}{vp}{}
          \cedgevv{nonedge}{vone}{vq}{}
          \cedgevv{mmedge}{vone}{vq}{}
          
          \cedgevv{nonedge}{vone}{vppone}{}
          \cedgevw{edge}{vone}{\wvqptwo/2+\wvn/2}{}
          \cedgevv{edge}{vone}{vqpone}{}

          \cedgevw{nonedge}{vppone}{\wvqptwo/2+\wvn/2}{}

          \cedgevv{nonedge}{vpmone}{vppone}{looseness=1}
          \cedgevv{mmedge}{vpmone}{vppone}{looseness=1}
          \cedgevv{nonedge}{vqmone}{vqpone}{looseness=1}
          \cedgevv{mmedge}{vqmone}{vqpone}{looseness=1}

          \cedgevv{nonedge}{vq}{vqptwo}{looseness=1}
          \cedgevv{mmedge}{vq}{vqptwo}{looseness=1}

          \cedgevv{edge}{vq}{vpmone}{}
          \cedgevv{edge}{vq}{vp}{}
        \end{scope}
      \end{tikzpicture}
      \caption{Subgraphs of the situations
        between~\behref{beh:2mmr:only-one-fellow-one-side},
        \behref{beh:2mmr:exactly-one-opponent} and
        \behref{beh:2mmr:second-metamour-vp}}
      \label{fig:2mmr:only-one-fellow-one-side}
    \end{figure}
    
    We are now in the situation shown in
    Figure~\ref{fig:2mmr:only-one-fellow-one-side}(a).
    
    \begin{beh}\label{beh:2mmr:exactly-one-opponent}
        It holds that $v_{p+1} = v_{q-1}$, so $v_1$ has exactly one opponent.
    \end{beh}
    \begin{proofbeh}{beh:2mmr:exactly-one-opponent}        
        The vertices $v_{q+1}$ and $v_{p-1}$ are 
        metamours of $v_{q-1}$ because 
        of~\behref{beh:2mmr:neighbors-of-metamours-are-metamours} 
        and~\behref{beh:2mmr:only-one-fellow-one-side}.
        Furthermore, $v_{p-1}$ and $v_{p+1}$ are metamours because 
        of~\behref{beh:2mmr:neighbors-of-metamours-are-metamours}.
        
        Now assume $p + 1 < q - 1$, so the vertices $v_{p+1}$ and 
        $v_{q-1}$ are distinct. Then $v_{p+1}$ and $v_{q+1}$ have the 
        common neighbor $v_q$ because of~\behref{beh:2mmr:neighbors-vl} and 
        they are not adjacent because 
        of~\behref{beh:2mmr:no-edges-fellows-opponents}, so they are metamours. 
        This implies that $(v_{q-1}, v_{q+1}, v_{p+1}, 
        v_{p-1},v_{q-1})$ is a cycle in the metamour graph that does not 
        contain all vertices, a contradiction to our assumption.
        So $p + 1 = q - 1$. 
    \end{proofbeh}
    
    Now we are in the situation shown in
    of Figure~\ref{fig:2mmr:only-one-fellow-one-side}(b).
    
    \begin{figure}[t]
      \centering
      \begin{tikzpicture}
        \begin{scope}[shift={(0,0)}]
          \node at (0,-3.5) {(a) between~\behref{beh:2mmr:second-metamour-vp}
            and \behref{beh:2mmr:only-two-fellows-other-side}};

          \draw [edge] (0,0) circle (\cradius);

          \def\wvone{90}
          \def\wvtwo{150}
          \def\wvpmone{150}
          \def\wvp{210}
          \def\wvppone{270}
          \def\wvqmone{270}
          \def\wvq{330}
          \def\wvqpone{345}
          \def\wvqptwo{360}
          \def\wvn{435}

          \cvertex{\wvone}{$v_1$}{vone}
          \cvertex{\wvpmone}{$v_{p-1}=v_2$}{vpmone}
          \cvertex{\wvp}{$v_p=v_3$}{vp}
          \cvertex{\wvppone}{\rotatebox{90}{$\begin{aligned}&v_{q-1}\\[-1ex]=&v_{p+1}\\[-1ex]=&v_4\end{aligned}$}}{vppone}
          \cvertex{\wvqmone}{}{vqmone}
          \node [svertex] (vqmone) at (\wvqmone:\cradius) {};
          \cvertex{\wvq}{$v_q=v_5$}{vq}
          \cvertex{\wvqpone}{$v_{q+1}=v_6$}{vqpone}
          \cvertices{\wvqptwo}{\wvn}{$v_{q+2}=v_7$}{$v_n$}{5}
          \node [svertex] (vqptwo) at (\wvqptwo:\cradius) {};

          \cedgevv{nonedge}{vone}{vp}{}
          \cedgevv{mmedge}{vone}{vp}{}
          \cedgevv{nonedge}{vone}{vq}{}
          \cedgevv{mmedge}{vone}{vq}{}
          
          \cedgevv{nonedge}{vone}{vppone}{}
          \cedgevw{edge}{vone}{\wvqptwo/2+\wvn/2}{}
          \cedgevv{edge}{vone}{vqpone}{}

          \cedgevw{nonedge}{vppone}{\wvqptwo/2+\wvn/2}{}

          \cedgevv{nonedge}{vpmone}{vppone}{looseness=1}
          \cedgevv{mmedge}{vpmone}{vppone}{looseness=1}
          \cedgevv{nonedge}{vqmone}{vqpone}{looseness=1}
          \cedgevv{mmedge}{vqmone}{vqpone}{looseness=1}

          \cedgevv{nonedge}{vq}{vqptwo}{looseness=1}
          \cedgevv{mmedge}{vq}{vqptwo}{looseness=1}

          \cedgevv{edge}{vq}{vpmone}{}
          \cedgevv{edge}{vq}{vp}{}

          \cedgevv{nonedge}{vp}{vqpone}{}
          \cedgevv{mmedge}{vp}{vqpone}{}
          \cedgevw{nonedge}{vp}{\wvqptwo/2+\wvn/2}{}
        \end{scope}
        \begin{scope}[shift={(7,0)}]
          \node at (0.5,-3.5) {(b)
          between~\behref{beh:2mmr:only-two-fellows-other-side}
            and \behref{beh:2mmr:g-is-h7b}};

          \draw [edge] (0,0) circle (\cradius);

          \def\wvone{90}
          \def\wvtwo{150}
          \def\wvpmone{150}
          \def\wvp{210}
          \def\wvppone{270}
          \def\wvqmone{270}
          \def\wvq{330}
          \def\wvqpone{345}
          \def\wvqptwo{398}
          \def\wvn{398}

          \cvertex{\wvone}{$v_1$}{vone}
          \cvertex{\wvpmone}{$v_{p-1}=v_2$}{vpmone}
          \cvertex{\wvp}{$v_p=v_3$}{vp}
          \cvertex{\wvppone}{\rotatebox{90}{$\begin{aligned}&v_{q-1}\\[-1ex]=&v_{p+1}\\[-1ex]=&v_4\end{aligned}$}}{vppone}
          \cvertex{\wvqmone}{}{vqmone}
          \node [svertex] (vqmone) at (\wvqmone:\cradius) {};
          \cvertex{\wvq}{$v_q=v_5$}{vq}
          \cvertex{\wvqpone}{$v_{q+1}=v_6$}{vqpone}
          \cvertex{\wvqptwo}{$\begin{aligned}&v_n\\[-1ex]=&v_{q+2}\\[-1ex]=&v_7\end{aligned}$}{vqptwo}

          \cedgevv{nonedge}{vone}{vp}{}
          \cedgevv{mmedge}{vone}{vp}{}
          \cedgevv{nonedge}{vone}{vq}{}
          \cedgevv{mmedge}{vone}{vq}{}
          
          \cedgevv{nonedge}{vone}{vppone}{}
          \cedgevv{edge}{vone}{vqpone}{}

          \cedgevv{nonedge}{vppone}{vqptwo}{}

          \cedgevv{nonedge}{vpmone}{vppone}{looseness=1}
          \cedgevv{mmedge}{vpmone}{vppone}{looseness=1}
          \cedgevv{nonedge}{vqmone}{vqpone}{looseness=1}
          \cedgevv{mmedge}{vqmone}{vqpone}{looseness=1}

          \cedgevv{nonedge}{vq}{vqptwo}{looseness=1}
          \cedgevv{mmedge}{vq}{vqptwo}{looseness=1}

          \cedgevv{edge}{vq}{vpmone}{}
          \cedgevv{edge}{vq}{vp}{}

          \cedgevv{nonedge}{vp}{vqpone}{}
          \cedgevv{mmedge}{vp}{vqpone}{}
          \cedgevv{nonedge}{vp}{vqptwo}{}

          \cedgevv{edge}{vpmone}{vqpone}{}
        \end{scope}
      \end{tikzpicture}
      \caption{Subgraphs of the situations
        between~\behref{beh:2mmr:second-metamour-vp},
        \behref{beh:2mmr:only-two-fellows-other-side} and
        \behref{beh:2mmr:g-is-h7b}}
      \label{fig:2mmr:exactly-one-opponent}
    \end{figure}
            
    \begin{beh}\label{beh:2mmr:second-metamour-vp}
        The vertices $v_{p}$ and $v_{q + 1}$ are metamours of each other. 
        Furthermore, $v_{p}$ is not adjacent to any of the vertices 
        $v_{q+2}$, \dots, 
        $v_{n}$. 
    \end{beh}
    \begin{proofbeh}{beh:2mmr:second-metamour-vp}
        If $v_p$ is adjacent to $v_i$ for $i \in \set{q+2, \dots, n}$, then 
        $v_{p+1}$ and $v_i$ are metamours because they have $v_p$ as a common 
        neighbor, and they are not adjacent due 
        to~\behref{beh:2mmr:no-edges-fellows-opponents}. This is a 
        contradiction as $v_{p+1}$ already has the two metamours 
        $v_{p-1}$ and $v_{q+1}$ because 
        of~\behref{beh:2mmr:neighbors-of-metamours-are-metamours} 
        and an implication of~\behref{beh:2mmr:exactly-one-opponent}.
        As a result, $v_p$ is not adjacent to any of $v_{q+2}$, \dots, 
        $v_{n}$. 
        
        If $v_p$ would be adjacent to $v_{q+1}$, then $v_p$ and $v_{q+2}$ 
        are metamours because of the common neighbor $v_{q+1}$ and because 
        they are not adjacent by the above. But then, due 
        to~\behref{beh:2mmr:second-metamour-vl}, 
        $(v_p,v_1,v_q,v_{q+2},v_p)$ is a cycle in the metamour graph 
        which does not contain all vertices, a contradiction to our
        assumption. Therefore, $v_p$ 
        is not adjacent to $v_{q+1}$. The vertex $v_p$ is adjacent to
        $v_{q}$ due to~\behref{beh:2mmr:neighbors-vl}, therefore $v_{q}$ is a
        common neighbor of $v_p$ and $v_{q+1}$, and hence these vertices
        are metamours.
    \end{proofbeh} 
    
    Figure~\ref{fig:2mmr:exactly-one-opponent}(a) shows the situation.

    \begin{beh}\label{beh:2mmr:only-two-fellows-other-side}
        It holds that $q + 2 = n$, so there are exactly two fellows of $v_1$ 
        on the cycle $\gamma$ between $v_q$ and $v_1$. 
        Furthermore, the vertices $v_{p-1}$ and $v_{q+2}$ are metamours of 
        each 
        other, and $v_{p-1}$ is adjacent to all vertices except its metamours.
    \end{beh}
    \begin{proofbeh}{beh:2mmr:only-two-fellows-other-side}        
        The vertex $v_{p-1}$ is a metamour of $v_{q-1}$ due 
        to~\behref{beh:2mmr:only-one-fellow-one-side}, and it is adjacent to 
        $v_q$ because of~\behref{beh:2mmr:neighbors-vl}. This together with 
        $p+1 = q-1$ by~\behref{beh:2mmr:exactly-one-opponent} implies that 
        $v_{p-1}$ is adjacent to one of its opponents, namely $v_{q}$. Then 
        by~\behref{beh:2mmr:adj-to-all-or-non-opponents} 
        and~\behref{beh:2mmr:adj-to-fellows}, this implies that $v_{p-1}$ is 
        adjacent to all vertices except its metamours.
        
        If $v_{p-1}$ is adjacent to a vertex $v_i$ for 
        $i \in \set{q+2, \dots, n}$, then $v_{p}$ and $v_{i}$ are metamours 
        because they have $v_{p-1}$ as common neighbor and are not adjacent due 
        to~\behref{beh:2mmr:second-metamour-vp}. But $v_p$ already has the two 
        metamours $v_1$ and $v_{q+1}$ due 
        to~\behref{beh:2mmr:second-metamour-vp}, a contradiction.
        As a result, $v_{p-1}$ is not adjacent to any vertex of $v_{q+2}$, 
        \dots, $v_n$.
        
        Now assume $q + 2 < n$, so the vertex $v_{q+3}$ exists.
        Due to the fact that $v_{p-1}$ is adjacent to all vertices except its 
        metamours and that it has $v_{p+1}$ as metamour 
        by~\behref{beh:2mmr:neighbors-of-metamours-are-metamours}, it follows 
        that it is adjacent to at least one of $v_{q+2}$ and $v_{q+3}$. 
        But we showed that $v_{p-1}$ is not adjacent to any of these two 
        vertices, a contradiction. Therefore, $q + 2 = n$ holds.
        Furthermore, $v_{p-1}$ is not adjacent to $v_{q+2}$, and therefore 
        these two vertices are metamours of each other.
    \end{proofbeh}
    
    Our final figure is
    Figure~\ref{fig:2mmr:exactly-one-opponent}(b).
    
    \begin{beh}\label{beh:2mmr:g-is-h7b}
        It holds that $G = \specialH{7}{b}$.
    \end{beh}
    \begin{proofbeh}{beh:2mmr:g-is-h7b}
        We 
        have $p-1 = 2$ by~\behref{beh:2mmr:only-one-fellow-one-side}, we have 
        $p+1=q-1$ by~\behref{beh:2mmr:exactly-one-opponent} and $q+2=n$ 
        by~\behref{beh:2mmr:only-two-fellows-other-side}. This implies that 
        $n=7$. 
        
        The properties we have derived so far fix all 
        edges and non-edges of $G$ except between~$v_2$ and $v_7$. This
        has to be a non-edge to close the metamour cycle.
        The result is $G = \specialH{7}{b}$.
        With respect to Figure~\ref{fig:graphs-7},
        $v_1$ is the top left vertex of $\specialH{7}{b}$ and the 
        vertices 
        are numbered clock-wise.        
    \end{proofbeh}
    
    This completes the proof of Lemma~\ref{lem:2mmr:deg-two-options}.
\end{proof}

Next we consider all cases of 
$2$-metamour-regular graphs whose metamour graph is connected, that contain  
a cycle of length~$n$ and whose degrees are not as in the previous lemma.

    \begin{lemma}\label{lemma:2mmr:all-deg-nMinus3}
      Let $G$ be a connected $2$-metamour-regular graph with $n$ vertices
      \begin{itemize}
      \item whose metamour graph equals the~$C_n$,
      \item that contains a cycle of length~$n$, and
      \item in which every vertex has degree $n-3$.
      \end{itemize}
      Then
      \begin{equation*}
        G=\complement{C_n}.
      \end{equation*}
    \end{lemma}
    \begin{proof}
        If a vertex~$v$ of $G$ has degree $n-3$, then $v$ 
        is adjacent to all but two vertices.
        These two vertices are exactly the metamours of~$v$.
        This implies that $G$ equals the complement of the
        metamour graph.
        Hence, $G=\complement{C_n}$ as the metamour graph of 
        $G$ is the~$C_n$.
    \end{proof}
  
    \begin{lemma}\label{lemma:2mmr:all-deg-2}
      Let $G$ be a connected $2$-metamour-regular graph with $n$ vertices
      \begin{itemize}
      \item whose metamour graph equals the~$C_n$,
      \item that contains a cycle of length~$n$, and
      \item in which every vertex has degree $2$.
      \end{itemize}
      Then
      \begin{equation*}
        G=C_n
      \end{equation*}
      and $n$ is odd.
    \end{lemma}
    \begin{proof}
      Let $\gamma$ be a cycle of length $n$ in $G$. 
      If every vertex of $G$ has degree~$2$, then every vertex in
      the induced subgraph $G[\gamma]$ has degree~$2$
      as $\gamma$ contains every vertex by assumption. As $G[\gamma]$ is connected,
      it equals $C_n$. In total this implies $G=G[\gamma]=C_n$.
      
      It is easy to see that if $n$ is even, then the metamour graph
      consists of exactly two cycles of length $\frac{n}{2}$
      which contradicts our assumption.
      Therefore, $n$ is odd.
    \end{proof}

\begin{lemma}\label{lem:2mmr:deg-2-and-nm3}
    Let $G$ be a connected $2$-metamour-regular graph with $n$ vertices
    \begin{itemize}
    \item whose metamour graph equals the~$C_n$,
    \item that contains a cycle of length~$n$,
    \item in which every vertex has degree $2$ or $n-3$, and
    \item that has a vertex of degree $2$ and
      a vertex of degree $n-3$.
    \end{itemize}
    Then
    \begin{equation*}
      G \in \set{C_5, \specialH{6}{c}}.
    \end{equation*}
\end{lemma}

\resetbeh
\begin{proof}
  Let $\gamma=(v_1,\dots,v_n,v_1)$ be
  a cycle of length~$n$ such that $\deg(v_1) = 2$ and $\deg(v_2) = n-3$.

    \begin{beh}\label{beh:2mmr:mixed-degrees-nle7}
      We have $5 \le n \le 7$ and the metamours of~$v_1$ are $v_3$ and 
      $v_{n-1}$.
    \end{beh}
      
    \begin{figure}
      \centering
      \begin{tikzpicture}
        \begin{scope}[shift={(0,0)}]
          \node at (0,-2) {$\edge{v_2}{v_i} \in E(G)$};

          \def\cradius{1}
          \def\cinner{0.8}
          \def\couter{1.2}
          \def\cstep{30}

          \draw [edge] (0,0) circle (\cradius);

          \def\wvone{90}
          \def\wvtwo{120}
          \def\wvthree{150}
          \def\wvfour{180}
          \def\wvi{240}
          \def\wvnmtwo{360}
          \def\wvnmone{390}
          \def\wvn{420}

          \cvertex{\wvone}{$v_1$}{vone}
          \cvertex{\wvtwo}{$v_2$}{vtwo}
          \cvertex{\wvthree}{$v_3$}{vthree}
          \cvertices{\wvfour}{\wvnmtwo}{$v_4$}{$v_{n-2}$}{6}
          \node [svertex] (vi) at (\wvi:\cradius) {};
          \cvlabel{\wvi}{$v_i$}
          \cvertex{\wvnmone}{$v_{n-1}$}{vnmone}
          \cvertex{\wvn}{$v_n$}{vn}

          \cedgevv{nonedge}{vone}{vthree}{}
          \cedgevv{mmedge}{vone}{vthree}{}
          \cedgevv{nonedge}{vone}{vnmone}{}
          \cedgevv{mmedge}{vone}{vnmone}{}

          \cedgevv{edge}{vtwo}{vi}{}
          \cedgevv{nonedge}{vone}{vi}{}
          \cedgevv{mmedge}{vone}{vi}{}
          
        \end{scope}
      \end{tikzpicture}
      \caption{Subgraph of the situation
        in the proof of \behref{beh:2mmr:mixed-degrees-nle7}}
      \label{fig:2mmr:mixed-degrees-nle7}
    \end{figure}

    \begin{proofbeh}{beh:2mmr:mixed-degrees-nle7}
    Clearly $v_1$ is only adjacent to $v_2$ and $v_n$. Hence, $v_3$ and 
    $v_{n-1}$ have to be the two metamours of $v_1$ and $G$ contains at least 
    $5$ different vertices, so $n \ge 5$.
    
    If $v_2$ is adjacent to some $v_i$ for $i\in \set{4, \dots, n-2}$, then 
    $v_1$ is a metamour of $v_i$ due to the common neighbor $v_2$;
    see Figure~\ref{fig:2mmr:mixed-degrees-nle7}. Hence, 
    $v_2$ is not adjacent to any vertex $v_4$, \dots, $v_{n-2}$. However, 
    because $\deg(v_2) = n-3$, the vertex $v_2$ is adjacent to every vertex but 
    its two metamours. This implies that $\abs{\set{v_4, \dots, v_{n-2}}} \le 2$, because $v_2$ 
    has at most two metamours among $v_{4}$, \dots, $v_{n-2}$. As a result, we 
    have $n \le 7$.
    \end{proofbeh} 
    
    This implies that $n=5$, $n=6$ and $n=7$ are the only cases to 
    consider. We do so in the following claims.

    \begin{beh}\label{beh:2mmr:case-n5}
        If $n=5$, then $G=C_n$.
    \end{beh}

    \begin{figure}
      \centering
      \begin{tikzpicture}
        \begin{scope}[shift={(0,0)}]
          \node at (0,-2.5) {(a) earlier};

          \def\cradius{1}
          \def\cstep{72}

          \draw [edge] (0,0) circle (\cradius);

          \def\wvone{90}
          \def\wvtwo{162}
          \def\wvthree{234}
          \def\wvnmone{306}
          \def\wvn{378}

          \cvertex{\wvone}{$v_1$}{vone}
          \cvertex{\wvtwo}{$v_2$}{vtwo}
          \cvertex{\wvthree}{$v_3$}{vthree}
          \cvertex{\wvnmone}{$v_{n-1}=v_4$}{vnmone}
          \cvertex{\wvn}{$v_n=v_5$}{vn}

          \cedgevv{nonedge}{vone}{vthree}{}
          \cedgevv{mmedge}{vone}{vthree}{}
          \cedgevv{nonedge}{vone}{vnmone}{}
          \cedgevv{mmedge}{vone}{vnmone}{}
        \end{scope}
        \node at (4,0) {$\Longrightarrow$};
        \begin{scope}[shift={(7,0)}]
          \node at (0,-2.5) {(b) later};

          \def\cradius{1}
          \def\cstep{72}

          \draw [edge] (0,0) circle (\cradius);

          \def\wvone{90}
          \def\wvtwo{162}
          \def\wvthree{234}
          \def\wvnmone{306}
          \def\wvn{378}

          \cvertex{\wvone}{$v_1$}{vone}
          \cvertex{\wvtwo}{$v_2$}{vtwo}
          \cvertex{\wvthree}{$v_3$}{vthree}
          \cvertex{\wvnmone}{$v_{n-1}=v_4$}{vnmone}
          \cvertex{\wvn}{$v_n=v_5$}{vn}

          \cedgevv{nonedge}{vone}{vthree}{}
          \cedgevv{mmedge}{vone}{vthree}{}
          \cedgevv{nonedge}{vone}{vnmone}{}
          \cedgevv{mmedge}{vone}{vnmone}{}

          \cedgevv{nonedge}{vtwo}{vnmone}{}
          \cedgevv{mmedge}{vtwo}{vnmone}{}
          \cedgevv{nonedge}{vthree}{vn}{}
          \cedgevv{mmedge}{vthree}{vn}{}

          \cedgevv{nonedge}{vtwo}{vn}{}
          \cedgevv{mmedge}{vtwo}{vn}{}
        \end{scope}
      \end{tikzpicture}
      \caption{Subgraphs of the situations
        in the proof of \behref{beh:2mmr:case-n5}}
      \label{fig:2mmr:case-n5}
    \end{figure}

    \begin{proofbeh}{beh:2mmr:case-n5}
      If $n = 5$, then $v_3$ and $v_{n-1}=v_4$ are the two metamours of $v_1$;
      see Figure~\ref{fig:2mmr:case-n5}(a). Then $v_2$ 
    is the only option as second metamour of $v_4$, and $v_5$ is the only option 
    as second metamour of $v_3$. Then $v_2$ and $v_5$ have to be metamours in 
    order to close the cycle in the metamour graph;
    see Figure~\ref{fig:2mmr:case-n5}(b). As a result,
    we have $G = C_5$.
    \end{proofbeh} 

    \begin{beh}\label{beh:2mmr:case-n6}
        If $n=6$, then $G=\specialH{6}{c}$.    
    \end{beh}

    \begin{figure}
      \centering
      \begin{tikzpicture}
        \begin{scope}[shift={(0,0)}]
          \node at (0,-2.5) [align=center] {(a) \\ $\set{v_3,v_5}\not\in E(G)$};

          \def\cradius{1}
          \def\cstep{60}

          \draw [edge] (0,0) circle (\cradius);

          \def\wvone{90}
          \def\wvtwo{150}
          \def\wvthree{210}
          \def\wvfour{270}
          \def\wvnmone{330}
          \def\wvn{390}

          \cvertex{\wvone}{$v_1$}{vone}
          \cvertex{\wvtwo}{$v_2$}{vtwo}
          \cvertex{\wvthree}{$v_3$}{vthree}
          \cvertex{\wvfour}{$v_4$}{vfour}
          \cvertex{\wvnmone}{$v_{n-1}=v_5$}{vnmone}
          \cvertex{\wvn}{$v_n=v_6$}{vn}

          \cedgevv{nonedge}{vone}{vthree}{}
          \cedgevv{mmedge}{vone}{vthree}{}
          \cedgevv{nonedge}{vone}{vnmone}{}
          \cedgevv{mmedge}{vone}{vnmone}{}

          \cedgevv{nonedge}{vthree}{vnmone}{}
          \cedgevv{mmedge}{vthree}{vnmone}{}
        \end{scope}
        \begin{scope}[shift={(5,0)}]
          \node at (0,-2.5) [align=center] {(b) \\ $\set{v_3,v_5}\in E(G)$ \\ earlier};

          \def\cradius{1}
          \def\cstep{60}

          \draw [edge] (0,0) circle (\cradius);

          \def\wvone{90}
          \def\wvtwo{150}
          \def\wvthree{210}
          \def\wvfour{270}
          \def\wvnmone{330}
          \def\wvn{390}

          \cvertex{\wvone}{$v_1$}{vone}
          \cvertex{\wvtwo}{$v_2$}{vtwo}
          \cvertex{\wvthree}{$v_3$}{vthree}
          \cvertex{\wvfour}{$v_4$}{vfour}
          \cvertex{\wvnmone}{$v_{n-1}=v_5$}{vnmone}
          \cvertex{\wvn}{$v_n=v_6$}{vn}

          \cedgevv{nonedge}{vone}{vthree}{}
          \cedgevv{mmedge}{vone}{vthree}{}
          \cedgevv{nonedge}{vone}{vnmone}{}
          \cedgevv{mmedge}{vone}{vnmone}{}

          \cedgevv{edge}{vthree}{vnmone}{}

          \cedgevv{nonedge}{vthree}{vn}{}
          \cedgevv{mmedge}{vthree}{vn}{}
          \cedgevv{nonedge}{vtwo}{vnmone}{}
          \cedgevv{mmedge}{vtwo}{vnmone}{}
        \end{scope}
        \node at (8,0) {$\Longrightarrow$};
        \begin{scope}[shift={(10,0)}]
          \node at (0,-2.5) [align=center] {(c) \\ $\set{v_3,v_5}\in E(G)$ \\ later};

          \def\cradius{1}
          \def\cstep{60}

          \draw [edge] (0,0) circle (\cradius);

          \def\wvone{90}
          \def\wvtwo{150}
          \def\wvthree{210}
          \def\wvfour{270}
          \def\wvnmone{330}
          \def\wvn{390}

          \cvertex{\wvone}{$v_1$}{vone}
          \cvertex{\wvtwo}{$v_2$}{vtwo}
          \cvertex{\wvthree}{$v_3$}{vthree}
          \cvertex{\wvfour}{$v_4$}{vfour}
          \cvertex{\wvnmone}{$v_{n-1}=v_5$}{vnmone}
          \cvertex{\wvn}{$v_n=v_6$}{vn}

          \cedgevv{nonedge}{vone}{vthree}{}
          \cedgevv{mmedge}{vone}{vthree}{}
          \cedgevv{nonedge}{vone}{vnmone}{}
          \cedgevv{mmedge}{vone}{vnmone}{}

          \cedgevv{edge}{vthree}{vnmone}{}

          \cedgevv{nonedge}{vthree}{vn}{}
          \cedgevv{mmedge}{vthree}{vn}{}
          \cedgevv{nonedge}{vtwo}{vnmone}{}
          \cedgevv{mmedge}{vtwo}{vnmone}{}

          \cedgevv{edge}{vtwo}{vn}{}
          
          \cedgevv{nonedge}{vfour}{vtwo}{}
          \cedgevv{mmedge}{vfour}{vtwo}{}
          \cedgevv{nonedge}{vfour}{vn}{}
          \cedgevv{mmedge}{vfour}{vn}{}
        \end{scope}
      \end{tikzpicture}
      \caption{Subgraphs of the situations
        in the proof of \behref{beh:2mmr:case-n6}}
      \label{fig:2mmr:case-n6}
    \end{figure}

    \begin{proofbeh}{beh:2mmr:case-n6}    
    If $n = 6$, then $v_3$ and $v_{n-1}=v_5$ are the two metamours of $v_1$. 
    
    If $v_3$ is not adjacent to $v_5$, then $v_3$ and $v_5$ are metamours 
    because of their common neighbor $v_4$;
    see Figure~\ref{fig:2mmr:case-n6}(a).
    But then $(v_1,v_3,v_5,v_1)$ is a 
    cycle in the metamour graph that does not contain all vertices, a 
    contradiction to our assumption. Hence, $v_3$ and $v_5$ are adjacent;
    see Figure~\ref{fig:2mmr:case-n6}(b).
    Then $v_6$ is the only option left as the second metamour of $v_3$, and 
    $v_2$ is the only option left as the second metamour of $v_5$.
    
    If $v_2$ and $v_6$ are not adjacent, then they are metamours because 
    of their common neighbor $v_1$. But then $(v_1,v_3,v_6,v_2,v_5,v_1)$ is a 
    cycle in the metamour graph that does not contain $v_4$, a contradiction.
    So $v_2$ and $v_6$ are adjacent;
    see Figure~\ref{fig:2mmr:case-n6}(c).
    
    But then $v_4$ has to have $v_2$ and $v_6$ as metamours, because they are 
    the only options left. Hence, we obtain $G =  \specialH{6}{c}$. 
    \end{proofbeh}         
    
    \begin{beh}\label{beh:2mmr:case-n7}
      We cannot have $n=7$.
    \end{beh}

    \begin{figure}
      \centering
      \begin{tikzpicture}
        \begin{scope}[shift={(0,0)}]
          \node at (0,-2.5) {(a) first paragraph};

          \def\cradius{1}
          \def\cstep{51}

          \draw [edge] (0,0) circle (\cradius);

          \def\wvone{90}
          \def\wvtwo{141} 
          \def\wvthree{193} 
          \def\wvfour{244} 
          \def\wvfive{296} 
          \def\wvnmone{347} 
          \def\wvn{399} 

          \cvertex{\wvone}{$v_1$}{vone}
          \cvertex{\wvtwo}{$v_2$}{vtwo}
          \cvertex{\wvthree}{$v_3$}{vthree}
          \cvertex{\wvfour}{$v_4$}{vfour}
          \cvertex{\wvfive}{$v_5$}{vfive}
          \cvertex{\wvnmone}{$v_{n-1}=v_6$}{vnmone}
          \cvertex{\wvn}{$v_n=v_7$}{vn}

          \cedgevv{nonedge}{vone}{vthree}{}
          \cedgevv{mmedge}{vone}{vthree}{}
          \cedgevv{nonedge}{vone}{vnmone}{}
          \cedgevv{mmedge}{vone}{vnmone}{}

          \cedgevv{nonedge}{vone}{vfive}{}
          
        \end{scope}
        \node at (4,0) {$\Longrightarrow$};
        \begin{scope}[shift={(7,0)}]
          \node at (0,-2.5) {(b) second paragraph};

          \def\cradius{1}
          \def\cstep{51}

          \draw [edge] (0,0) circle (\cradius);

          \def\wvone{90}
          \def\wvtwo{141} 
          \def\wvthree{193} 
          \def\wvfour{244} 
          \def\wvfive{296} 
          \def\wvnmone{347} 
          \def\wvn{399} 

          \cvertex{\wvone}{$v_1$}{vone}
          \cvertex{\wvtwo}{$v_2$}{vtwo}
          \cvertex{\wvthree}{$v_3$}{vthree}
          \cvertex{\wvfour}{$v_4$}{vfour}
          \cvertex{\wvfive}{$v_5$}{vfive}
          \cvertex{\wvnmone}{$v_{n-1}=v_6$}{vnmone}
          \cvertex{\wvn}{$v_n=v_7$}{vn}

          \cedgevv{nonedge}{vone}{vthree}{}
          \cedgevv{mmedge}{vone}{vthree}{}
          \cedgevv{nonedge}{vone}{vnmone}{}
          \cedgevv{mmedge}{vone}{vnmone}{}

          \cedgevv{nonedge}{vone}{vfive}{}

          \cedgevv{nonedge}{vthree}{vfive}{}
          \cedgevv{mmedge}{vthree}{vfive}{}
          \cedgevv{nonedge}{vn}{vfive}{}
          \cedgevv{mmedge}{vn}{vfive}{}
          
          \cedgevv{nonedge}{vtwo}{vfive}{}
          
        \end{scope}
      \end{tikzpicture}
      \caption{Subgraphs of the situations
        in the proof of \behref{beh:2mmr:case-n7}}
      \label{fig:2mmr:case-n7}
    \end{figure}

    \begin{proofbeh}{beh:2mmr:case-n7}      
    If $n=7$, then $v_3$ and $v_{n-1}=v_6$ are the two metamours of $v_1$. As 
    $\deg(v_1) = 2$, the vertices $v_1$ and $v_5$ are not adjacent, and 
    they are also not metamours; see Figure~\ref{fig:2mmr:case-n7}(a).
    Therefore, $\deg(v_5)<n-3=4$. As the only options are
    $\deg(v_5) \in \set{2,n-3}$, we conclude $\deg(v_5)=2$.

    As a result, $v_5$ is only adjacent to $v_4$ and $v_6$, and the vertices $v_3$ 
    and $v_7$ have to be the two metamours of $v_5$;
    see Figure~\ref{fig:2mmr:case-n7}(b). In particular, $v_5$ is not 
    adjacent to $v_2$, and the vertices $v_5$ and $v_2$ are not metamours. This
    implies $\deg(v_2)<n-3=4$ which is a contradiction to $\deg(v_2)=4$.
    Hence, $n=7$ is not possible.
    \end{proofbeh}

    To summarize, in the case that not all vertices of $G$ have
    the same degree in $\set{2,n-3}$,
    $G=C_5$ and $G=\specialH{6}{c}$ are the only possible graphs due 
    to~\behref{beh:2mmr:mixed-degrees-nle7},~\behref{beh:2mmr:case-n5},~\behref{beh:2mmr:case-n6}
     and~\behref{beh:2mmr:case-n7}. This finishes 
    the proof of Lemma~\ref{lem:2mmr:deg-2-and-nm3}.
\end{proof}

Eventually, we can collect all results on $2$-metamour-regular graphs that have 
a connected metamour graph in the following proposition.

\begin{proposition}\label{pro:2mmr:Cn-or-complementCn}
  Let $G$ be a connected $2$-metamour-regular graph with $n$ vertices
  whose metamour graph is the~$C_n$. Then $n \ge 5$ and one of
  \begin{enumerate}[(a)]
  \item $G = C_n$ and $n$ is odd,
  \item $G = \complement{C_n}$, or
  \item
    $G \in \set{
      \specialH{6}{a},
      \specialH{6}{b},
      \specialH{6}{c},
      \specialH{7}{a},
      \specialH{7}{b}
    }$
  \end{enumerate}
  holds.
\end{proposition}

\resetbeh
\begin{proof}
  First we derive two properties 
  of $G$ in the following claims.
  
  \begin{beh}\label{beh:2mmr:n-ge-5}
    We have $n\ge5$.
  \end{beh}
  \begin{proofbeh}{beh:2mmr:n-ge-5}
    The graph $G$ is connected, hence it contains at least $n-1$ edges.
    Furthermore, the graph~$G$ has the metamour graph $C_n$, so the complement 
    of $G$ contains at least $n$ edges.
    As the sum of the number of edges of $G$ and of the complement of $G$ 
    is 
    equal to $\binom{n}{2}$, we have $\binom{n}{2} \ge (n-1) + n$. This is only 
    true for $n \ge 5$. 
  \end{proofbeh}

  \begin{beh}\label{beh:2mmr:not-tree}
    The graph~$G$ is not a tree.
  \end{beh}
  \begin{proofbeh}{beh:2mmr:not-tree}
    Suppose that~$G$ is a tree.
    We first show that the maximum degree of~$G$ is at most~$2$.
    
    Let $v$ be a vertex and $d$ its degree,
    and let $v_1$, \dots, $v_d$ its neighbors.
    Then no vertices of a pair in
    $\set{v_1, \dots, v_d}$ are adjacent, as otherwise we would have a cycle.
    Therefore, the vertices of every such pair are metamours.

    We cannot have $d\ge4$, as otherwise one vertex of $\set{v_1, \dots, v_d}$ would have at
    least three metamours, and this contradicts the 
    $2$-metamour-regularity of the graph~$G$. 
    If $d=3$, then there is a $3$-cycle
    in the metamour graph of~$G$ which contradicts that the metamour
    graph is~$C_n$ and $n\ge5$. Therefore, $d\le2$, and consequently
    we have indeed shown that the maximum degree of $G$ is at most~$2$.
    
    This now implies that $G$ has to be the path graph~$P_n$ which
    is again a contradiction to~$G$ being $2$-metamour-regular,
    as an end vertex of~$P_n$ only has one metamour.
  \end{proofbeh}

  So by~\behref{beh:2mmr:not-tree}, $G$ is not a tree, therefore it
  contains a cycle.
  If $G$ does not contain a cycle of length~$n$, then we
  can apply Lemma~\ref{lem:2mmr:contain-Cn} and conclude
  that $G \in \set{\specialH{6}{a}, \specialH{6}{b}, \specialH{7}{a}}$.
  We are finished in this case.
  
  Otherwise, 
  the graph~$G$ contains a cycle of length~$n$.
  If there is a vertex $v$ of~$G$ with $2 < \deg(v) < n-3$, then 
  we can use Lemma~\ref{lem:2mmr:deg-two-options}, deduce that 
  $G = \specialH{7}{b}$ and the proof is complete in this case.

  Otherwise, every vertex has degree at most~$2$ or at least~$n-3$.
  Due to the fact that $G$ contains a cycle of length $n$, the degree
  of every vertex is at least $2$. Because $G$ is $2$-metamour-regular,
  every vertex is not adjacent to at least two vertices, so the degree of
  every vertex
  is at most $n-3$. This implies that every vertex has degree $2$ or $n-3$.
    
  If all vertices of $G$ have the same degree, then
  Lemma~\ref{lemma:2mmr:all-deg-nMinus3} (for degrees~$n-3$)
  implies $G=\complement{C_n}$
  and Lemma~\ref{lemma:2mmr:all-deg-2} (for degrees~$2$)
  implies $G=C_n$ and $n$ odd.
  Hence, in these cases we are finished with the proof as well.
  
    What is left to consider is the situation that there are two vertices with different 
    degrees in~$G$. This is done in Lemma~\ref{lem:2mmr:deg-2-and-nm3},
    and we conclude $G \in \set{C_5, \specialH{6}{c}}$ in this case.

    This completes the proof.
\end{proof}    

\subsection{Graphs with disconnected metamour graph}

After characterizing all graphs that are $2$-metamour-regular and that have a 
connected metamour graph, we now turn to $2$-metamour-regular graphs that do 
not have a connected metamour graph. In this case either 
statement~\ref{it:general:general} or 
statement~\ref{it:general:exceptional} of 
Theorem~\ref{thm:kmmr:metamour-graph-not-connected} is satisfied. In the 
case of \ref{it:general:general} there is nothing left to do, because 
it provides a characterization. In the other case we determine all 
graphs and capture them in the following proposition. 

\begin{proposition}\label{pro:2mmr:metamour-graph-not-cn}
  Let $G$ be a connected $2$-metamour-regular graph with $n$ vertices.
  Suppose statement~\ref{it:general:exceptional}
  of Theorem~\ref{thm:kmmr:metamour-graph-not-connected} is satisfied.
  Then $n \ge 6$ and one of
    \begin{enumerate}[(a)]
        \item $G = C_n$ and $n$ is even, or 
        \item
        $G \in \set{
            \specialH{4,4}{a},
            \specialH{4,4}{b},
            \specialH{4,4}{c},
            \specialH{4,3}{a},
            \specialH{4,3}{b},
            \specialH{4,3}{c},
            \specialH{4,3}{d},
            \specialH{3,3}{a}, 
            \specialH{3,3}{b},
            \specialH{3,3}{c}, 
            \specialH{3,3}{d},
            \specialH{3,3}{e}            
        }$
    \end{enumerate}
    holds.
\end{proposition}

\resetbeh
\begin{proof}
  Theorem~\ref{thm:kmmr:metamour-graph-not-connected}\ref{it:general:exceptional}
  implies that the metamour graph is not connected.
  First observe that by Observation~\ref{obs:2mr_metamour-graph-consists-of-cycles},
  each connected component of the metamour graph of a
  $2$-metamour-regular graph is a cycle.

    The proof is split into several claims. 
    As a first step, we consider the number of 
    vertices of $G$.
    \begin{beh}\label{beh:2mmr:nle6}
        We have $n \ge 6$.
    \end{beh}
    \begin{proofbeh}{beh:2mmr:nle6}
        As the metamour graph of $G$ is not connected, the metamour graph 
        contains at least two connected components, which are cycles. Each cycle 
        has to contain 
        at least three 
        vertices, so $n \ge 6$.
    \end{proofbeh}
    
    Now we come to the main part of the proof.
    Theorem~\ref{thm:kmmr:metamour-graph-not-connected}\ref{it:general:exceptional} states that the metamour graph consists of exactly
    two connected components;
    we denote these by $M'$ and $M^\ast$.
    Set $G'=G[V(M')]$ and $G^\ast=G[V(M^\ast)]$. Then $G^M$ (as in
    Theorem~\ref{thm:kmmr:metamour-graph-not-connected}) equals
    $G' \cup G^\ast$.
    The definitions of $G'$ and $G^\ast$ are symmetric and we might switch
    the roles of the two without loss of generality during the proof
    and in the statements.

    We introduce the following notion: The \emph{signature $\sigma$}
    of a graph is the tuple of the numbers of vertices of its connected 
    components,
    sorted in descending order.
    If follows from \ref{it:at-most-k-vert} of
    Theorem~\ref{thm:kmmr:metamour-graph-not-connected}\ref{it:general:exceptional}
    that all connected components of $G'$ and $G^\ast$  have at most $2$ 
    vertices. As a consequence, the signatures $\sigma(G')$ and $\sigma(G^\ast)$
    have entries in $\set{1,2}$.
    Note that in case a connected component has two vertices, then these
    vertices are adjacent, i.e., this component equals~$P_2$.

    We perform a case distinction by the signatures of the
    graphs~$G'$ and $G^\ast$; this is stated as the following claims.
    We start with the case that at least one connected component of $G'$ 
    or $G^\ast$ has two vertices.
    
    \begin{beh}\label{beh:2mmr:component-with-2-vertices-implies-4-vertices-total}
      If the first (i.e., largest) entry of $\sigma(G')$ is $2$, then
      either $\sigma(G')=(2,2)$ or $\sigma(G')=(2,1,1)$.
      In the latter case, the two vertices of the two connected components 
      containing
      only one vertex do not have any common neighbor in~$G$.
    \end{beh}
    \begin{proofbeh}{beh:2mmr:component-with-2-vertices-implies-4-vertices-total}
      Suppose we have $G'_1 \in \CC{G'}$ with $\abs{V(G'_1)} = 2$.
        Let $v_1$ be one of the two vertices of $G'_1$. 
        Then $v_1$ is adjacent to the other vertex of $G'_1$, 
        therefore it must have its two metamours in another component of 
        $G'$.
        
        Let us assume that a metamour of $v_1$
        is in a connected component~$G'_2$ of $G'$ that consists of two 
        vertices. Then every
        vertex of $G'_1$ is a metamour of every vertex of 
        $G'_2$ due to \ref{it:metamour-components-contain-metamours} of
        Theorem~\ref{thm:kmmr:metamour-graph-not-connected}\ref{it:general:exceptional}.
        This implies that the four vertices of $G'_1$ and 
        $G'_2$ form a $C_4$ 
        in the metamour graph and consequently that $M'=C_4$. Therefore, 
        $G'$ cannot contain other vertices and $\sigma(G')=(2,2)$.
        
        Let us now assume that the two metamours of $v_1$
        are in different connected components $G'_2$ and $G'_3$ of $G'$ that 
        consist of only one vertex each.
        Then also the other vertex of $G'_1$ 
        is a metamour of the two 
        vertices in $G'_2$ and $G'_3$ due to
        \ref{it:metamour-components-contain-metamours} of
        Theorem~\ref{thm:kmmr:metamour-graph-not-connected}\ref{it:general:exceptional}.
        Hence, these four vertices form again a $C_4$ in the metamour graph
        and consequently $M'=C_4$. As a result, $G'$ 
        cannot contain
        any more vertices, so $\sigma(G')=(2,1,1)$.
        
        If the two vertices of $G'_2$ and $G'_3$
        have a common neighbor, then they are metamours of 
        each other because they are not adjacent. This is a contradiction to 
        the fact that the vertex of 
        $G'_2$ already has two metamours; they are in $G'_1$.
        Hence, the vertices of $G'_2$ and 
        $G'_3$ do not have any common neighbor.
    \end{proofbeh}
    
    By~\behref{beh:2mmr:component-with-2-vertices-implies-4-vertices-total} we 
    can 
    deduce how the graphs $G'$ and $G^\ast$ look like, if one of their 
    connected component contains $2$ vertices.    
     We will now continue by going 
    through all possible combinations of signatures of $G'$ and $G^\ast$
    implied by~\behref{beh:2mmr:component-with-2-vertices-implies-4-vertices-total}.
    In every case, we have to determine 
    which edges between vertices of $G'$ and $G^\ast$ exist
    and which do not exist in order to specify the graph~$G$.

    Due to
    \ref{it:one-egdge-implies-all-edges} of
    Theorem~\ref{thm:kmmr:metamour-graph-not-connected}\ref{it:general:exceptional},
    we know that as soon as there is an edge in~$G$ between a connected
    component of $G'$ and a connected component of $G^\ast$,
    then there are all possible edges between these two components in $G$.
    This implies, now rephrased in the language introduced in Section~\ref{sec:definitions},
    adjacency of two connected components of $G^M$
    is equivalent to complete adjacency of these components.
    Therefore, we equip $G^M=G' \cup G^\ast$ with a graph structure:
    The set~$\CC{G^M}$ is the vertex set and the edge set---we simply
    write it as $E(G^M)$---is determined by the adjacency relation above.
    Note that this graph is bipartite.
    
     \begin{beh}\label{beh:2mmr:2-2-and-2-2-components}
         If $\sigma(G')=(2,2)$ and $\sigma(G^\ast)=(2,2)$, then we have 
         \begin{equation*}
           G \in \set{
             \complement{C_4} \join \complement{C_4},
             \specialH{4,4}{a}}.
         \end{equation*}
       \end{beh}
     \begin{proofbeh}{beh:2mmr:2-2-and-2-2-components}
       Let $\set{G'_1,G'_2} = \CC{G'}$ and $\set{G^\ast_1,G^\ast_2} = \CC{G^\ast}$.
       Each of these four components has size~$2$, therefore, $n=8$.
       This proof is accompanied by Figure~\ref{fig:2mmr:2-2}(a).
       
       The components $G'_1$ and $G'_2$ need a common neighbor in $G^\ast$
       with respect to~$G^M$
         because their vertices are metamours; see the proof of
         \behref{beh:2mmr:component-with-2-vertices-implies-4-vertices-total}.
         So, we assume without loss of generality (by 
         renumbering the connected components of $G^\ast$) that 
         $\edge{G'_1}{G^\ast_1} \in E(G^M)$ and $\edge{G'_2}{G^\ast_1} \in E(G^M)$.
         Furthermore, $G^\ast_2$ needs to be adjacent to at least one connected 
         component of $G'$ because $G$ is connected, so assume without loss of 
         generality (by renumbering the connected components of $G'$) that 
         $\edge{G'_2}{G^\ast_2} \in E(G^M)$. 
         Now there is only the edge between $G'_1$ and $G^\ast_2$ left to consider. If 
         $\edge{G'_1}{G^\ast_2} \in E(G^M)$, then $G = \complement{C_4} \join 
         \complement{C_4}$ and if $\edge{G'_1}{G^\ast_2} \not\in E(G^M)$ then 
         $G=\specialH{4,4}{a}$. This completes the proof.
     \end{proofbeh}      
    
         \begin{figure}
           \centering
           \begin{tikzpicture}
             \begin{scope}[shift={(0,0)}]
               \node at (0,-2) {(a) in proof of \behref{beh:2mmr:2-2-and-2-2-components}};
               
               \node [vertex, label=west:$G'_1$] (a1) at (-0.5,0.5) {};
               \node [vertex, label=west:$G'_2$] (a2) at (-0.5,-0.5) {};
               \node [vertex, label=east:$G^\ast_1$] (b1) at (0.5,0.5) {};
               \node [vertex, label=east:$G^\ast_2$] (b2) at (0.5,-0.5) {};

               \draw [mmedge] (a1) -- (a2);
               \draw [mmedge] (b1) -- (b2);

               \draw [edge] (a1) -- (b1);
               \draw [edge] (a2) -- (b1);
               \draw [edge] (a2) -- (b2);
             \end{scope}
             \begin{scope}[shift={(5,0)}]
               \node at (0,-2) {(b) in proof of \behref{beh:2mmr:2-2-and-2-1-1-components}};

               \node [vertex, label=west:$G'_1$] (a1) at (-0.5,0.5) {};
               \node [vertex, label=west:$G'_2$] (a2) at (-0.5,-0.5) {};
               \node [vertex, label=east:$G^\ast_1$] (b1) at (0.5,1) {};
               \node [vertex, label={[label distance=0.25cm]east:$G^\ast_2$}] (b2) at (0.5,0) {};
               \node [vertex, label=east:$G^\ast_3$] (b3) at (0.5,-1) {};

               \draw [mmedge] (a1) -- (a2);
               \draw [mmedge] (b1) -- (b2);
               \draw [mmedge] (b3) to [bend right] (b1);

               \draw [edge] (a1) -- (b1);
               \draw [nonedge] (a1) -- (b3);
               \draw [edge] (a2) -- (b3);
               \draw [edge] (a2) -- (b1);
               \draw [nonedge] (a2) -- (b2);
               \draw [edge] (a1) -- (b2);
             \end{scope}
           \end{tikzpicture}
           \caption{Subgraphs of the situations
             in the proofs of \behref{beh:2mmr:2-2-and-2-2-components}
           and \behref{beh:2mmr:2-2-and-2-1-1-components}}
           \label{fig:2mmr:2-2}
         \end{figure}

     \begin{beh}\label{beh:2mmr:2-2-and-2-1-1-components}
         If $\sigma(G')=(2,2)$ and $\sigma(G^\ast)=(2,1,1)$, then we have 
         \begin{equation*}
           G = \specialH{4,4}{b}.
         \end{equation*}
     \end{beh}
     \begin{proofbeh}{beh:2mmr:2-2-and-2-1-1-components}
       Let $\set{G'_1,G'_2} = \CC{G'}$ and
       $\set{G^\ast_1,G^\ast_2,G^\ast_3} = \CC{G^\ast}$
       such that $\abs{V(G^\ast_1)}=2$.
      The size of each other component is then determined, specifically
       we have $\abs{V(G'_1)} = 2$, $\abs{V(G'_2)} = 2$,
       $\abs{V(G^\ast_2)}=1$ and $\abs{V(G^\ast_3)}=1$.
       Therefore, $n=8$.
       This proof is accompanied by Figure~\ref{fig:2mmr:2-2}(b).

         The component $G^\ast_1$ need to be adjacent in $G^M$ to at 
         least one 
         connected component of $G'$, so assume without loss of generality (by 
         renumbering the connected components of $G'$) that  
         $\edge{G'_1}{G^\ast_1} \in E(G^M)$. 
         It is not possible that both $G^\ast_2$ and $G^\ast_3$ 
         are adjacent in $G^M$ to $G'_1$ due 
         to~\behref{beh:2mmr:component-with-2-vertices-implies-4-vertices-total},
         so assume without 
         loss of generality (by renumbering $G^\ast_2$ and $G^\ast_3$) that 
         $\edge{G'_1}{G^\ast_3} \not \in E(G^M)$. But $G^\ast_3$ must have a common 
         neighbor in $G'$ with $G^\ast_1$ because their vertices are metamours, 
         so this implies 
         $\edge{G'_2}{G^\ast_3} \in E(G^M)$ and $\edge{G'_2}{G^\ast_1} \in E(G^M)$. 
        Due 
        to~\behref{beh:2mmr:component-with-2-vertices-implies-4-vertices-total},
          this implies that $\edge{G'_2}{G^\ast_2} \not \in E(G^M)$. 
         But as $G^\ast_2$ needs to be adjacent to at least one connected 
         component of $G'$, we find $\edge{G'_1}{G^\ast_2} \in E(G^M)$.
         As a consequence, we obtain $G = \specialH{4,4}{b}$. 
       \end{proofbeh}      
    
     \begin{beh}\label{beh:2mmr:2-1-1-and-2-1-1-components}
         If $\sigma(G')=(2,1,1)$ and $\sigma(G^\ast)=(2,1,1)$, then we have 
         \begin{equation*}
           G = \specialH{4,4}{c}.
         \end{equation*}
        \end{beh}
    \begin{proofbeh}{beh:2mmr:2-1-1-and-2-1-1-components}
      Let $\set{G'_1,G'_2,G'_3} = \CC{G'}$ and
      $\set{G^\ast_1,G^\ast_2,G^\ast_3} = \CC{G^\ast}$
      such that $\abs{V(G'_1)}=2$ and $\abs{V(G^\ast_1)}=2$.
      The size of each other component is then determined.
       We get $n=8$.
        
       Because 
       of~\behref{beh:2mmr:component-with-2-vertices-implies-4-vertices-total},
       not both $G^\ast_2$ and $G^\ast_3$ can be adjacent in~$G^M$ to $G'_1$, so 
       assume without loss of generality (by renumbering $G^\ast_2$ and 
       $G^\ast_3$) that $\edge{G'_1}{G^\ast_3} \not \in E(G^M)$.
       If $\edge{G'_1}{G^\ast_2} \not \in E(G^M)$, then $G^\ast_1$ 
       is the only 
       possible common neighbor in $G^M$ for 
        $G'_1$ and $G'_2$ as well as $G'_1$ and $G'_3$. But then 
        $G'_2$ and $G'_3$ 
        would have the common neighbor $G^\ast_1$ in $G^M$, a contradiction
     to~\behref{beh:2mmr:component-with-2-vertices-implies-4-vertices-total}. 
     As a result, we obtain $\edge{G'_1}{G^\ast_2} \in E(G^M)$.
        
        By symmetric arguments for $G^\ast$, we can assume without loss of 
        generality (by renumbering $G'_2$ and $G'_3$) that 
        $\edge{G'_3}{G^\ast_1} \not \in E(G^M)$ and then deduce that 
        $\edge{G'_2}{G^\ast_1} \in E(G^M)$.
        The current situation is shown in
        Figure~\ref{fig:2mmr:2-1-1-and-2-1-1}(a).
         \begin{figure}
           \centering
           \begin{tikzpicture}
             \begin{scope}[shift={(0,0)}]
               \node at (0,-1.75) {(a)};
               \node [vertex, label=west:$G'_1$] (a1) at (-0.5,1) {};
               \node [vertex, label={[label distance=0.25cm]west:$G'_2$}] (a2) at (-0.5,0) {};
               \node [vertex, label=west:$G'_3$] (a3) at (-0.5,-1) {};
               \node [vertex, label=east:$G^\ast_1$] (b1) at (0.5,1) {};
               \node [vertex, label={[label distance=0.25cm]east:$G^\ast_2$}] (b2) at (0.5,0) {};
               \node [vertex, label=east:$G^\ast_3$] (b3) at (0.5,-1) {};

               \draw [mmedge] (a1) -- (a2);
               \draw [mmedge] (a3) to [bend left] (a1);
               \draw [mmedge] (b1) -- (b2);
               \draw [mmedge] (b3) to [bend right] (b1);

               \draw [nonedge] (a1) -- (b3);
               \draw [edge] (a1) -- (b2);
               \draw [nonedge] (a3) -- (b1);
               \draw [edge] (a2) -- (b1);
             \end{scope}
             \node at (2.5,0) {$\Longrightarrow$};
             \begin{scope}[shift={(5,0)}]
               \node at (0,-1.75) {(b)};
               \node [vertex, label=west:$G'_1$] (a1) at (-0.5,1) {};
               \node [vertex, label={[label distance=0.25cm]west:$G'_2$}] (a2) at (-0.5,0) {};
               \node [vertex, label=west:$G'_3$] (a3) at (-0.5,-1) {};
               \node [vertex, label=east:$G^\ast_1$] (b1) at (0.5,1) {};
               \node [vertex, label={[label distance=0.25cm]east:$G^\ast_2$}] (b2) at (0.5,0) {};
               \node [vertex, label=east:$G^\ast_3$] (b3) at (0.5,-1) {};

               \draw [mmedge] (a1) -- (a2);
               \draw [mmedge] (a3) to [bend left] (a1);
               \draw [mmedge] (b1) -- (b2);
               \draw [mmedge] (b3) to [bend right] (b1);

               \draw [nonedge] (a1) -- (b3);
               \draw [edge] (a1) -- (b2);
               \draw [nonedge] (a3) -- (b1);
               \draw [edge] (a2) -- (b1);

               \draw [edge] (a2) -- (b3);
               \draw [edge] (a3) -- (b2);
               \draw [nonedge] (a2) -- (b2);
               \draw [nonedge] (a3) -- (b3);
               \draw [edge] (a1) -- (b1);
             \end{scope}
           \end{tikzpicture}
           \caption{Subgraphs of the situations
             in the proof of \behref{beh:2mmr:2-1-1-and-2-1-1-components}}
           \label{fig:2mmr:2-1-1-and-2-1-1}
         \end{figure}
        
        As a result, $G'_2$ is the only possible common neighbor in~$G^M$
        of $G^\ast_1$ and $G^\ast_3$, so $\edge{G'_2}{G^\ast_3} \in 
        E(G^M)$. 
        Analogously, we obtain $\edge{G'_3}{G^\ast_2} \in E(G^M)$.
        Now due 
        to~\behref{beh:2mmr:component-with-2-vertices-implies-4-vertices-total}, 
        we can deduce that 
        $\edge{G'_2}{G^\ast_2} \not \in E(G^M)$ and $\edge{G'_3}{G^\ast_3} \not \in E(G^M)$.
        So far $G'_1$ and $G'_2$ do not have a common neighbor, 
        and $G^\ast_1$ is the only possibility for that left, 
        so 
        $\edge{G'_1}{G^\ast_1} \in E(G^M)$;
        see Figure~\ref{fig:2mmr:2-1-1-and-2-1-1}(b).

        This fully determines $G$ and it holds that $G = \specialH{4,4}{c}$. 
    \end{proofbeh}      
    
    With~\behref{beh:2mmr:2-2-and-2-2-components},~\behref{beh:2mmr:2-2-and-2-1-1-components}
    and~\behref{beh:2mmr:2-1-1-and-2-1-1-components} we have considered all 
    cases 
    in which both $G'$ and $G^\ast$ have a connected component consisting of two vertices.
    
    So from now on we can assume that at least one of $G'$ and $G^\ast$ has no 
    connected component consisting of two vertices.
    Next we will deduce a result in the case that the signature of
    one of $G'$ and $G^\ast$ is $(1,\dots,1)$ with at least four entries.
    
    \begin{beh}\label{beh:2mmr:many-components-with-1-vertex-Cn}
      If $\sigma(G')=(1,\dots,1)$, $r$ times with $r\ge4$, then we have 
      \begin{equation*}
        G=C_n
      \end{equation*}
      and $n$ is even.
    \end{beh}
    \begin{proofbeh}{beh:2mmr:many-components-with-1-vertex-Cn}
      Let $\set{G'_1,\dots,G'_r} = \CC{G'}$.
        Let without loss of generality (by renumbering the connected components 
        of $G'$) the vertices of $G'_{i-1}$ and $G'_{i+1}$ be the two 
        metamours of the vertex of $G'_i$.
        Note that we take the indices modulo $r$ and that
        we keep doing this for the remaining proof.
        
        Let $i \in \set{1,\dots,r}$.
        Then clearly 
        $G'_i$ and $G'_{i+1}$ have a common neighbor $G^\ast_i \in 
        \CC{G^\ast}$.
        If $G^\ast_i$ is adjacent to any other connected component of $\CC{G'}$,
        then the vertex of this component together with the two
        vertices of $G'_i$ and $G'_{i+1}$
        form a $C_3$ in the metamour 
        graph. This is a contradiction, because $M'$ is $C_r$ with $r\ge4$.
        Hence, $G^\ast_i$ is adjacent in $G$ to only $G'_i$ and $G'_{i+1}$ of 
        $G'$.
        In particular, this implies that the components $G^\ast_1$, \dots, $G^\ast_r$ are 
        pairwise disjoint due to \ref{it:one-egdge-implies-all-edges} of
        Theorem~\ref{thm:kmmr:metamour-graph-not-connected}\ref{it:general:exceptional}.
        
        Now because of the common neighbors, the vertices of $G^\ast_i$ have the
        vertices of $G^\ast_{i-1}$ and $G^\ast_{i+1}$ as metamours.
        Therefore, as we are $2$-metamour-regular, every component $G^\ast_i$
        consists of exactly one vertex.
        As a consequence, $G^\ast_1$, \dots, $G^\ast_r$ lead to a $C_r$ in the metamour graph
        of $G$, specifically $M^\ast=C_r$.
        It 
        is easy to see that the vertices of
        $G'_1$, $G^\ast_1$, $G'_2$, $G^\ast_2$, \dots, $G'_r$, $G^\ast_r$ 
        form a $C_{2r}$. As we have ruled out all other possible edges, this 
        implies that $G=C_{2r}$. Hence, $n=2r$ and $G=C_n$ for $n$ even.
    \end{proofbeh}   

    In~\behref{beh:2mmr:many-components-with-1-vertex-Cn}, we have
    dealt with signatures $(1,\dots,1)$ of length at least~$4$.
    We will consider~$(1,1,1)$ below. There cannot be fewer than
    three connected components of only single vertices because
    each connected component of the metamour graph
    is a cycle and therefore has at least $3$ vertices.
    
    So what is left to consider are the two cases that $G'$ contains a 
    connected component with two vertices and $G^\ast$ has three isolated 
    vertices and the case that both $G'$ and $G^\ast$ have three isolated vertices.
    We consider these cases in the following claims.

    \begin{beh}\label{beh:2mmr:2-2-and-1-1-1}
      If $\sigma(G')=(2,2)$ and $\sigma(G^\ast)=(1,1,1)$, then we have 
      \begin{equation*}
        G \in \set{\complement{C_4} \join \complement{C_3}, \specialH{4,3}{a}, 
        \specialH{4,3}{b}}.
    \end{equation*}
   \end{beh}
    \begin{proofbeh}{beh:2mmr:2-2-and-1-1-1}
      Let $\set{G'_1,G'_2} = \CC{G'}$ and
      $\set{G^\ast_1,G^\ast_2,G^\ast_3} = \CC{G^\ast}$.
      The size of each component is then determined, and we get $n=7$.
        
        The single-vertex components $G^\ast_1$ and $G^\ast_2$ need a 
        common neighbor, as well as $G^\ast_2$ and $G^\ast_3$, and
        $G^\ast_1$ and $G^\ast_3$. At least two of these common neighbors 
        are from the same connected component of $G'$, because $G'$ has only 
        two connected components.
        Let without loss of generality 
        (by renumbering $G'_1$ and $G'_2$) this connected component be $G'_1$. 
        As a result, every component
        $G^\ast_1$, $G^\ast_2$ and $G^\ast_3$ is
        adjacent to $G'_1$, and
        $\edge{G'_1}{G^\ast_1}$, $\edge{G'_1}{G^\ast_2}$ and
        $\edge{G'_1}{G^\ast_2} \in E(G^M)$.
        
        The connected component $G'_2$ has to be adjacent to some component of 
        $G^\ast$ because $G$ is connected, so assume without loss of generality 
        (by renumbering $G^\ast_1$, $G^\ast_2$ and $G^\ast_3$) that 
        $\edge{G'_2}{G^\ast_1} \in E(G^M)$. 
         The current situation is shown in Figure~\ref{fig:2mmr:1-1-1}(a).
        
        Now if both $\edge{G'_2}{G^\ast_2}$ and $\edge{G'_2}{G^\ast_3}$ are in $E(G^M)$, 
        then $G$ is 
        fully determined and $G=\complement{C_4} \join \complement{C_3}$.
        If only one of  $\edge{G'_2}{G^\ast_2}$  and $\edge{G'_2}{G^\ast_3}$ is in 
        $E(G^M)$, then we have $G =  \specialH{4,3}{a}$, and if none of 
        $\edge{G'_2}{G^\ast_2}$ and $\edge{G'_2}{G^\ast_3}$ is in $E(G^M)$, then 
        $G=\specialH{4,3}{b}$.
        As one of these three settings has to occur, this proof is completed.
    \end{proofbeh}
    
         \begin{figure}
           \centering
           \begin{tikzpicture}
             \begin{scope}[shift={(0,0)}]
               \node at (0,-2) {(a) in proof of \behref{beh:2mmr:2-2-and-1-1-1}};

               \node [vertex, label=west:$G'_1$] (a1) at (-0.5,0.5) {};
               \node [vertex, label=west:$G'_2$] (a2) at (-0.5,-0.5) {};
               \node [vertex, label=east:$G^\ast_1$] (b1) at (0.5,1) {};
               \node [vertex, label={[label distance=0.25cm]east:$G^\ast_2$}] (b2) at (0.5,0) {};
               \node [vertex, label=east:$G^\ast_3$] (b3) at (0.5,-1) {};

               \draw [mmedge] (a1) -- (a2);
               \draw [mmedge] (b1) -- (b2) -- (b3);
               \draw [mmedge] (b3) to [bend right] (b1);

               \draw [edge] (a1) -- (b1);
               \draw [edge] (a1) -- (b2);
               \draw [edge] (a1) -- (b3);
               \draw [edge] (a2) -- (b1);
             \end{scope}
             \begin{scope}[shift={(5,0)}]
               \node at (0,-2) {(b) in proof of \behref{beh:2mmr:2-1-1-and-1-1-1}};

               \node [vertex, label=west:$G'_1$] (a1) at (-0.5,1) {};
               \node [vertex, label={[label distance=0.25cm]west:$G'_2$}] (a2) at (-0.5,0) {};
               \node [vertex, label=west:$G'_3$] (a3) at (-0.5,-1) {};
               \node [vertex, label=east:$G^\ast_1$] (b1) at (0.5,1) {};
               \node [vertex, label={[label distance=0.25cm]east:$G^\ast_2$}] (b2) at (0.5,0) {};
               \node [vertex, label=east:$G^\ast_3$] (b3) at (0.5,-1) {};

               \draw [mmedge] (a1) -- (a2);
               \draw [mmedge] (a3) to [bend left] (a1);
               \draw [mmedge] (b1) -- (b2) -- (b3);
               \draw [mmedge] (b3) to [bend right] (b1);

               \draw [edge] (a2) -- (b2);
               \draw [edge] (a3) -- (b3);
               \draw [nonedge] (a2) -- (b3);
               \draw [nonedge] (a3) -- (b2);
               \draw [edge] (a1) -- (b2);
               \draw [edge] (a1) -- (b3);
               \draw [edge] (a1) -- (b1);
             \end{scope}
           \end{tikzpicture}
           \caption{Subgraphs of the situations
             in the proofs of \behref{beh:2mmr:2-2-and-1-1-1}
             and \behref{beh:2mmr:2-1-1-and-1-1-1}}
           \label{fig:2mmr:1-1-1}
         \end{figure}
    
    \begin{beh}\label{beh:2mmr:2-1-1-and-1-1-1}
     If $\sigma(G')=(2,1,1)$ and $\sigma(G^\ast)=(1,1,1)$, then we have 
     \begin{equation*}
       G \in \set{\specialH{4,3}{c}, \specialH{4,3}{d}}.
     \end{equation*}
    \end{beh}
    \begin{proofbeh}{beh:2mmr:2-1-1-and-1-1-1}
      Let $\set{G'_1,G'_2,G'_3} = \CC{G'}$ and
      $\set{G^\ast_1,G^\ast_2,G^\ast_3} = \CC{G^\ast}$
      such that $\abs{V(G'_1)}=2$.
      The size of each other component is then determined.
      We get $n=7$.

      As $G$ is connected, let us assume without loss of generality
      (by renumbering $G^\ast_1$, $G^\ast_2$ and $G^\ast_3$)
      that $\edge{G'_2}{G^\ast_2}$ and $\edge{G'_3}{G^\ast_3} \in E(G^M)$.
         Because 
         of~\behref{beh:2mmr:component-with-2-vertices-implies-4-vertices-total},
         the vertices of $G'_2$ and $G'_3$ cannot have a common neighbor, 
         therefore $\edge{G'_2}{G^\ast_3}$ and $\edge{G'_3}{G^\ast_2} \not\in E(G^M)$
         holds.

         The only choice for a common neighbor of $G^\ast_2$ and $G^\ast_3$
         is $G'_1$, therefore
         $\edge{G'_1}{G^\ast_2}$ and $\edge{G'_1}{G^\ast_3} \in E(G^M)$.

         If $G'_1$ is not adjacent in~$G^M$ to $G^\ast_1$, then
         we have to have
         $\edge{G'_2}{G^\ast_1}$ and $\edge{G'_3}{G^\ast_1} \in E(G^M)$
         so that $G^\ast_1$ and $G^\ast_2$ as well as $G^\ast_1$ and $G^\ast_3$
         have a common neighbor. But then $G'_2$ and $G'_3$ get the
         common neighbor~$G^\ast_1$ which is a contradiction
         to~\behref{beh:2mmr:component-with-2-vertices-implies-4-vertices-total}.
         Therefore, we have $\edge{G'_1}{G^\ast_1} \in E(G^M)$.
         The current situation is shown in Figure~\ref{fig:2mmr:1-1-1}(b).

         Furthermore, not both of $G'_2$ and $G'_3$ can be adjacent to the 
         vertex of $G^\ast_1$, so assume without loss of generality (by 
         renumbering $G'_2$ and $G'_3$) that $\edge{G'_3}{G^\ast_1} \not \in E(G^M)$.
         
         Now if $\edge{G'_2}{G^\ast_1} \in E(G^M)$, then $G=\specialH{4,3}{c}$.
         If $\edge{G'_2}{G^\ast_1} \not \in E(G^M)$, then $G=\specialH{4,3}{d}$.
         This completes the proof.
     \end{proofbeh}
     
     Now the only case left to consider is that both $G'$ and $G^\ast$ contain 
     three isolated vertices.

    \begin{beh}\label{beh:2mmr:1-1-1-and-1-1-1}
        If $\sigma(G')=(1,1,1)$ and $\sigma(G^\ast)=(1,1,1)$, then we have 
        \begin{equation*}
          G \in \set{C_6, \complement{C_3} \join \complement{C_3}, 
            \specialH{3,3}{a}, \specialH{3,3}{b},
            \specialH{3,3}{c}, \specialH{3,3}{d},
            \specialH{3,3}{e}}.
        \end{equation*}
    \end{beh}
    \begin{proofbeh}{beh:2mmr:1-1-1-and-1-1-1}
      Let $\set{G'_1,G'_2,G'_3} = \CC{G'}$ and
      $\set{G^\ast_1,G^\ast_2,G^\ast_3} = \CC{G^\ast}$.
      Then clearly $n=6$.
      
        We first consider the case that every connected component has at most 
        degree $2$ in $G^M$. If a vertex has degree $1$, then in order to have 
        a common neighbor with its both metamours, the component it is adjacent to has 
        to have degree $3$, a contradiction. Therefore, every 
        vertex has degree $2$. Due to the fact that $G$ 
        is connected, this implies that $G=C_6$, so in this case we are done.

        Now assume there is at least one component that has degree $3$ in $G^M$. Let without 
        loss of generality (by switching $G'$ and $G^\ast$ and by renumbering 
        the connected components of $G'$) this component be $G'_1$.
        Then we have 
        $\edge{G'_1}{G^\ast_1}$, $\edge{G'_1}{G^\ast_2}$ and
        $\edge{G'_1}{G^\ast_3} \in E(G^M)$.
                
        If every component of $G'$ has degree $3$ in $G^M$, then $G = \complement{C_3} \join 
        \complement{C_3}$, so also in this case we are done.
        Hence, we can assume without loss of generality (by renumbering $G'_2$ 
        and $G'_3$ and by renumbering the components of $G^\ast$) that 
         $\edge{G'_3}{G^\ast_3} \not \in E(G^M)$. 
         
         Now $G'_3$ and $G'_2$ need a common neighbor. This cannot be $G^\ast_3$.  
         Assume without loss of generality (by renumbering $G^\ast_1$ and 
         $G^\ast_2$) that the common neighbor is $G^\ast_2$. Then 
         this implies that
         $\edge{G'_2}{G^\ast_2}$ and $\edge{G'_3}{G^\ast_2} \in E(G^M)$.
         The current situation is shown in 
         Figure~\ref{fig:2mmr:1-1-1-and-1-1-1}.
         
         \begin{figure}
           \centering
           \begin{tikzpicture}
             \begin{scope}[shift={(0,0)}]
               \node [vertex, label=west:$G'_1$] (a1) at (-0.5,1) {};
               \node [vertex, label={[label distance=0.25cm]west:$G'_2$}] (a2) at (-0.5,0) {};
               \node [vertex, label=west:$G'_3$] (a3) at (-0.5,-1) {};
               \node [vertex, label=east:$G^\ast_1$] (b1) at (0.5,1) {};
               \node [vertex, label={[label distance=0.25cm]east:$G^\ast_2$}] (b2) at (0.5,0) {};
               \node [vertex, label=east:$G^\ast_3$] (b3) at (0.5,-1) {};

               \draw [mmedge] (a1) -- (a2) -- (a3);
               \draw [mmedge] (a3) to [bend left] (a1);
               \draw [mmedge] (b1) -- (b2) -- (b3);
               \draw [mmedge] (b3) to [bend right] (b1);

               \draw [edge] (a1) -- (b1);
               \draw [edge] (a1) -- (b2);
               \draw [edge] (a1) -- (b3);
               \draw [nonedge] (a3) -- (b3);
               \draw [edge] (a2) -- (b2);
               \draw [edge] (a3) -- (b2);
             \end{scope}
             \node at (2.5,0) {$=$};
             \begin{scope}[shift={(5, 0)}]
               \node [vertex, label={[label distance=-0.15cm]north east:$G'_2$}] (ru) at (45:1) {};
               \node [vertex, label={[label distance=-0.15cm]north west:$G^\ast_3$}] (lu) at (135:1) {};
               \node [vertex, label={[label distance=-0.15cm]south west:$G^\ast_1$}] (ll) at (225:1) {};
               \node [vertex, label={[label distance=-0.15cm]south east:$G'_3$}] (rl) at (315:1) {};
               \node [vertex, label={[label distance=-0.1cm]west:$G'_1$}] (lc) at (-0.3,0) {};
               \node [vertex, label={[label distance=-0.1cm]east:$G^\ast_2$}] (rc) at (0.3,0) {};

               \draw [edge] (lc) -- (rc);
               \draw [edge] (ru) -- (rc);
               \draw [edge] (rl) -- (rc);
               \draw [edge] (lu) -- (lc);
               \draw [edge] (ll) -- (lc);

               \draw [mmedge] (ru) to [bend right] (lc);
               \draw [mmedge] (lc) to [bend right] (rl);
               \draw [mmedge] (rl) to [bend right=45] (ru);
               
               \draw [mmedge] (lu) to [bend left] (rc);
               \draw [mmedge] (rc) to [bend left] (ll);
               \draw [mmedge] (ll) to [bend left=45] (lu);
             \end{scope}
           \end{tikzpicture}
           \caption{Subgraphs of the situation
             in the proof of \behref{beh:2mmr:1-1-1-and-1-1-1}}
           \label{fig:2mmr:1-1-1-and-1-1-1}
         \end{figure}
         The potential edges, which status is still 
         undetermined, are $\edge{G'_2}{G^\ast_1}$, $\edge{G'_2}{G^\ast_3}$ and 
         $\edge{G'_3}{G^\ast_1}$.
         At this stage, if each of these pairs is a non-edge in~$G^M$, then
         it is easy to see that 
         whenever two vertices should be metamours they are metamours,
         to be precise
         all components of $G'$ have the common neighbor $G^\ast_2$
         and all components of 
         $G^\ast$ have the common neighbor $G'_1$.
         Therefore, we have enough edges in~$G^M$, so that additional 
         edges between components of $G'$ and $G^\ast$ can be included without 
         interfering with the metamours.
         
         Now we first consider all cases where $\edge{G'_2}{G^\ast_3} \in E(G^M)$.
         In this case 
         if both of $\edge{G'_2}{G^\ast_1}$ and $\edge{G'_3}{G^\ast_1}$ are in 
         $E(G^M)$, then $G=\specialH{3,3}{a}$.
         If $\edge{G'_2}{G^\ast_1} \not \in E(G^M)$ and $\edge{G'_3}{G^\ast_1} \in E(G^M)$,
         then $G=\specialH{3,3}{b}$.         
         If $\edge{G'_2}{G^\ast_1} \in E(G^M)$ and $\edge{G'_3}{G^\ast_1} \not \in E(G^M)$,
         then $G=\specialH{3,3}{c}$.
         And finally, if $\edge{G'_2}{G^\ast_1} \not \in E(G^M)$ and 
         $\edge{G'_3}{G^\ast_1} \not \in E(G^M)$,
         then $G=\specialH{3,3}{d}$.        
         
         Next we consider the case where  $\edge{G'_2}{G^\ast_3} \not \in E(G^M)$.
         If in this case both of $\edge{G'_2}{G^\ast_1}$ and $\edge{G'_3}{G^\ast_1}$ 
         are in 
         $E(G^M)$, then $G=\specialH{3,3}{c}$.
         If one of $\edge{G'_2}{G^\ast_1}$ and $\edge{G'_3}{G^\ast_1}$ is in $E(G^M)$, 
         then $G=\specialH{3,3}{d}$.
         If none of  $\edge{G'_2}{G^\ast_1}$ and $\edge{G'_3}{G^\ast_1}$ is in $E(G^M)$, 
         then $G=\specialH{3,3}{e}$.
         
         Eventually, we have considered all cases and proven what we wanted to 
         show.
    \end{proofbeh}    

    Finally, we are finished in all cases and therefore the proof of 
    Proposition~\ref{pro:2mmr:metamour-graph-not-cn} is complete.
\end{proof}

\subsection{Assembling results \& other proofs}

With all results above,
we are now able to prove the main theorem of this section which provides a 
characterization of $2$-metamour-regular graphs.

\begin{proof}[Proof of Theorem~\ref{thm:2mriff}]
  Let $G$ be $2$-metamour-regular and $M$ its metamour graph. We apply
  Theorem~\ref{thm:kmmr:metamour-graph-not-connected} with $k=2$.
  This leads us to one of three cases.
  
  Case~\ref{it:general:general} of
  Theorem~\ref{thm:kmmr:metamour-graph-not-connected}
  gives
  \begin{equation*}
    G = \complement{M_1} \join \dots \join \complement{M_t}
  \end{equation*}
  with $\set{M_1,\dots,M_t} = \CC{M}$ and $t \ge 2$.
  By Observation~\ref{obs:2mr_metamour-graph-consists-of-cycles}
  every connected component~$M_i$, $i\in\set{1,\dots,t}$, is a
  cycle~$C_{n_i}$. This results in \ref{it:thm-2mmr:general}
  of Theorem~\ref{thm:2mriff} for $t\ge2$.

  If we are in case~\ref{it:general:mm-connected} of
  Theorem~\ref{thm:kmmr:metamour-graph-not-connected},
  then the metamour graph~$M$ is connected and we apply
  Proposition~\ref{pro:2mmr:Cn-or-complementCn}.
  If we are in case~\ref{it:general:exceptional} of
  Theorem~\ref{thm:kmmr:metamour-graph-not-connected}, then the metamour
  graph consists of exactly two connected components and
  we can apply Proposition~\ref{pro:2mmr:metamour-graph-not-cn}.
  Collecting all graphs coming from these two propositions yields
  the remaining graphs of \ref{it:thm-2mmr:general},
  \ref{it:thm-2mmr:Cn} and \ref{it:thm-2mmr:exceptional}
  of Theorem~\ref{thm:2mriff}.

  For the other direction,
  Proposition~\ref{pro:join-of-complements-is-kmr} implies that
  the graph~$G$ in \ref{it:thm-2mmr:general} of Theorem~\ref{thm:2mriff}
  is a $2$-metamour-regular graph for~$t \ge 2$.
    Furthermore, it is easy to check that all other mentioned graphs are 
    $2$-metamour-regular, which proves this side of the 
    equivalence and completes the proof.
\end{proof}

Finally we are able to prove the following corollaries of 
Theorem~\ref{thm:2mriff}.

\begin{proof}[Proof of Corollary~\ref{cor:2mm:collect-non-standard-cases}]
    The result is an immediate consequence of Theorem~\ref{thm:2mriff}.
\end{proof}

\begin{proof}[Proof of Corollary~\ref{cor:2mm:2mm-iff-regular}]
    Corollary~\ref{cor:2mm:collect-non-standard-cases} provides a 
    characterization of all $2$-metamour-regular graphs with $n \ge 9$ 
    vertices. It is easy to see that all of these graphs are either $2$-regular 
    (in the case of $C_n$) or $(n-3)$-regular. This proves one direction of the 
    equivalence.
    
    For the other direction first consider a connected $2$-regular graph on $n$ 
    vertices. Clearly, this graph equals $C_n$, therefore this graph is 
    $2$-metamour-regular. If a connected graph is $(n-3)$-regular, then its 
    complement~$\complement{G}$ is a $2$-regular graph. As a result, each 
    connected component of $\complement{G}$ is a cycle graph. Let 
    $C_{n_1}$, \dots, $C_{n_t}$ be the connected components of 
    $\complement{G}$. 
    It is easy to see that then $n_i \ge 3$ and $n=n_1+\dots+n_t$ hold. In 
    consequence, $G = \complement{C_{n_1}} \join \dots \join 
    \complement{C_{n_t}}$ holds and therefore $G$ is $2$-metamour-regular. This 
    completes the proof.
\end{proof}

\begin{proof}[Proof of Corollary~\ref{cor:2mm:mm-ne-complement}]
    The statement of the corollary is a direct consequence of Theorem~\ref{thm:2mriff}.
\end{proof}

\begin{proof}[Proof of Corollary~\ref{cor:2mmr:counting}]
  We use the characterization provided by Theorem~\ref{thm:2mriff}.
  So let us consider $2$-metamour-regular graphs.
  Such a graph has at least~$n\ge5$ vertices.
  
  In case~\ref{it:thm-2mmr:general}, there is one graph per integer partition
  of~$n$ into a sum, where each summand is at least~$3$. Note that
  the graph operator~$\join$ is commutative which coincides with
  the irrelevance of the order of the summands of the sum. There
  are $p_3(n)$ many such partitions.

  Case~\ref{it:thm-2mmr:Cn} gives exactly one graph for
  each~$n\ge5$. The graph~$C_5$ is counted
  in both~\ref{it:thm-2mmr:general} and~\ref{it:thm-2mmr:Cn}; see first
  item of Remark~\ref{rem:thm-2mmr}.
  Case~\ref{it:thm-2mmr:exceptional} brings in additionally $8$ graphs
  for $n=6$, $6$ graphs for $n=7$ and $3$ graphs for $n=8$.

  In total, this gives the claimed numbers.
\end{proof}

This completes all proofs of the present paper.

\section{Conclusions \& open problems}
\label{sec:conclusions}

In this paper we have introduced the metamour graph~$M$ of a graph~$G$: The set 
of vertices of $M$ is the set of vertices of $G$ and two 
vertices are adjacent in $M$ if and only if they are at distance $2$ in~$G$,
i.e., they are metamours. This definition is motivated by polyamorous 
relationships, where two persons are metamours if they have a
relationship with a common partner, but are not in a relationship 
themselves. 

We focused on $k$-metamour-regular graphs, i.e., graphs in which every vertex 
has exactly~$k$ metamours.
We presented a generic construction to obtain $k$-metamour-regular graphs
from $k$-regular graphs for an arbitrary $k \ge 0$.
Furthermore, in our main results, we provided a full characterization of 
all $k$-metamour-regular graphs for each $k \in \set{0,1,2}$. These 
characterizations revealed that with a few exceptions, all graphs
come from the generic construction. In particular,
\begin{itemize}
\item for $k=0$ every $k$-metamour-regular graph is 
  obtained by the generic construction.
\item For $k=1$ there is only one exceptional graph that is 
  $k$-metamour-regular and not obtained by the generic 
  construction.
\item In the case of $k=2$ there are 17 exceptional graphs with
  at most $8$ vertices and a family of graphs,
  one for each number of vertices at least $6$, that 
  are $2$-metamour-regular and cannot be created by the generic construction.
\end{itemize}
Additionally,
we were able to characterize all graphs where every vertex has
at most one metamour and give properties of the structure of graphs where every
vertex has at most~$k$ metamours for arbitrary $k \ge 0$.
Every characterization is accompanied by counting for each number of
vertices how many unlabeled graphs there are.

The obvious unanswered question is clearly the following.
\begin{question}
  What is a characterization of 
  $k$-metamour-regular graphs for each $k \ge 3$?
\end{question}
This is of particular interest for
$k = 3$. As our generic construction yields $k$-metamour-regular graphs for every 
$k \ge 0$, we clearly already have determined a lot of $3$-metamour-regular 
graphs. It would, however, be lovely to determine all remaining graphs. 
Another interesting question is about fixed maximum metamour-degree.
\begin{question}
  What is a characterization of all graphs that have maximum 
  metamour-degree~$k$?
\end{question}
We have answered this question for $k\in\set{0,1}$ and 
would be delighted to know the answer in general,
but as first steps specifically for $k=2$ and $k=3$.

It would also be interesting to find some structure in the graphs that are 
$k$-metamour-regular and cannot be obtained with our generic construction.
In particular, we ask the following.
\begin{question}
  Is it possible to give properties (necessary or sufficient)
  of the exceptional graphs or graph classes?
\end{question}

When dealing with metamour graphs, one question to ask is whether it is 
possible to characterize all graphs whose metamour graph has a certain property.
In the present paper we have started to give an answer for the 
feature that the metamour graph is $k$-regular. But what about other
graph classes?
Of course it would be interesting to answer the following questions.
\begin{question}
  Is it possible to characterize all graphs whose metamour graph is
  in some graph class like
  planar, 
  bipartite, Eulerian or Hamiltonian graphs or like graphs of a 
  certain diameter, girth, stability number or chromatic number?
\end{question}

Another question of interest concerns constructing graphs, namely
given a graph $M$, is there a 
graph $G$ such that $M$ is the metamour graph of $G$?
If $M$ is not connected, then
the answer is easy and also provided in this paper,
namely $G = \complement{M}$ is such a graph. However, if 
$M$ is connected this question is still open and an answer more complicated.
This give rise to the following question.
\begin{question}
  What is a characterization of the class of graphs with the property that
  each graph in this class is the metamour graph of some graph?
\end{question}

Motivated by~\cite{Zypen:2019:graphs-distance-2} we ask the following.
\begin{question}
  What is a characterization of the class of graphs, where
  every graph is isomorphic to its metamour graph?
\end{question}

Going into another direction, one can also think about random graphs
like the graphs from the Erd\H{o}s--R\'enyi model $G(n,p)$.
\begin{question}
  Given a random graph of $G(n,p)$, which 
  properties does its metamour graph have?
  Is there a critical value for $p$ (depending on~$n$) such 
  that the metamour graph is connected?
\end{question}

In enumerative and probabilistic combinatorics the following question arise.
\begin{question}
  Given a random graph model, for example that all graphs
  with the same number of vertices are equally likely,
  what is the expected value of the  
  metamour-degree? What about its distribution?
\end{question}
Most of the results and open questions focus on the number of vertices
of the graph and metamour graph respectively, as these two numbers match.
But it would be interesting to know how the number of 
edges of the metamour graph of a graph relates to the number of edges in this graph. Specifically, we ask the following questions.
\begin{question}
  Given a graph $G$ with $m$ edges, in which range can the number of
  edges of the metamour graph of $G$ be?
\end{question}
\begin{question}
  What is the distribution of the number of edges of the 
  metamour graph over all 
  possible graphs with $m$ edges?
\end{question}

{\footnotesize
\newcommand\MR[1]{}
\bibliography{bib/cheub,poly-bib/poly}
\bibliographystyle{bibstyle/amsplainurl}
}

\end{document}

%% file: figures-2mmr.tex
\begin{figure}
  \centering
  \begin{tikzpicture}
    \begin{scope}[shift={(0, 0)}]
    \begin{scope}[shift={(0, 0)}]
      \node at (0,0) {$C_5$};

      \node [vertex] (a5) at (90:1) {};
      \node [vertex] (b5) at (162:1) {};
      \node [vertex] (c5) at (234:1) {};
      \node [vertex] (d5) at (306:1) {};
      \node [vertex] (e5) at (18:1) {};

      \draw [edge] (0,0) circle (1);
      \draw [mmedge] (a5) -- (c5) -- (e5) -- (b5) -- (d5) -- (a5);
    \end{scope}
    \fill [gray!25] (1.5,0) circle (0.35);
    \node at (1.5, 0) {$=$};
    \begin{scope}[shift={(3, 0)}]
      \fill [gray!25] (0,0) circle (1.4);
      \fill [gray!25] (1.5,0) circle (0.35);
      \node at (1.5,0) {$\complement{C_5}$};

      \draw [mmedge] (0,0) circle (1);
      
      \node [vertex] (a5) at (90:1) {};
      \node [vertex] (b5) at (162:1) {};
      \node [vertex] (c5) at (234:1) {};
      \node [vertex] (d5) at (306:1) {};
      \node [vertex] (e5) at (18:1) {};

      \draw [edge] (a5) -- (c5) -- (e5) -- (b5) -- (d5) -- (a5);
    \end{scope}
    \end{scope}
  \end{tikzpicture}
  \caption{The only graph of order~$5$
    (\colorbox{gray!25}{+ one differently drawn copy})
    where each vertex has exactly $2$ metamours}
  \label{fig:graphs-5}
\end{figure}

\begin{figure}
  \centering
  \begin{tikzpicture}
    \begin{scope}[shift={(1.5, 3)}]
    \begin{scope}[shift={(6, 0)}]
      \node at (0,0) {$C_6$};

      \node [vertex] (a) at (90:1) {};
      \node [vertex] (b) at (150:1) {};
      \node [vertex] (c) at (210:1) {};
      \node [vertex] (d) at (270:1) {};
      \node [vertex] (e) at (330:1) {};
      \node [vertex] (f) at (390:1) {};

      \draw [edge] (0,0) circle (1);

      \draw [mmedge] (a) -- (c) -- (e) -- (a);
      \draw [mmedge] (b) -- (d) -- (f) -- (b);
    \end{scope}
    \begin{scope}[shift={(3, 0)}]
      \fill [gray!25] (0,0) circle (1.4);
      \node [vertex] (a) at (90:1) {};
      \node [vertex] (b) at (150:1) {};
      \node [vertex] (c) at (210:1) {};
      \node [vertex] (d) at (270:1) {};
      \node [vertex] (e) at (330:1) {};
      \node [vertex] (f) at (390:1) {};

      \draw [edge] (a) -- (b) -- (c) -- (d) -- (e) -- (f) -- (a);
      \draw [edge] (a) -- (d);
      \draw [edge] (b) -- (e);
      \draw [edge] (c) -- (f);

      \draw [mmedge] (a) -- (c) -- (e) -- (a);
      \draw [mmedge] (b) -- (d) -- (f) -- (b);
    \end{scope}
    \fill [gray!25] (1.5,0) circle (0.35);
    \node at (1.5, 0) {$=$};
    \begin{scope}[shift={(0, 0)}]
      \node at (0,-1.25) {$K_{3,3} = \complement{C_3} \join \complement{C_3}$};

      \node [vertex] (a) at (-1,-0.5) {};
      \node [vertex] (b) at (-1,0.5) {};
      \node [vertex] (c) at (0,-0.5) {};
      \node [vertex] (d) at (0,0.5) {};
      \node [vertex] (e) at (1,-0.5) {};
      \node [vertex] (f) at (1,0.5) {};

      \draw [edge] (a) -- (b) -- (c) -- (d) -- (e) -- (f) -- (a);
      \draw [edge] (a) -- (d);
      \draw [edge] (b) -- (e);
      \draw [edge] (c) -- (f);

      \draw [mmedge] (a) -- (c) -- (e) to [bend left] (a);
      \draw [mmedge] (b) -- (d) -- (f) to [bend right] (b);
    \end{scope}
    \end{scope}
    \begin{scope}[shift={(6, -3)}]
    \begin{scope}[shift={(6, 0)}]
      \node at (0,-1.1) {$\specialH{3,3}{e}$};

      \node [vertex] (ru) at (45:1) {};
      \node [vertex] (lu) at (135:1) {};
      \node [vertex] (ll) at (225:1) {};
      \node [vertex] (rl) at (315:1) {};
      \node [vertex] (lc) at (-0.3,0) {};
      \node [vertex] (rc) at (0.3,0) {};

      \draw [edge] (lc) -- (rc);
      \draw [edge] (ru) -- (rc);
      \draw [edge] (rl) -- (rc);
      \draw [edge] (lu) -- (lc);
      \draw [edge] (ll) -- (lc);

      \draw [mmedge] (ru) to [bend right] (lc);
      \draw [mmedge] (lc) to [bend right] (rl);
      \draw [mmedge] (rl) to [bend right] (ru);

      \draw [mmedge] (lu) to [bend left] (rc);
      \draw [mmedge] (rc) to [bend left] (ll);
      \draw [mmedge] (ll) to [bend left] (lu);
     \end{scope}
     \begin{scope}[shift={(3, 0)}]
      \node at (0,-1.1) {$\specialH{3,3}{d}$};

      \node [vertex] (ru) at (45:1) {};
      \node [vertex] (lu) at (135:1) {};
      \node [vertex] (ll) at (225:1) {};
      \node [vertex] (rl) at (315:1) {};
      \node [vertex] (lc) at (-0.3,0) {};
      \node [vertex] (rc) at (0.3,0) {};

      \draw [edge] (lc) -- (rc);
      \draw [edge] (ru) -- (rc);
      \draw [edge] (rl) -- (rc);
      \draw [edge] (lu) -- (lc);
      \draw [edge] (ll) -- (lc);
      \draw [edge] (ll) -- (rl);

      \draw [mmedge] (ru) to [bend right] (lc);
      \draw [mmedge] (lc) to [bend right] (rl);
      \draw [mmedge] (rl) to [bend right] (ru);

      \draw [mmedge] (lu) to [bend left] (rc);
      \draw [mmedge] (rc) to [bend left] (ll);
      \draw [mmedge] (ll) to [bend left] (lu);
     \end{scope}
     \begin{scope}[shift={(0, -0.375)}]
      \node at (-0.75,-0.75) {$\specialH{3,3}{c}$};

      \node [vertex] (cb) at (0,-0.75) {};
      \node [vertex] (cm) at (0,0) {};
      \node [vertex] (rm) at (0.75,0) {};
      \node [vertex] (lm) at (-0.75,0) {};
      \node [vertex] (ct) at (0,0.75) {};
      \node [vertex] (cs) at (0,1.5) {};

      \draw [edge] (cb) -- (rm);
      \draw [edge] (cb) -- (lm);
      \draw [edge] (cm) -- (rm);
      \draw [edge] (cm) -- (lm);
      \draw [edge] (ct) -- (rm);
      \draw [edge] (ct) -- (lm);
      \draw [edge] (ct) -- (cs);

      \draw [mmedge] (cs) to [bend right] (lm);
      \draw [mmedge] (lm) to [bend right] (rm);
      \draw [mmedge] (rm) to [bend right] (cs);

      \draw [mmedge] (ct) to [bend right] (cb);
      \draw [mmedge] (cb) -- (cm) -- (ct);
     \end{scope}
     \end{scope}
     \begin{scope}[shift={(1.5, 0)}]
     \begin{scope}[shift={(0, 0)}]
      \node at (0,-1) {$\specialH{6}{a}$};

      \node [vertex] (ai) at (90:0.43) {};
      \node [vertex] (bi) at (210:0.43) {};
      \node [vertex] (ci) at (330:0.43) {};
      \node [vertex] (ao) at (90:1) {};
      \node [vertex] (bo) at (210:1) {};
      \node [vertex] (co) at (330:1) {};

      \draw [edge] (ai) -- (bi) -- (ci) -- (ai);
      \draw [edge] (ai) -- (ao);
      \draw [edge] (bi) -- (bo);
      \draw [edge] (ci) -- (co);

      \draw [mmedge] (ao) to [bend right] (bi);
      \draw [mmedge] (bi) to [bend right] (co);
      \draw [mmedge] (co) to [bend right] (ai);
      \draw [mmedge] (ai) to [bend right] (bo);
      \draw [mmedge] (bo) to [bend right] (ci);
      \draw [mmedge] (ci) to [bend right] (ao);
     \end{scope}
     \begin{scope}[shift={(3, 0)}]
      \node at (0,-1) {$\specialH{6}{b}$};

      \node [vertex] (ai) at (90:0.43) {};
      \node [vertex] (bi) at (210:0.43) {};
      \node [vertex] (ci) at (330:0.43) {};
      \node [vertex] (ao) at (90:1) {};
      \node [vertex] (bo) at (210:1) {};
      \node [vertex] (co) at (330:1) {};

      \draw [edge] (ai) -- (bi) -- (ci) -- (ai);
      \draw [edge] (bo) -- (co);
      \draw [edge] (ai) -- (ao);
      \draw [edge] (bi) -- (bo);
      \draw [edge] (ci) -- (co);

      \draw [mmedge] (ao) to [bend right] (bi);
      \draw [mmedge] (bi) to [bend right] (co);
      \draw [mmedge] (co) to [bend right] (ai);
      \draw [mmedge] (ai) to [bend right] (bo);
      \draw [mmedge] (bo) to [bend right] (ci);
      \draw [mmedge] (ci) to [bend right] (ao);
     \end{scope}
     \begin{scope}[shift={(9, 0)}]
      \fill [gray!25] (0,0) circle (1.4);
      \node [vertex] (ai) at (90:0.43) {};
      \node [vertex] (bi) at (210:0.43) {};
      \node [vertex] (ci) at (330:0.43) {};
      \node [vertex] (ao) at (90:1) {};
      \node [vertex] (bo) at (210:1) {};
      \node [vertex] (co) at (330:1) {};

      \draw [edge] (ai) -- (bi) -- (ci) -- (ai);
      \draw [edge] (ao) -- (bo) -- (co) -- (ao);
      \draw [edge] (ai) -- (ao);
      \draw [edge] (bi) -- (bo);
      \draw [edge] (ci) -- (co);

      \draw [mmedge] (ao) to [bend right] (bi);
      \draw [mmedge] (bi) to [bend right] (co);
      \draw [mmedge] (co) to [bend right] (ai);
      \draw [mmedge] (ai) to [bend right] (bo);
      \draw [mmedge] (bo) to [bend right] (ci);
      \draw [mmedge] (ci) to [bend right] (ao);
    \end{scope}
    \fill [gray!25] (9, 1.5) circle (0.35);
    \node at (9, 1.5) {$=$};
    \begin{scope}[shift={(9, 3)}]
      \node at (1,-1) {$\complement{C_6}$};

      \draw [mmedge] (0,0) circle (1);

      \node [vertex] (a) at (90:1) {};
      \node [vertex] (c) at (150:1) {};
      \node [vertex] (f) at (210:1) {};
      \node [vertex] (d) at (270:1) {};
      \node [vertex] (b) at (330:1) {};
      \node [vertex] (e) at (390:1) {};

      \draw [edge] (a) -- (b) -- (c) -- (d) -- (e) -- (f) -- (a);
      \draw [edge] (c) -- (e);
      \draw [edge] (b) -- (f);
      \draw [edge] (a) -- (d);

     \end{scope}
     \begin{scope}[shift={(6, 0)}]
      \node at (0,-1) {$\specialH{6}{c}$};

      \node [vertex] (ai) at (90:0.43) {};
      \node [vertex] (bi) at (210:0.43) {};
      \node [vertex] (ci) at (330:0.43) {};
      \node [vertex] (ao) at (90:1) {};
      \node [vertex] (bo) at (210:1) {};
      \node [vertex] (co) at (330:1) {};

      \draw [edge] (ai) -- (bi) -- (ci) -- (ai);
      \draw [edge] (ao) -- (bo) -- (co) -- (ao);
      \draw [edge] (bi) -- (bo);
      \draw [edge] (ci) -- (co);

      \draw [mmedge] (ao) to [bend right] (bi);
      \draw [mmedge] (bi) to [bend right] (co);
      \draw [mmedge] (co) to [bend right] (ai);
      \draw [mmedge] (ai) to [bend right] (bo);
      \draw [mmedge] (bo) to [bend right] (ci);
      \draw [mmedge] (ci) to [bend right] (ao);
    \end{scope}
    \end{scope}      
    \begin{scope}[shift={(0, -3)}]
    \begin{scope}[shift={(3, 0)}]
      \node at (-1,-1) {$\specialH{3,3}{b}$};

      \node [vertex] (a) at (90:1) {};
      \node [vertex] (b) at (150:1) {};
      \node [vertex] (c) at (210:1) {};
      \node [vertex] (d) at (270:1) {};
      \node [vertex] (e) at (330:1) {};
      \node [vertex] (f) at (390:1) {};

      \draw [edge] (a) -- (b) -- (c) -- (d) -- (e) -- (f) -- (a);
      \draw [edge] (a) -- (d);

      \draw [mmedge] (a) -- (c) -- (e) -- (a);
      \draw [mmedge] (b) -- (d) -- (f) -- (b);
    \end{scope}
    \begin{scope}[shift={(0, 0)}]
      \node at (-1,-1) {$\specialH{3,3}{a}$};

      \node [vertex] (a) at (90:1) {};
      \node [vertex] (b) at (150:1) {};
      \node [vertex] (c) at (210:1) {};
      \node [vertex] (d) at (270:1) {};
      \node [vertex] (e) at (330:1) {};
      \node [vertex] (f) at (390:1) {};

      \draw [edge] (a) -- (b) -- (c) -- (d) -- (e) -- (f) -- (a);
      \draw [edge] (c) -- (f);
      \draw [edge] (b) -- (e);

      \draw [mmedge] (a) -- (c) -- (e) -- (a);
      \draw [mmedge] (b) -- (d) -- (f) -- (b);
    \end{scope}
    \end{scope}
  \end{tikzpicture}
  \caption{All 11 graphs (\colorbox{gray!25}{+ 2 differently drawn copies})
    of order~$6$ where each vertex has exactly $2$ metamours}
  \label{fig:graphs-6}
\end{figure}

\begin{figure}
  \centering
  \begin{tikzpicture}
    \useasboundingbox (-4, 1) rectangle (7, -8.5); 

    \begin{scope}[shift={(-1.5, 0)}]
    \begin{scope}[shift={(0, 0)}]
      \node at (0,0) {$C_7$};

      \node [vertex] (a) at (90:1) {};
      \node [vertex] (b) at (141.428571428571:1) {};
      \node [vertex] (c) at (192.857142857143:1) {};
      \node [vertex] (d) at (244.285714285714:1) {};
      \node [vertex] (e) at (295.714285714286:1) {};
      \node [vertex] (f) at (347.142857142857:1) {};
      \node [vertex] (g) at (398.571428571429:1) {};

      \draw [edge] (0,0) circle (1);

      \draw [mmedge] (a) to [bend left=15] (c);
      \draw [mmedge] (c) to [bend left=15] (e);
      \draw [mmedge] (e) to [bend left=15] (g);
      \draw [mmedge] (g) to [bend left=15] (b);
      \draw [mmedge] (b) to [bend left=15] (d);
      \draw [mmedge] (d) to [bend left=15] (f);
      \draw [mmedge] (f) to [bend left=15] (a);
    \end{scope}
    \begin{scope}[shift={(3, 0)}]
      \node at (1,-1) {$\complement{C_7}$};

      \draw [mmedge] (0,0) circle (1);

      \node [vertex] (a) at (90:1) {};
      \node [vertex] (b) at (141.428571428571:1) {};
      \node [vertex] (c) at (192.857142857143:1) {};
      \node [vertex] (d) at (244.285714285714:1) {};
      \node [vertex] (e) at (295.714285714286:1) {};
      \node [vertex] (f) at (347.142857142857:1) {};
      \node [vertex] (g) at (398.571428571429:1) {};

      \draw [edge] (a) -- (c) -- (e) -- (g) -- (b) -- (d) -- (f) -- (a);
      \draw [edge] (a) -- (d) -- (g) -- (c) -- (f) -- (b) -- (e) -- (a);

    \end{scope}
    \begin{scope}[shift={(6, 0)}, scale=1.5]
      \node at (0,-1.3) {$\complement{C_4} \join \complement{C_3}$};

      \node [vertex] (tl) at (-0.25,1) {};
      \node [vertex] (tr) at (0.25,1) {};
      \node [vertex] (bl) at (-0.25,-1) {};
      \node [vertex] (br) at (0.25,-1) {};
      \node [vertex] (cl) at (-0.5,0) {};
      \node [vertex] (cm) at (0,0) {};
      \node [vertex] (cr) at (0.5,0) {};

      \draw [edge] (tl) -- (tr);
      \draw [edge] (bl) -- (br);

      \draw [edge] (tl) -- (cl) -- (bl);
      \draw [edge] (tl) -- (cm) -- (bl);
      \draw [edge] (tl) -- (cr) -- (bl);
      \draw [edge] (tr) -- (cl) -- (br);
      \draw [edge] (tr) -- (cm) -- (br);
      \draw [edge] (tr) -- (cr) -- (br);

      \draw [mmedge] (cl) -- (cm) -- (cr);
      \draw [mmedge] (cr) to [bend right] (cl);

      \draw [mmedge] (tl) to [bend right=60] (bl);
      \draw [mmedge] (tr) to [bend left=60] (br);
      \draw [mmedge] (tl) to [bend right=15] (br);
      \draw [mmedge] (tr) to [bend left=15] (bl);
    \end{scope}  
    \end{scope}  
    \begin{scope}[shift={(-3, -4)}]
    \begin{scope}[shift={(3, 0)}]
      \node at (-1,-1.75) {$\specialH{4,3}{c}$};

      \node [vertex] (a1) at (0,2) {};
      \node [vertex] (b1) at (0,1) {};
      \node [vertex] (c1) at (0.5,0) {};
      \node [vertex] (c2) at (-0.5,0) {};
      \node [vertex] (d1) at (0.5,-1) {};
      \node [vertex] (d2) at (-0.5,-1) {};
      \node [vertex] (e1) at (0,-2) {};

      \draw [edge] (a1) -- (b1);
      \draw [edge] (b1) -- (c1) -- (d1) -- (e1);
      \draw [edge] (b1) -- (c2) -- (d2) -- (e1);
      \draw [edge] (c1) -- (d2);
      \draw [edge] (c2) -- (d1);
      \draw [edge] (c1) -- (c2);

      \draw [mmedge] (a1) to [bend right] (c2);
      \draw [mmedge] (c2) to [bend right=60] (e1);
      \draw [mmedge] (a1) to [bend left] (c1);
      \draw [mmedge] (c1) to [bend left=60] (e1);

      \draw [mmedge] (d1) to [bend right=15] (d2);
      \draw [mmedge] (b1) to [bend left=15] (d2);
      \draw [mmedge] (b1) to [bend right=15] (d1);
    \end{scope}
    \begin{scope}[shift={(9, 0)}]
      \node at (-1,-1.75) {$\specialH{4,3}{a}$};

      \node [vertex] (a1) at (0,1.5) {};
      \node [vertex] (b1) at (0.5,0.5) {};
      \node [vertex] (b2) at (-0.5,0.5) {};
      \node [vertex] (c1) at (0.5,-0.5) {};
      \node [vertex] (c2) at (-0.5,-0.5) {};
      \node [vertex] (d1) at (0.5,-1.5) {};
      \node [vertex] (d2) at (-0.5,-1.5) {};

      \draw [edge] (a1) -- (b1) -- (c1) -- (d1);
      \draw [edge] (a1) -- (b2) -- (c2) -- (d2);
      \draw [edge] (b1) -- (b2);
      \draw [edge] (d1) -- (d2);
      \draw [edge] (b1) -- (c2) -- (d1);
      \draw [edge] (b2) -- (c1) -- (d2);

      \draw [mmedge] (b1) -- (d2);
      \draw [mmedge] (b2) -- (d1);
      \draw [mmedge] (b1) to [bend left] (d1);
      \draw [mmedge] (b2) to [bend right] (d2);

      \draw [mmedge] (c1) to [bend left=15] (c2);
      \draw [mmedge] (a1) to [bend right=60] (c2);
      \draw [mmedge] (a1) to [bend left=60] (c1);
    \end{scope}
    \begin{scope}[shift={(0, 0)}]
      \node at (-1,-1.75) {$\specialH{4,3}{d}$};

      \node [vertex] (a1) at (0,1.5) {};
      \node [vertex] (b1) at (0.5,0.5) {};
      \node [vertex] (b2) at (-0.5,0.5) {};
      \node [vertex] (c1) at (0.5,-0.5) {};
      \node [vertex] (c2) at (-0.5,-0.5) {};
      \node [vertex] (d1) at (0.5,-1.5) {};
      \node [vertex] (d2) at (-0.5,-1.5) {};

      \draw [edge] (a1) -- (b1) -- (c1) -- (d1);
      \draw [edge] (a1) -- (b2) -- (c2) -- (d2);
      \draw [edge] (b1) -- (b2);
      \draw [edge] (b1) -- (c2);
      \draw [edge] (b2) -- (c1);

      \draw [mmedge] (b1) -- (d2);
      \draw [mmedge] (b2) -- (d1);
      \draw [mmedge] (b1) to [bend left] (d1);
      \draw [mmedge] (b2) to [bend right] (d2);

      \draw [mmedge] (c1) to [bend left=15] (c2);
      \draw [mmedge] (a1) to [bend right=60] (c2);
      \draw [mmedge] (a1) to [bend left=60] (c1);
    \end{scope}
    \begin{scope}[shift={(6, 0)}]
      \node at (-1,-1.75) {$\specialH{4,3}{b}$};

      \node [vertex] (a1) at (0.5,1.5) {};
      \node [vertex] (a2) at (-0.5,1.5) {};
      \node [vertex] (b1) at (0,0.5) {};
      \node [vertex] (c1) at (0.5,-0.5) {};
      \node [vertex] (c2) at (-0.5,-0.5) {};
      \node [vertex] (d1) at (0.5,-1.5) {};
      \node [vertex] (d2) at (-0.5,-1.5) {};

      \draw [edge] (a1) -- (b1) -- (c1) -- (d1);
      \draw [edge] (a2) -- (b1) -- (c2) -- (d2);
      \draw [edge] (a1) -- (a2);
      \draw [edge] (c1) -- (c2);
      \draw [edge] (c1) -- (d2);
      \draw [edge] (c2) -- (d1);

      \draw [mmedge] (a1) to [bend right] (c2);
      \draw [mmedge] (a2) to [bend left] (c1);
      \draw [mmedge] (a1) to [bend left] (c1);
      \draw [mmedge] (a2) to [bend right] (c2);

      \draw [mmedge] (d1) to [bend right=15] (d2);
      \draw [mmedge] (b1) to [bend left=15] (d2);
      \draw [mmedge] (b1) to [bend right=15] (d1);
    \end{scope}
    \end{scope}
    \begin{scope}[shift={(-1, -7.5)}]
    \begin{scope}[shift={(0, 0)}, rotate=90]
      \node at (-1,-1.75) {$\specialH{7}{b}$};

      \node [vertex] (a1) at (-0.5,-1.5) {};
      \node [vertex] (a2) at (0.5,-1) {};
      \node [vertex] (a3) at (-0.5,-0.5) {};
      \node [vertex] (a4) at (0.5,0) {};
      \node [vertex] (a5) at (-0.5,0.5) {};
      \node [vertex] (a6) at (0.5,1) {};
      \node [vertex] (a7) at (-0.5,1.5) {};

      \draw [edge] (a1) -- (a2) -- (a3) -- (a4) -- (a5) -- (a6) -- (a7);
      \draw [edge] (a1) -- (a3) -- (a5) -- (a7);
      \draw [edge] (a2) -- (a4) -- (a6);

      \draw [mmedge] (a4) -- (a1);
      \draw [mmedge] (a1) to [bend left] (a5);
      \draw [mmedge] (a5) -- (a2);
      \draw [mmedge] (a2) to [bend right] (a6);
      \draw [mmedge] (a6) -- (a3);
      \draw [mmedge] (a3) to [bend left] (a7);
      \draw [mmedge] (a7) -- (a4);
    \end{scope}  
    \begin{scope}[shift={(5, 0)}, rotate=90]
      \node at (-1,-1.75) {$\specialH{7}{a}$};

      \node [vertex] (a1) at (-0.5,-1.5) {};
      \node [vertex] (a2) at (0.5,-1) {};
      \node [vertex] (a3) at (-0.5,-0.5) {};
      \node [vertex] (a4) at (0.5,0) {};
      \node [vertex] (a5) at (-0.5,0.5) {};
      \node [vertex] (a6) at (0.5,1) {};
      \node [vertex] (a7) at (-0.5,1.5) {};

      \draw [edge] (a1) -- (a2) -- (a3) -- (a4) -- (a5) -- (a6) -- (a7);
      \draw [edge] (a3) -- (a5);
      \draw [edge] (a2) -- (a4) -- (a6);

      \draw [mmedge] (a4) -- (a1);
      \draw [mmedge] (a1) to [bend left] (a5);
      \draw [mmedge] (a5) -- (a2);
      \draw [mmedge] (a2) to [bend right] (a6);
      \draw [mmedge] (a6) -- (a3);
      \draw [mmedge] (a3) to [bend left] (a7);
      \draw [mmedge] (a7) -- (a4);
    \end{scope}  
    \end{scope}  
  \end{tikzpicture}
  \caption{All 9 graphs of order~$7$ where each vertex has exactly $2$ metamours}
  \label{fig:graphs-7}
\end{figure}

\begin{figure}
  \centering
  \begin{tikzpicture}
    \begin{scope}[shift={(-1.5, 0)}]
    \begin{scope}[shift={(0, 0)}]
      \node at (0,0) {$C_8$};

      \node [vertex] (a) at (45:1) {};
      \node [vertex] (b) at (90:1) {};
      \node [vertex] (c) at (135:1) {};
      \node [vertex] (d) at (180:1) {};
      \node [vertex] (e) at (225:1) {};
      \node [vertex] (f) at (270:1) {};
      \node [vertex] (g) at (315:1) {};
      \node [vertex] (h) at (360:1) {};

      \draw [edge] (0,0) circle (1);

      \draw [mmedge] (a) to [bend left=30] (c);
      \draw [mmedge] (c) to [bend left=30] (e);
      \draw [mmedge] (e) to [bend left=30] (g);
      \draw [mmedge] (g) to [bend left=30] (a);
      \draw [mmedge] (b) to [bend left=30] (d);
      \draw [mmedge] (d) to [bend left=30] (f);
      \draw [mmedge] (f) to [bend left=30] (h);
      \draw [mmedge] (h) to [bend left=30] (b);
    \end{scope}
    \begin{scope}[shift={(3, 0)}]
      \node at (1.1,-1) {$\complement{C_8}$};

      \draw [mmedge] (0,0) circle (1);

      \node [vertex] (a) at (45:1) {};
      \node [vertex] (b) at (90:1) {};
      \node [vertex] (c) at (135:1) {};
      \node [vertex] (d) at (180:1) {};
      \node [vertex] (e) at (225:1) {};
      \node [vertex] (f) at (270:1) {};
      \node [vertex] (g) at (315:1) {};
      \node [vertex] (h) at (360:1) {};

      \draw [edge] (a) -- (c) -- (e) -- (g) -- (a);
      \draw [edge] (b) -- (d) -- (f) -- (h) -- (b);
      \draw [edge] (a) -- (d) -- (g) -- (b) -- (e) -- (h) -- (c) -- (f) -- (a);
      \draw [edge] (a) -- (e);
      \draw [edge] (b) -- (f);
      \draw [edge] (c) -- (g);
      \draw [edge] (d) -- (h);

    \end{scope}
    \end{scope}

    \begin{scope}[shift={(0, -4.5)}]
    \def\squared#1#2{
    \begin{scope}[shift={#1}]
      \node [vertex] (#2-a) at (0.5,0.5) {};
      \node [vertex] (#2-b) at (-0.5,0.5) {};
      \node [vertex] (#2-c) at (-0.5,-0.5) {};
      \node [vertex] (#2-d) at (0.5,-0.5) {};

      \draw [edge] (#2-a) -- (#2-b) -- (#2-c) -- (#2-d) -- (#2-a);
      \draw [edge] (#2-a) -- (#2-c);
      \draw [edge] (#2-b) -- (#2-d);
    \end{scope}}
    \begin{scope}[shift={(0, 0)}]
      \node at (-1,-1.75) {$\specialH{4,4}{c}$};

      \squared{(0,0)}{s}

      \node [vertex] (at) at (0,1.5) {};
      \node [vertex] (att) at (0,2.5) {};
      \draw [edge] (s-a) -- (at) -- (s-b);
      \draw [edge] (at) -- (att);

      \node [vertex] (bt) at (0,-1.5) {};
      \node [vertex] (btt) at (0,-2.5) {};
      \draw [edge] (s-c) -- (bt) -- (s-d);
      \draw [edge] (bt) -- (btt);

      \draw [mmedge] (att) to [bend left=15] (s-a);
      \draw [mmedge] (s-a) to [bend right=15] (bt);
      \draw [mmedge] (att) to [bend right=15] (s-b);
      \draw [mmedge] (s-b) to [bend left=15] (bt);

      \draw [mmedge] (btt) to [bend left=15] (s-c);
      \draw [mmedge] (s-c) to [bend right=15] (at);
      \draw [mmedge] (btt) to [bend right=15] (s-d);
      \draw [mmedge] (s-d) to [bend left=15] (at);
    \end{scope}
    \begin{scope}[shift={(3, 0)}]
      \node at (-1,-1.75) {$\specialH{4,4}{b}$};

      \squared{(0,-0.5)}{sb}
      \squared{(0,0.5)}{st}
      
      \node [vertex] (at) at (0,2) {};
      \draw [edge] (st-a) -- (at) -- (st-b);

      \node [vertex] (bt) at (0,-2) {};
      \draw [edge] (sb-c) -- (bt) -- (sb-d);

      \draw [mmedge] (at) to [bend left=60] (st-d);
      \draw [mmedge] (at) to [bend right=60] (st-c);
      \draw [mmedge] (bt) to [bend left=60] (sb-b);
      \draw [mmedge] (bt) to [bend right=60] (sb-a);

      \draw [mmedge] (st-a) to [bend left=60] (sb-d);
      \draw [mmedge] (st-a) -- (sb-c);
      \draw [mmedge] (st-b) to [bend right=60] (sb-c);
      \draw [mmedge] (st-b) -- (sb-d);
    \end{scope}
    \begin{scope}[shift={(6, 0)}]
      \node at (-1,-1.75) {$\specialH{4,4}{a}$};

      \squared{(0,0)}{sm}
      \squared{(0,-1)}{sb}
      \squared{(0,1)}{st}

      \draw [mmedge] (st-a) to [bend left=60] (sm-d);
      \draw [mmedge] (st-a) -- (sm-c);
      \draw [mmedge] (st-b) to [bend right=60] (sm-c);
      \draw [mmedge] (st-b) -- (sm-d);

      \draw [mmedge] (sm-a) to [bend left=60] (sb-d);
      \draw [mmedge] (sm-a) -- (sb-c);
      \draw [mmedge] (sm-b) to [bend right=60] (sb-c);
      \draw [mmedge] (sm-b) -- (sb-d);
    \end{scope}
    \end{scope}

    \begin{scope}[shift={(-1.5, 0)}]
    \begin{scope}[shift={(9, 0)}, scale=1.5]
      \node at (0,-1.3) {$\complement{C_4} \join \complement{C_4}$};

      \node [vertex] (b1) at (-0.25,-1) {};
      \node [vertex] (b2) at (0.25,-1) {};
      \node [vertex] (m1) at (-0.75,0) {};
      \node [vertex] (m2) at (-0.25,0) {};
      \node [vertex] (m3) at (0.25,0) {};
      \node [vertex] (m4) at (0.75,0) {};
      \node [vertex] (t1) at (-0.25,1) {};
      \node [vertex] (t2) at (0.25,1) {};

      \draw [edge] (b1) -- (b2);
      \draw [edge] (m1) -- (m2);
      \draw [edge] (m3) -- (m4);
      \draw [edge] (t1) -- (t2);

      \draw [edge] (b1) -- (m1) -- (t1);
      \draw [edge] (b1) -- (m2) -- (t1);
      \draw [edge] (b1) -- (m3) -- (t1);
      \draw [edge] (b1) -- (m4) -- (t1);
      \draw [edge] (b2) -- (m1) -- (t2);
      \draw [edge] (b2) -- (m2) -- (t2);
      \draw [edge] (b2) -- (m3) -- (t2);
      \draw [edge] (b2) -- (m4) -- (t2);

      \draw [mmedge] (t1) -- (b2);
      \draw [mmedge] (t2) -- (b1);
      \draw [mmedge] (t1) to [bend right=90] (b1);
      \draw [mmedge] (t2) to [bend left=90] (b2);

      \draw [mmedge] (m1) to [bend left=45] (m3);
      \draw [mmedge] (m2) to [bend left=45] (m4);
      \draw [mmedge] (m3) to [bend left=45] (m2);
      \draw [mmedge] (m4) to [bend left=45] (m1);
    \end{scope}
    \begin{scope}[shift={(6, 0)}, scale=1.5]
      \node at (0,-1.3)
      {$\complement{C_5} \join \complement{C_3}$};

      \node [vertex] (a5) at (-0.5,1) {};
      \node [vertex] (b5) at (-0.5,0.5) {};
      \node [vertex] (c5) at (-0.5,0) {};
      \node [vertex] (d5) at (-0.5,-0.5) {};
      \node [vertex] (e5) at (-0.5,-1) {};
      \draw [edge] (a5) -- (b5) -- (c5) -- (d5) -- (e5) to [bend left] (a5);

      \node [vertex] (a3) at (0.5,0.5) {};
      \node [vertex] (b3) at (0.5,0) {};
      \node [vertex] (c3) at (0.5,-0.5) {};

      \draw [edge] (a3) -- (a5);
      \draw [edge] (a3) -- (b5);
      \draw [edge] (a3) -- (c5);
      \draw [edge] (a3) -- (d5);
      \draw [edge] (a3) -- (e5);
      \draw [edge] (b3) -- (a5);
      \draw [edge] (b3) -- (b5);
      \draw [edge] (b3) -- (c5);
      \draw [edge] (b3) -- (d5);
      \draw [edge] (b3) -- (e5);
      \draw [edge] (c3) -- (a5);
      \draw [edge] (c3) -- (b5);
      \draw [edge] (c3) -- (c5);
      \draw [edge] (c3) -- (d5);
      \draw [edge] (c3) -- (e5);

      \draw [mmedge] (a3) -- (b3) -- (c3);
      \draw [mmedge] (c3) to [bend right] (a3);

      \draw [mmedge] (a5) to [bend right] (c5);
      \draw [mmedge] (b5) to [bend right] (d5);
      \draw [mmedge] (c5) to [bend right] (e5);

      \draw [mmedge] (a5) to [bend left] (d5);
      \draw [mmedge] (b5) to [bend left] (e5);
    \end{scope}
    \end{scope}  
  \end{tikzpicture}
  \caption{All 7 graphs of order~$8$ where each vertex has exactly $2$ metamours}
  \label{fig:graphs-8}
\end{figure}

\begin{figure}
  \centering
  \begin{tikzpicture}
    \begin{scope}[shift={(0, 0)}]
    \begin{scope}[shift={(0, 0)}]
      \node at (0,0) {$C_9$};

      \node [vertex] (a) at (90:1) {};
      \node [vertex] (b) at (130:1) {};
      \node [vertex] (c) at (170:1) {};
      \node [vertex] (d) at (210:1) {};
      \node [vertex] (e) at (250:1) {};
      \node [vertex] (f) at (290:1) {};
      \node [vertex] (g) at (330:1) {};
      \node [vertex] (h) at (370:1) {};
      \node [vertex] (i) at (410:1) {};

      \draw [edge] (0,0) circle (1);
      
      \draw [mmedge] (a) to [bend left=30] (c);
      \draw [mmedge] (c) to [bend left=30] (e);
      \draw [mmedge] (e) to [bend left=30] (g);
      \draw [mmedge] (g) to [bend left=30] (i);
      \draw [mmedge] (i) to [bend left=30] (b);
      \draw [mmedge] (b) to [bend left=30] (d);
      \draw [mmedge] (d) to [bend left=30] (f);
      \draw [mmedge] (f) to [bend left=30] (h);
      \draw [mmedge] (h) to [bend left=30] (a);
    \end{scope}
    \begin{scope}[shift={(3, 0)}]
      \node at (1,-1) {$\complement{C_9}$};

      \draw [mmedge] (0,0) circle (1);
      
      \node [vertex] (a) at (90:1) {};
      \node [vertex] (b) at (130:1) {};
      \node [vertex] (c) at (170:1) {};
      \node [vertex] (d) at (210:1) {};
      \node [vertex] (e) at (250:1) {};
      \node [vertex] (f) at (290:1) {};
      \node [vertex] (g) at (330:1) {};
      \node [vertex] (h) at (370:1) {};
      \node [vertex] (i) at (410:1) {};

      \draw [edge] (a) -- (c) -- (e) -- (g) -- (i) -- (b) -- (d) -- (f) -- (h) -- (a);
      \draw [edge] (a) -- (d) -- (g) -- (a);
      \draw [edge] (b) -- (e) -- (h) -- (b);
      \draw [edge] (c) -- (f) -- (i) -- (c);
      \draw [edge] (a) -- (e) -- (i) -- (d) -- (h) -- (c) -- (g) -- (b) -- (f) -- (a);

    \end{scope}
    \begin{scope}[shift={(6, 0)}]
      \node at (1.25,-1.75) {$\complement{C_6} \join \complement{C_3}$};

      \draw [mmedge] (0,0) circle (0.6);
      \draw [mmedge] (0,0) circle (1.5);

      \node [vertex] (a) at (120:0.6) {};
      \node [vertex] (c) at (180:0.6) {};
      \node [vertex] (f) at (240:0.6) {};
      \node [vertex] (d) at (300:0.6) {};
      \node [vertex] (b) at (360:0.6) {};
      \node [vertex] (e) at (420:0.6) {};

      \draw [edge] (a) -- (b) -- (c) -- (d) -- (e) -- (f) -- (a);
      \draw [edge] (c) -- (e);
      \draw [edge] (b) -- (f);
      \draw [edge] (a) -- (d);


      \node [vertex] (A) at (30:1.5) {};
      \node [vertex] (B) at (150:1.5) {};
      \node [vertex] (C) at (270:1.5) {};

      \draw [edge] (C) to [bend right=15] (d);
      \draw [edge] (C) to [bend left=15] (f);
      \draw [edge] (C) to [bend right] (b);
      \draw [edge] (C) to [bend left] (c);
      \draw [edge] (C) to [bend right=75] (e);
      \draw [edge] (C) to [bend left=75] (a);

      \draw [edge] (B) to [bend right=15] (c);
      \draw [edge] (B) to [bend left=15] (a);
      \draw [edge] (B) to [bend right] (f);
      \draw [edge] (B) to [bend left] (e);
      \draw [edge] (B) to [bend right=75] (d);
      \draw [edge] (B) to [bend left=75] (b);

      \draw [edge] (A) to [bend right=15] (e);
      \draw [edge] (A) to [bend left=15] (b);
      \draw [edge] (A) to [bend right] (a);
      \draw [edge] (A) to [bend left] (d);
      \draw [edge] (A) to [bend right=75] (c);
      \draw [edge] (A) to [bend left=75] (f);
    \end{scope}
    \end{scope}

    \begin{scope}[shift={(0.5, -3.5)}]
    \begin{scope}[shift={(0, 0)}]
      \node at (-1.5,-1) {$\complement{C_5} \join \complement{C_4}$};

      \node [vertex] (a5) at (0,1) {};
      \node [vertex] (b5) at (0,0.5) {};
      \node [vertex] (c5) at (0,0) {};
      \node [vertex] (d5) at (0,-0.5) {};
      \node [vertex] (e5) at (0,-1) {};
      \draw [edge] (a5) -- (b5) -- (c5) -- (d5) -- (e5) to [bend left=45] (a5);

      \node [vertex] (l1) at (-1.5,0.5) {};
      \node [vertex] (l2) at (-1.5,-0.5) {};
      \draw [edge] (l1) -- (l2);
      \node [vertex] (r1) at (1.5,0.5) {};
      \node [vertex] (r2) at (1.5,-0.5) {};
      \draw [edge] (r1) -- (r2);

      \draw [edge] (l1) -- (a5);
      \draw [edge] (l1) -- (b5);
      \draw [edge] (l1) -- (c5);
      \draw [edge] (l1) -- (d5);
      \draw [edge] (l1) -- (e5);
      \draw [edge] (r1) -- (a5);
      \draw [edge] (r1) -- (b5);
      \draw [edge] (r1) -- (c5);
      \draw [edge] (r1) -- (d5);
      \draw [edge] (r1) -- (e5);
      \draw [edge] (l2) -- (a5);
      \draw [edge] (l2) -- (b5);
      \draw [edge] (l2) -- (c5);
      \draw [edge] (l2) -- (d5);
      \draw [edge] (l2) -- (e5);
      \draw [edge] (r2) -- (a5);
      \draw [edge] (r2) -- (b5);
      \draw [edge] (r2) -- (c5);
      \draw [edge] (r2) -- (d5);
      \draw [edge] (r2) -- (e5);

      \draw [mmedge] (l1) to [bend left=75] (r1);
      \draw [mmedge] (l2) to [bend right=75] (r2);
      \draw [mmedge] (l1) to [out=45, in=180] (0,1.25) to [out=0, in=105] (r2);
      \draw [mmedge] (l2) to [out=-45, in=-180] (0,-1.25) to [out=-0, in=-105] (r1);

      \draw [mmedge] (a5) to [bend right=45] (c5);
      \draw [mmedge] (b5) to [bend right=45] (d5);
      \draw [mmedge] (c5) to [bend right=45] (e5);

      \draw [mmedge] (a5) to [bend left=45] (d5);
      \draw [mmedge] (b5) to [bend left=45] (e5);
    \end{scope}  
    \begin{scope}[shift={(5, 0)}]
      \node at (0,-1.5)
      {$K_{3,3,3} = K_{3,3} \join \complement{C_3} = \complement{C_3} \join \complement{C_3} \join \complement{C_3}$};

      \node [vertex] (a1) at (70:1) {};
      \node [vertex] (a2) at (90:1) {};
      \node [vertex] (a3) at (110:1) {};
      \node [vertex] (b1) at (190:1) {};
      \node [vertex] (b2) at (210:1) {};
      \node [vertex] (b3) at (230:1) {};
      \node [vertex] (c1) at (310:1) {};
      \node [vertex] (c2) at (330:1) {};
      \node [vertex] (c3) at (350:1) {};

      \draw [edge] (a1) -- (b1) -- (c1) -- (a1);
      \draw [edge] (a2) -- (b2) -- (c2) -- (a2);
      \draw [edge] (a3) -- (b3) -- (c3) -- (a3);

      \draw [edge] (a1) -- (b2) -- (c3) -- (a1);
      \draw [edge] (a2) -- (b3) -- (c1) -- (a2);
      \draw [edge] (a3) -- (b1) -- (c2) -- (a3);

      \draw [edge] (a1) -- (b3) -- (c2) -- (a1);
      \draw [edge] (a2) -- (b1) -- (c3) -- (a2);
      \draw [edge] (a3) -- (b2) -- (c1) -- (a3);

      \draw [mmedge] (a1) -- (a2) -- (a3);
      \draw [mmedge] (a3) to [bend left=90] (a1);
      \draw [mmedge] (b1) -- (b2) -- (b3);
      \draw [mmedge] (b3) to [bend left=90] (b1);
      \draw [mmedge] (c1) -- (c2) -- (c3);
      \draw [mmedge] (c3) to [bend left=90] (c1);
    \end{scope}  
    \end{scope}
  \end{tikzpicture}
  \caption{All 5 graphs of order~$9$ where each vertex has exactly $2$ metamours}
  \label{fig:graphs-9}
\end{figure}
